\documentclass[reqno]{amsart}

\usepackage{tikz-cd}

\usepackage{fullpage,dsfont}

\usepackage{hyperref} 

\usepackage{parskip}
\makeatletter 
\def\thm@space@setup{%
 \thm@preskip=\parskip \thm@postskip=0pt
}
\def\th@remark{%
  \thm@headfont{\itshape}%
  \normalfont 
  \thm@preskip\parskip \thm@postskip=0pt
}
\makeatother

\usepackage[nobysame,alphabetic,initials,msc-links]{amsrefs}

\usepackage[final]{pdfpages}

\usepackage{multirow}

\usepackage{array}
\usepackage{kbordermatrix}

\DefineSimpleKey{bib}{how}
\DefineSimpleKey{bib}{mrclass}
\DefineSimpleKey{bib}{mrnumber}
\DefineSimpleKey{bib}{fjournal}
\DefineSimpleKey{bib}{mrreviewer}

\renewcommand{\PrintDOI}[1]{%
  \href{http://dx.doi.org/#1}{{\tt DOI:#1}}%
}
\renewcommand{\eprint}[1]{#1}
\BibSpec{book}{%
    +{}  {\PrintPrimary}                {transition}
    +{.} { \PrintDate}                  {date}
    +{.} { \textit}                     {title}
    +{.} { }                            {part}
    +{:} { \textit}                     {subtitle}
    +{,} { \PrintEdition}               {edition}
    +{}  { \PrintEditorsB}              {editor}
    +{,} { \PrintTranslatorsC}          {translator}
    +{,} { \PrintContributions}         {contribution}
    +{,} { }                            {series}
    +{,} { \voltext}                    {volume}
    +{,} { }                            {publisher}
    +{,} { }                            {organization}
    +{,} { }                            {address}
    +{,} { }                            {status}
    +{,} { \PrintDOI}                   {doi}
    +{,} { \PrintISBNs}                 {isbn}
    +{}  { \parenthesize}               {language}
    +{}  { \PrintTranslation}           {translation}
    +{;} { \PrintReprint}               {reprint}
    +{.} { }                            {note}
    +{.} {}                             {transition}
    +{}  {\SentenceSpace \PrintReviews} {review}
}
\BibSpec{article}{%
    +{}  {\PrintAuthors}                {author}
    +{,} { \textit}                     {title}
    +{.} { }                            {part}
    +{:} { \textit}                     {subtitle}
    +{,} { \PrintContributions}         {contribution}
    +{.} { \PrintPartials}              {partial}
    +{,} { }                            {journal}
    +{}  { \textbf}                     {volume}
    +{}  { \PrintDatePV}                {date}
    +{,} { \issuetext}                  {number}
    +{,} { \eprintpages}                {pages}
    +{,} { }                            {status}
    +{,} { \PrintDOI}                   {doi}
    +{,} { \eprint}        {eprint}
    +{}  { \parenthesize}               {language}
    +{}  { \PrintTranslation}           {translation}
    +{;} { \PrintReprint}               {reprint}
    +{.} { }                            {note}
    +{.} {}                             {transition}
    +{}  {\SentenceSpace \PrintReviews} {review}
}
\BibSpec{collection.article}{%
    +{}  {\PrintAuthors}                {author}
    +{,} { \textit}                     {title}
    +{.} { }                            {part}
    +{:} { \textit}                     {subtitle}
    +{,} { \PrintContributions}         {contribution}
    +{,} { \PrintConference}            {conference}
    +{}  {\PrintBook}                   {book}
    +{,} { }                            {booktitle}
    +{,} { \PrintDateB}                 {date}
    +{,} { pp.~}                        {pages}
    +{,} { }                            {publisher}
    +{,} { }                            {organization}
    +{,} { }                            {address}
    +{,} { }                            {status}
    +{,} { \PrintDOI}                   {doi}
    +{,} { \eprint}        {eprint}
    +{}  { \parenthesize}               {language}
    +{}  { \PrintTranslation}           {translation}
    +{;} { \PrintReprint}               {reprint}
    +{.} { }                            {note}
    +{.} {}                             {transition}
    +{}  {\SentenceSpace \PrintReviews} {review}
}
\BibSpec{misc}{%
  +{}{\PrintAuthors}  {author}
  +{,}{ \textit}      {title}
  +{.}{ }             {how}
  +{}{ \parenthesize} {date}
  +{,} { available at \eprint}        {eprint}
  +{,}{ available at \url}{url}
  +{,}{ }             {note}
  +{.}{}              {transition}
}
\usepackage{amssymb, amsfonts, amsxtra, amsmath}
\usepackage{mathrsfs}
\usepackage{mathdots}
\usepackage{wasysym}
\usepackage[all]{xy}
\usepackage{bbm}
\usepackage{calc}
\usepackage{accents}

\usepackage{stackengine} 
\newcommand\oast{\stackMath\mathbin{\stackinset{c}{0ex}{c}{0ex}{\ast}{\bigcirc}}}

\numberwithin{equation}{section}

\newtheorem{TheoIntro}{Theorem}

\newtheorem{Theorem}{Theorem}[section]
\newtheorem*{Theorem*}{Theorem}
\newtheorem{Def}[Theorem]{Definition}
\newtheorem*{Def*}{Definition}
\newtheorem{Lem}[Theorem]{Lemma}
\newtheorem{Prop}[Theorem]{Proposition}
\newtheorem{Cor}[Theorem]{Corollary}

\newcommand\bp{\begin{proof}}
\newcommand\ep{\end{proof}}

\mathchardef\mhyph="2D

\DeclareMathOperator{\diag}{\mathrm{diag}}

\DeclareMathOperator{\Ad}{\mathrm{Ad}}

\DeclareMathOperator{\End}{\mathrm{End}}

\DeclareMathOperator{\loc}{\mathrm{loc}}
\DeclareMathOperator{\ext}{\mathrm{ext}}

\DeclareMathOperator{\id}{\mathrm{id}}

\DeclareMathOperator{\rd}{\mathrm{d}\!}
\DeclareMathOperator{\Tr}{\mathrm{Tr}}

\DeclareMathOperator{\Ker}{\mathrm{Ker}}
\DeclareMathOperator{\Spec}{\mathrm{Spec}}

\DeclareMathOperator{\Det}{\mathrm{Det}}
\DeclareMathOperator{\Sym}{\mathrm{Sym}}

\DeclareMathOperator{\standard}{\mathrm{st}}
\DeclareMathOperator{\pos}{\mathrm{w.int}}
\DeclareMathOperator{\post}{\mathrm{pos}}
\DeclareMathOperator{\posint}{\mathrm{pos.int}} 

\DeclareMathOperator{\Sign}{\mathrm{Sign}}
\DeclareMathOperator{\dom}{\mathrm{dom}}

\DeclareMathOperator{\inte}{\mathrm{int}}

\newcommand{\cop}{\mathrm{cop}}

\newcommand{\wt}{\mathrm{wt}}

\newcommand{\msA}{\mathscr{A}}

\newcommand{\msE}{\mathscr{E}}

\newcommand{\msH}{\mathscr{H}}

\newcommand{\msN}{\mathscr{N}}
\newcommand{\msP}{\mathscr{P}}
\newcommand{\msQ}{\mathscr{Q}}
\newcommand{\msR}{\mathscr{R}}

\newcommand{\msU}{\mathscr{U}}

\newcommand{\msZ}{\mathscr{Z}}

\newcommand{\mfb}{\mathfrak{b}}

\newcommand{\mfh}{\mathfrak{h}}

\newcommand{\mfn}{\mathfrak{n}}

\newcommand{\mfsu}{\mathfrak{su}}

\newcommand{\mfu}{\mathfrak{u}}

\newcommand{\mfgl}{\mathfrak{gl}}

\newcommand{\mcH}{\mathcal{H}}
\newcommand{\Hsp}{\mathcal{H}}
\newcommand{\Gsp}{\mathcal{G}}

\newcommand{\mcO}{\mathcal{O}}

\newcommand{\mcS}{\mathcal{S}}

\newcommand{\mbr}{\mathbf{r}}
\newcommand{\mbAd}{\mathbf{A}\mathbf{d}}

\newcommand{\mbs}{\mathbf{s}}

\newcommand{\mbGL}{\mathbf{G}\mathbf{L}}

\newcommand{\mbH}{\mathbf{H}}

\newcommand{\mbT}{\mathbf{T}}
\newcommand{\mbU}{\mathbf{U}}

\newcommand{\mbalpha}{\mathbf{a}}

\newcommand{\brone}{[1]}

\newcommand{\brr}{[r]}
\newcommand{\brk}{[k]}
\newcommand{\brl}{[l]}

\newcommand{\brM}{[M]}
\newcommand{\brN}{[N]}

\newcommand{\C}{\mathbb{C}}
\newcommand{\G}{\mathbb{G}}

\newcommand{\N}{\mathbb{N}}

\newcommand{\R}{\mathbb{R}}

\newcommand{\Z}{\mathbb{Z}}

\newcommand{\GL}{\mathrm{GL}}

\newcommand{\opp}{\mathrm{op}}

\newcommand{\res}{\mathrm{res}}

\newcommand{\sgn}{\mathrm{sgn}}

\newcommand\lhdb{\blacktriangleleft}
\newcommand{\HC}{\mathrm{HC}}

\newcommand{\br}{\mathrm{br}}

\newcommand{\Lin}{\mathrm{Lin}}

\title{Representation theory of the reflection equation algebra I: A quantization of Sylvester's law of inertia}

\author{Kenny De Commer}
\address{Vrije Universiteit Brussel}
\email{kenny.de.commer@vub.be}

\author{Stephen T. Moore}
\address{Institute of Mathematics, Polish Academy of Sciences}
\email{stm862@gmail.com}

\thanks{The work of K.DC.~ and S.T.M.~was supported by the FWO grants G025115N and G032919N. S.T.M. was additionally supported by Narodowe Centrum Nauki, grant number 2017/26/A/ST1/00189.}

\begin{document}
\maketitle

\begin{abstract}
We  prove a version of Sylvester's law of inertia for the Reflection Equation Algebra (=REA). We will only be concerned with the REA constructed from the $R$-matrix associated to the standard $q$-deformation of $GL(N,\C)$. For $q$ positive, this particular REA comes equipped with a natural $*$-structure, by which it can be viewed as a $q$-deformation of the $*$-algebra of polynomial functions on the space of self-adjoint $N$-by-$N$-matrices. We will show that this REA satisfies a type $I$-condition, so that its irreducible representations can in principle be classified. Moreover, we will show that, up to the adjoint action of quantum $GL(N,\C)$, any irreducible representation of the REA is determined by its \emph{extended signature}, which is a classical signature vector extended by a parameter in $\R/\Z$. It is this latter result that we see as a quantized version of Sylvester's law of inertia. 
\end{abstract}

\section*{Introduction and statement of the main results}

Sylvester's \emph{law of inertia} states that the $GL(N,\C)$-adjoint orbit $\{x^*zx \mid x\in GL(N,\C)\}$ of a self-adjoint $N\times N$-matrix $z$ only remembers its \emph{signature} 
\[
\varsigma(z) = (N_+,N_-,N_0),
\]
with $N_0$ the number of zero eigenvalues and $N_{\pm}$ the number of positive/negative eigenvalues of $z$. With 
\[
\Sign = \{(N_+,N_-,N_0) \in \Z_{\geq 0}^3\mid N_++N_- +N_0= N\},
\]
we have more precisely that, with $H(N)$ the space of self-adjoint $N$-by-$N$-matrices, we obtain a bijection 
\begin{equation}\label{EqLawInert}
H(N)/GL(N,\C) \cong \Sign,\qquad \{x^*zx\mid x\in GL(N,\C)\} \mapsto \varsigma(z).
\end{equation}

In this paper, we will look for a quantum analogue of \eqref{EqLawInert}. 

Let  $0<q<1$, and let $\hat{R} \in M_N(\C) \otimes M_N(\C)$ be the \emph{braid operator}
\begin{equation}\label{EqBraidOp}
\hat{R} =  \sum_{ij} q^{-\delta_{ij}}e_{ji}\otimes e_{ij} + (q^{-1} -q)\sum_{i<j} e_{jj} \otimes e_{ii},
\end{equation}
where $e_{ij}$ is the standard matrix unit taking the basis vector $e_j$ to $e_i$. It satisfies the \emph{braid relation}
\[
\hat{R}_{12}\hat{R}_{23}\hat{R}_{12} = \hat{R}_{23} \hat{R}_{12}\hat{R}_{23},
\] 
where we use the usual leg numbering notation $\hat{R}_{23} = 1\otimes \hat{R}$ etc. It also satisfies the \emph{Hecke relation}
\[
(\hat{R}-q^{-1})(\hat{R}+q)  = \hat{R}^2 +(q-q^{-1})\hat{R} -1 = 0
\]
and is self-adjoint, $\hat{R}^* = \hat{R}$. The matrix $\hat{R}$ is in particular invertible, with 
\[
\hat{R}^{-1} = \sum_{ij} q^{\delta_{ij}}e_{ij}\otimes e_{ji} + (q -q^{-1})\sum_{i<j} e_{ii} \otimes e_{jj}.
\]

\begin{Def*}
The \emph{Reflection Equation Algebra} (\emph{REA}) is the universal unital complex algebra $\mcO_{q}^{\br}(M_N(\C))$ generated by the matrix entries of 
\[
Z = \sum_{ij} e_{ij} \otimes Z_{ij} \in M_N(\C)\otimes \mcO_q^{\br}(M_N(\C)) = M_N(\mcO_q^{\br}(M_N(\C))),
\] 
with universal relations given by the \emph{reflection equation}
\begin{equation}\label{EqUnivRelZ}
\hat{R}_{12}Z_{23} \hat{R}_{12}Z_{23} = Z_{23} \hat{R}_{12}Z_{23}\hat{R}_{12}. 
\end{equation}

We define the $*$-REA $\mcO_{q}(H(N))$ as $\mcO_{q}^{\br}(M_N(\C))$ endowed with the $*$-structure $Z^* = Z$. 
\end{Def*}

The reflection equation (with parameters) arose in the work of Cherednik \cite{Ch84} on quantum scattering on the half line, and the associated algebra (together with its variants) was studied from the algebraic perspective in, for example, \cite{Maj91,KS92,KS93,KSS93,Maj94,PS95,DM02a,Mud02,DKM03,Mud06,KS09,JW20}. One can view $\mcO_q(H(N))$ as a quantization of the $*$-algebra of complex-valued polynomial functions on $H(N)$. 

We will be interested in studying the representations of $\mcO_q(H(N))$. By \emph{representation} of a unital $*$-algebra $A$ on a Hilbert space $\Hsp$ we mean a unital $*$-homomorphism of $A$ into the $*$-algebra $B(\Hsp)$ of bounded operators on $\Hsp$,
\[
\pi: A\rightarrow B(\Hsp).
\] 
It is called \emph{irreducible} if $\Hsp$ is the only non-zero $\pi(A)$-invariant closed subspace of $\Hsp$, and a \emph{factor representation} if the commutant $\pi(A)'$ is a factor (i.e.\ has trivial center). It is called \emph{type $I$} if the bicommutant $\pi(A)''$ is a type $I$ von Neumann algebra. We call $A$ type $I$ if all its representations are type $I$, or equivalently, if for any irreducible representation $\pi$ the norm-closure of $\pi(A)$ contains the compact operators \cite{Sak67}. If $A$ is type $I$, any factor representation is irreducible up to multiplicity. We call $A$ \emph{C$^*$-faithful} if its representations separate elements. 


The following result is proven as Theorem \ref{TheoTypeI} in the paper.

\begin{TheoIntro}\label{TheoIntroTypeI}
The unital $*$-algebra $\mcO_q(H(N))$ is C$^*$-faithful and type $I$. 
\end{TheoIntro} 
In fact, C$^*$-faithfulness follows from well-known results, but for the type $I$-condition this is less clear.

An important observation concerning $\mcO_q(H(N))$ is that its generating matrix $Z$ satisfies a \emph{quantum Cayley-Hamilton equation} \cite{PS95,JW20}: there exist canonical central self-adjoint $\sigma_k \in \mcO_q(H(N))$ with
\[
\sum_{k=0}^N (-1)^k \sigma_k Z ^{N-k} = 0,\qquad \textrm{where }\sigma_0 = 1.
\]
In particular, $\sigma^{\pi}_k:= \pi(\sigma_k)$ is a real scalar for each factor representation $\pi$ of $\mcO_q(H(N))$. We then call 
\begin{equation}\label{EqCharPolpi}
P_{\pi}(x) = \sum_{k=0}^N (-1)^k \sigma_k^{\pi} x^{N-k}
\end{equation}
the \emph{characteristic polynomial} of $\pi$. Our second result can be stated as follows, see Theorem \ref{TheoEigZRef}. 

\begin{TheoIntro}\label{TheoEigZ}
Let $P$ be a monic real polynomial of degree $N$. Then $P =P_{\pi}$ for a factor representation $\pi$ of $\mcO_q(H(N))$ if and only if there exist $\alpha,\beta \in \R$ such that the multiset of roots of $P$ is of the form 
\begin{equation}\label{EqFormRootsPs}
\{q^{2\alpha+2m_1},\ldots,q^{2\alpha + 2m_{N_+}},-q^{2\beta + 2n_1},\ldots,-q^{2\beta + 2n_{N_-}},\underbrace{0,\ldots, 0}_{N_0}\}
\end{equation}
for certain $N_+,N_-,N_0 \in \Z_{\geq 0}$ and $m_i,n_i \in \Z$ with $m_i \neq m_j$ and $n_i \neq n_j$ for $i\neq j$.  
\end{TheoIntro} 

Write $[r] = r+ \Z \in \R/\Z$. For $\pi$ a factor representation with multiset of roots of $P_{\pi}$ given by \eqref{EqFormRootsPs}, we write 
\[
\varsigma(\pi) = (N_+,N_-,N_0), \qquad \varsigma_{\ext}(\pi) = ([\beta-\alpha],N_+,N_-,N_0,).
\] 
We call $\varsigma(\pi)$ the \emph{signature} of $\pi$ and $\varsigma_{\ext}(\pi)$ the \emph{extended signature} of $\pi$. Here we put, by convention, $[\beta - \alpha] = 0$ if $N_+N_-=0$. In general, we call a quadruple $\varsigma_{\ext} = ([r],N_+,N_-,N_0)$ an extended signature (of rank $N$) if $[r]\in \R/\Z$ and $N_+,N_-,N_0 \in \Z_{\geq 0}$ with  $N_++N_-+N_0  = N$ and $[r]=0$ if $N_+N_- = 0$. We denote 
\[
\Sign^{\ext} = \{\textrm{extended signatures}\}.
\]

Note that we have the purely quantum phenomenon of an extra parameter $[r]$ in the extended signature. This extra parameter is connected in a natural way to the extra parameter present in the quantization of the Hermitian symmetric space $U(N)/U(m)\times U(n)$ \cite{Let99,Kol14,DCNTY19,DCNTY20}. For earlier particular instances of this, see \cite{Pod87} for the case of quantum spheres ($m=n=1$) and \cite{DN98} for the case of quantum projective spaces ($m=1$ and $n$ arbitrary).

To state our next main result, consider the W$^*$-category\footnote{We only use this as a convenient terminology: the reader may simply consider these as representation categories of unital $*$-algebras on Hilbert spaces, with bounded intertwiners as morphisms.} \cite{GLR85}
\begin{equation}\label{EqHqN}
\mbH_q(N) = \{\textrm{representations of }\mcO_q(H(N)) \textrm{ on Hilbert spaces} \}.
\end{equation} 
Recall that for a given unital $*$-algebra $A$, a representation $\pi': A \rightarrow B(\Gsp)$ is \emph{weakly contained} in a representation $\pi: A \rightarrow B(\Hsp)$ if $\pi'$ factors through a $*$-homomorphism $C^*(\pi) \rightarrow B(\Gsp)$, with
\begin{equation}\label{EqNormClosPiA}
C^*(\pi) = \textrm{ norm-closure of }\pi(A).
\end{equation}
Equivalently, we have $\|\pi'(a)\|\leq \|\pi(a)\|$ for all $a\in A$. In this case, we write
\begin{equation}\label{EqWeakCont}
\pi' \preccurlyeq \pi.
\end{equation}

We can then consider for each extended signature $\varsigma_{\ext}$ the full W$^*$-subcategory
\[
\mbH_q^{\varsigma_{\ext}}(N) \subseteq \mbH_q(N)
\]
consisting of representations $\pi$ such that all factor representations weakly contained in $\pi$ have extended signature $\varsigma_{\ext}$. It is immediate from Theorem \ref{TheoEigZ} that any representation in $\mbH_q(N)$ has a disintegration into components of the various $\mbH_q^{\varsigma_{\ext}}(N)$, and that each of the $\mbH_q^{\varsigma_{\ext}}(N)$ is non-empty. By Theorem \ref{TheoIntroTypeI}, we may also replace `factor representations' by `irreducible representations' in the above. 

To link the above with a quantum version of Sylvester's law of inertia, note first that we may identify 
\begin{equation}\label{EqHqNAlt}
\mbH_q(N) \cong \{(\Hsp,Z)\mid \Hsp \textrm{ Hilbert space and }Z\in M_N(\C)\otimes B(\Hsp) \textrm{ self-adjoint, satisfying }\eqref{EqUnivRelZ}\}.
\end{equation}
We can then also consider $\mbGL_q(N,\C)$, the W$^*$-category of couples $(\Gsp,X)$ with $\Gsp$ a Hilbert space and $X \in M_N(\C)\otimes B(\Gsp)$ invertible and satisfying 
\begin{equation}\label{EqOtherEq}
\hat{R}_{12}X_{13}X_{23} = X_{13}X_{23}\hat{R}_{12},\qquad X_{23}\hat{R}_{12} X_{23}^* = X_{13}^* \hat{R}_{12}X_{13}.
\end{equation}
It is easily seen that $\mbGL_q(N,\C)$ has a tensor product by putting 
\[
(\Gsp,X)*(\Gsp',X') = (\Gsp \otimes \Gsp',X_{12}X_{13}'),
\]
and that $\mbGL_q(N,\C)$ acts on $\mbH_q(N)$ by the right adjoint action,
\begin{equation}\label{EqDefAdGLq}
\mbAd: \mbH_q(N) \times \mbGL_q(N,\C) \rightarrow \mbH_q(N),\qquad \mbAd_{(\Gsp,X)}(\Hsp,Z) =  (\Hsp \otimes \Gsp,X_{13}^*Z_{12}X_{13}),
\end{equation}
with the usual tensor product map $(f,g) \mapsto f\otimes g$ on morphisms. This makes $\mbH_q(N)$ into a right $\mbGL_q(N,\C)$-module W$^*$-category. We refrain from spelling out the details of the latter assertion, as we will only need the operation \eqref{EqDefAdGLq} itself, but let us mention in this context Woronowicz's approach for working on this categorical level to get a grip on operator algebraic quantizations \cite{Wor83,Wor95}. 

\begin{Def*}
We say that two factor representations $\pi,\pi'$ in $\mbH_q(N)$ lie in the same $\mbGL_q(N,\C)$-orbit if there exist $\lambda,\lambda'\in \mbGL_q(N,\C)$ such that 
\[
\pi \preccurlyeq \mbAd_{\lambda}(\pi'),\qquad \pi' \preccurlyeq \mbAd_{\lambda'}(\pi). 
\]
\end{Def*}
It is easily seen that this is indeed an equivalence relation. Denote by 
\[
\mbH_q(N)/\mbGL_q(N,\C)
\]
the quotient of the class of factor representations by this equivalence relation. 
We can now state our next main result as follows. It will follow immediately from our Theorem \ref{TheoQSLI}. 


\begin{TheoIntro}[Quantized law of inertia]\label{TheoQLI}
The map 
\begin{equation}
\mbH_q(N)/\mbGL_q(N,\C) \rightarrow \Sign^{\ext},\qquad [\pi] \cong \varsigma_{\ext}(\pi)
\end{equation}
is a well-defined bijection. 
\end{TheoIntro}

The precise contents of this paper are as follows.  In Section \ref{section: quantum group intros}, we introduce the quantizations of the groups $U(N), T(N)$ and $GL(N,\C)$, and recall also the connection with the quantized universal enveloping algebra $U_q(\mfgl(N,\C))$ and its real form $U_q(\mfu(N))$. In Section \ref{section: REA intro}, we give some more information on the commutation relations which hold within $\mcO_q(H(N))$, and consider as well the associated actions of quantized $GL(N,\C),U(N)$ and $T(N)$ in a purely algebraic framework. We present the classification of $*$-characters of $\mcO_q(H(N))$, taken from \cite{Mud02}, and of general irreducible $*$-representations of $\mcO_q(H(2))$. In Section \ref{section: spectral weight}, we show that factor representations of $\mcO_q(H(N))$ allow a \emph{spectral weight}, and prove a quantized version of the spectral theorem. In Section \ref{section: big cell representations}, we look at a particular class of representations called \emph{big cell representations}. We obtain in this case a complete classification of all \emph{irreducible} big cell representations. In the final Section \ref{section: sylvester's law} we then obtain the main results of the paper. 

In follow-up papers \cite{DCMo24} and \cite{Mo24}, the above results are refined to obtain a complete classification of all irreducible $*$-representations of the reflection equation algebra. 

We end this introduction with some common notations that we will make use of.

Throughout the paper, we fix $0<q<1$.

For a given set $T$, we denote by $\mathscr{P}(T)$ the powerset of $T$. When $k \in \Z_{\geq 0}$, we denote $\brk$ for the set $\{1,2,\ldots k\}$. We denote $\binom{\brN}{k} \subseteq \mathscr{P}(\brN)$ for the set of subsets $I \subseteq \brN$ with $|I| = k$. We view $I \in \binom{\brN}{k}$ as a totally ordered set, and write
\[
I = \{i_1<\ldots < i_k\}
\]
to indicate this. We view $\binom{\brN}{k}$ itself as a partially ordered set by 
\begin{equation}\label{EqOrdPart}
I \preceq J \qquad \textrm{if and only if}\qquad i_p \leq j_p\textrm{ for all }1\leq p \leq N.
\end{equation}
Note that we have
\[
I \preceq J \qquad \textrm{if and only if}\qquad |J \cap \brl|\leq |I\cap \brl| \textrm{ for all }1\leq l\leq N.
\]

We further use the following notation:  For $I\in \binom{\brN}{k}$ and $K\in \binom{\brk}{l}$ we write
\begin{equation}\label{EqSetUnderUpper}
I_K = \{i_{p}\mid p \in K\},\qquad I^K = I \setminus I_K.
\end{equation}
For $I \in \binom{\brN}{k}$, we define the \emph{weight} of $I$ to be the sum of its elements, 
\begin{equation}\label{EqWeightVecto}
\wt(I) = \sum_{r=1}^k i_r.
\end{equation}
For $X \in M_N(\C)$ and $I,J \in \binom{[N]}{k}$, we denote 
\[
X[I,J] \in M_k(\C)
\]
for the submatrix with $X[I,J]_{r,s} = X_{i_r,j_s}$. The associated minor is written as 
\[
X_{I,J} = \Det(X[I,J]). 
\]

It will also be convenient to use the following shorthand for a given vector $\lambda \in \C^N$:
\begin{equation}\label{EqShHandMult}
\lambda_I = \lambda_{i_1}\dots \lambda_{i_k},\quad I = \{i_1,\ldots,i_k\},\qquad \lambda_{\omega} = \lambda_1^{\omega_1}\dots \lambda_{N}^{\omega_N},\quad \omega \in \Z_{\geq0}^{N}.   
\end{equation}

Finally, for any bijection $\sigma: A \rightarrow B$ between finite totally ordered sets, we denote
\begin{equation}\label{EqLengthSym}
l(\sigma) = |\{(i,j) \in A\times B \mid i<j,\sigma(i)>\sigma(j)\}|
\end{equation}
for the number of inversions in $\sigma$.

\section{The W$^*$-categories $\mbU_q(N),\mbT_q(N)$, $\mbGL_q(N,\C)$}\label{section: quantum group intros}


In this section, we recall results and set conventions for the quantizations of the groups $U(N),T(N)$ and $GL(N,\C)$ which arise as quantum symmetries for the REA  $\mcO_q(H(N))$.


\subsection{The Hopf algebra $U_q(\mfgl(N,\C))$ and its real form $U_q(\mfu(N))$}

For the following, we refer to any standard reference work on quantum groups, e.g.\ \cite{KS97}. We adapt results to fit our conventions. 



\begin{Def}\label{DefUqglN}
We define $U_q(\mfn(N))$ as the universal unital $\C$-algebra generated by elements $E_1,\ldots, E_{N-1}$ such that $E_iE_j = E_jE_i$ for all $1\leq i,j< N$ with $|i-j|\geq 2$, and such that for all $1\leq i,j< N$ with $|i-j|= 1$ it holds that
\begin{equation}\label{EqCommSerreE}
E_i^2E_j - (q+q^{-1})E_iE_jE_i + E_jE_i^2 =0,
\end{equation}
We define $U_q(\mfn^-(N))$ as an independent copy of $U_q(\mfn(N))$, with generators written as $F_i$ for $1\leq i < N$.

We define $U_q(\mfh(N))$ as the $\C$-algebra $\C[K_1^{\pm 1},\ldots, K_N^{\pm 1}]$, and write $\hat{K}_i = K_iK_{i+1}^{-1}$ for $1\leq i < N$. 

We define $U_q(\mfgl(N,\C))$ as the algebra generated by $U_q(\mfn(N)),U_q(\mfh(N)),U_q(\mfn^-(N))$ with extra relations 
\begin{equation}\label{EqWeightComm}
K_iE_j =q^{\delta_{ij}-\delta_{i,j+1}}E_j K_i, \qquad K_iF_j =q^{\delta_{i,j+1}-\delta_{ij}}F_j K_i,\qquad 1 \leq i\leq N,1\leq j< N,
\end{equation}
\begin{equation}\label{EqFundCommEF}
E_i F_j - F_jE_i= \delta_{ij} \frac{\hat{K}_i - \hat{K}_i^{-1}}{q-q^{-1}},\qquad 1\leq i,j<N. 
\end{equation}
\end{Def}
Note that $U_q(\mfh(N))$ is just the group algebra of $\Z^N$, and in particular does not depend on $q$. The notation $U_q(\mfh(N))$ is purely meant to guide classical intuition as to the role that this algebra plays in the theory. 

By e.g.\ the proof of \cite[Theorem 6.14]{KS97}, the following multiplication maps are bijective:
\[
U_q(\mfn(N)) \otimes U_q(\mfh(N)) \otimes U_q(\mfn^-(N)) \rightarrow U_q(\mfgl(N,\C)),
\]
\[
U_q(\mfn^-(N)) \otimes U_q(\mfh(N)) \otimes U_q(\mfn(N)) \rightarrow U_q(\mfgl(N,\C)).
\]

\begin{Def}
We endow $U_q(\mfgl(N,\C))$ with the Hopf algebra structure 
\begin{equation}\label{EqComultUq}
\Delta(K_i) = K_i \otimes K_i,\qquad \Delta(E_i) = E_i \otimes 1 +   \hat{K}_i\otimes E_i,\qquad \Delta(F_i) = F_i \otimes  \hat{K}_i^{-1} + 1 \otimes F_i,
\end{equation}
where counit and antipode are determined by 
\begin{equation}\label{EqComultUqVarepsS}
\varepsilon(E_i) = \varepsilon(F_i)= 0,\quad \varepsilon(K_i) = 1,\qquad S(E_i) = -\hat{K}_i^{-1}E_i,\quad S(F_i) = -F_i\hat{K}_i,\quad S(K_i) = K_i^{-1}.
\end{equation}
We let $U_q(\mfb(N))$ and $U_q(\mfb^-(N))$ be the Hopf subalgebras generated by the $K_{i}^{\pm 1}$ and $E_i$, respectively the $K_i^{\pm 1}$ and $F_i$. 
\end{Def}
The $U_q(\mfb(N))$ and $U_q(\mfb^-(N))$ are also determined by the corresponding universal relations. 

\begin{Def}
The Hopf $*$-algebra $U_q(\mfu(N))$ has underlying Hopf algebra $U_q(\mfgl(N,\C))$ and $*$-structure
\begin{equation}\label{EqDefStar}
K_i^* = K_i,\qquad E_i^* = q^{-1}F_i\hat{K}_i,\qquad F_i^* = q\hat{K}_i^{-1}E_i.
\end{equation}
\end{Def}

A \emph{weight vector} in a $U_q(\mfgl(N,\C))$-module is a joint eigenvector for all $K_i$. A weight vector $v$ is said to be 
\begin{itemize}
\item a \emph{highest weight vector} if $E_i v = 0$ for all $i$,
\item an \emph{admissible weight vector} if the associated eigenvalues of the $K_i$ are positive. The unique $\lambda\in \R^N$ such that $K_iv = q^{\lambda_i}v$ is then called the \emph{weight} $\wt(v)$ of $v$,
\item an \emph{integral weight vector} if $v$ is admissible and $\wt(v)\in \Z^N$, and
\item a \emph{positive integral weight vector} if $v$ is integral and $\wt(v) \in \Z_{\geq0}^{N}$.
\end{itemize} 
A $U_q(\mfgl(N,\C))$-module is \emph{admissible}, resp.\ \emph{integral}, resp.\ \emph{positive integral} if it has a basis of admissible, resp.\ integral, resp.\ positive integral weight vectors. 

Denote the set of \emph{weakly integral}, \emph{integral} and \emph{positive integral} weights respectively as
\[
P_{\pos} = \{\lambda = (\lambda_1,\ldots,\lambda_N) \in \R^N \mid \lambda_i - \lambda_j \in \Z \textrm{ for all }i,j\},\qquad P_{\inte} = \Z^N,\qquad P_{\posint} = \Z_{\geq0}^{N}.
\]
Call a vector $\lambda \in \R^N$ \emph{dominant} if $\lambda_i \geq \lambda_{j}$ whenever $i\leq j$. The dominant vectors in $P_{\bullet}$ are denoted as $P_{\bullet}^{\dom}$. We also denote the standard bilinear form on $\R^N$ as 
\[
(\omega,\chi) = \sum_{i=1}^N \omega_i \chi_i. 
\]

We write $Q \subseteq P_{\inte}$ for the \emph{root lattice} consisting of integral weights $\lambda$ with $\sum \lambda_i = 0$, and $Q^+ \subseteq Q$ for the $\Z_{\geq 0}$-span of the  $\mbalpha_i = e_i -e_{i+1}$ for $1\leq i<N$. For $\omega,\omega'\in P_{\inte}$ we write 
\[
\omega \preceq \omega' \qquad \Leftrightarrow \qquad \omega' - \omega \in Q^+ \qquad  \Leftrightarrow \qquad  \sum_{k=1}^N \omega_k = \sum_{k=1}^N \omega_k'\textrm{ and } \sum_{k=1}^l \omega_k \leq \sum_{k=1}^l \omega_k' \textrm{ for all }1\leq l \leq N.
\]
Note that the order \eqref{EqOrdPart} can be related to the partial order on $P_{\inte}$ as follows: with
\[
\varpi: \mathscr{P}(\brN) \rightarrow P_{\inte},\qquad I\mapsto \varpi_I = \sum_{i\in I}  e_{i}, 
\]
we have 
\[
I \preceq J \qquad \textrm{if and only if} \qquad \varpi_J \preceq \varpi_I.
\]
In the following, we will simply identify a $\{0,1\}$-valued weight $\lambda$ with its associated set,
\begin{equation}\label{EqIdentSetWeight}
\lambda = \varpi_I = I = \{i \mid \lambda_i= 1\}.
\end{equation}

The following is well-known, see e.g.\ \cite[Theorem 7.23]{KS97} for the case of integral weights. 

\begin{Theorem}
\begin{itemize}
\item 
All admissible, resp.\ (positive) integral finite-dimensional $U_q(\mfgl(N,\C))$-modules decompose as a direct sum of irreducible admissible, resp.\ (positive) integral modules.
\item 
An admissible module $V$ is irreducible if and only if it has a cyclic highest weight vector, and the associated highest weight $\lambda$ completely determines $V$ (up to isomorphism). We write $V_{\lambda}$ for a fixed model of such an irreducible admissible module with highest weight $\lambda$.  
\item 
An element $\lambda \in \R^N$ arises as a highest weight of an irreducible admissible, resp.\ (positive) integral $U_q(\mfgl(N,\C))$-module if and only if $\lambda \in P_{\pos}^{\dom}$, resp.\ $\lambda \in P_{(pos.)\inte}^{\dom}$. 
\item The weights $\omega$ of an irreducible admissible $V_{\lambda}$ satisfy $\omega \preceq \lambda$. 
\item A finite-dimensional admissible $U_q(\mfgl(N,\C))$-module $V$ has a Hilbert space structure for which it is a representation of $U_q(\mfu(N))$. Its scalar product is unique up to a positive scalar when $V$ is irreducible. 
\end{itemize}
\end{Theorem}

We briefly comment below on why the $V_{\lambda}$ with $\lambda \in P_{\posint}^{\dom}$ are positive integral modules. 

Using the notation \eqref{EqIdentSetWeight}, consider the irreducible positive integral modules $V_{[k]}$ for $0\leq k \leq N$. Here $V_{[0]}$ is the trivial module, given by the counit, while we can take $V_{[1]}= \C^N$ the \emph{vector module} 
\begin{equation}\label{EqVectRep}
\theta(K_i)e_j  = q^{\delta_{ij}}e_j,\qquad \theta(E_i) e_j = \delta_{i,j-1}e_{j-1},\qquad \theta(F_i) e_j = \delta_{ij} e_{j+1}.
\end{equation}
Note that this determines a representation of $U_q(\mfu(N))$, with the standard basis of $\C^N$ as orthonormal basis. 

To construct the other modules $V_{[k]}$, we follow \cite[Chapter 8]{PW91}. Consider the quadratic algebra $\wedge_q(\C^N)$ generated by vectors $\{e_i\mid1\leq i\leq N\}$ with defining  relations given by $e_i \wedge_q e_i = 0$ and 
\[
e_j \wedge_q e_i = -q e_i \wedge_q e_j,\qquad  1\leq i<j \leq N.
\]
In other words, we are dividing out the tensor algebra of $\C^N$ by the range of 
\[
\hat{R}+q: \C^N \otimes \C^N \rightarrow \C^N \otimes \C^N.
\]
Noting that $\hat{R}$ is an intertwiner for $\theta^{*2} := (\theta\otimes \theta)\circ \Delta$, we get that $\wedge_q(\C^N)$ is a $U_q(\mfgl(N,\C))$-module algebra:
\begin{equation}\label{EqModAlgCliff}
\rhd: U_q(\mfgl(N,\C))\otimes \Lambda_q(\C^N) \rightarrow \Lambda_q(\C^N),\qquad x\rhd (ab) = (x_{(1)}\rhd a)(x_{(2)}\rhd b), \quad x \rhd 1 = \varepsilon(x)1,
\end{equation}
where we use the (sumless) Sweedler notation 
\[
\Delta(x) = x_{(1)}\otimes x_{(2)}. 
\]
Denoting 
\[
\wedge_q^k(\C^N) = \textrm{lin. span}\left\{e_I: = e_{i_1}\wedge_q \ldots \wedge_q e_{i_k} \mid I\in \binom{\brN}{k}\right\} \subseteq \wedge_q(\C^N),
\]
we obtain that $\wedge_q^k(\C^N)$ is a model for $V_{[k]}$. We can also represent $\wedge_q^k(\C^N)$ directly as 
\[
\wedge_q^k(\C^N) = (\C^N)^{\otimes k}/  \sum_{i=1}^{k-1} \mathrm{Im}(\hat{R}_{i,i+1}+q),
\]
and the $e_I$ will give a basis for $\wedge_q^k(\C^N)$. With the $e_I$ orthonormal, we obtain a representation of $U_q(\mfu(N))$. 

If we now consider a general $V_{\lambda}$ with $\lambda \in P_{\posint}^{\dom}$, we see that $V_{\lambda}$ is embedded in some $V_{[1]}^{k_1}\otimes \ldots \otimes V_{[N]}^{k_N}$, sending the highest weight vector of $V_{\lambda}$ to the tensor product of the highest weight vectors in the range. Since the above tensor product is positive integral, the same is true of $V_{\lambda}$. 


\subsection{The quantum matrix algebra $\mcO_q(M_N(\C))$ and its localisation $\mcO_q(GL(N,\C))$}

\begin{Def}[\cite{FRT88}]
We define the FRT (= Faddeev-Reshetikhin-Takhtajan) algebra $\mcO_q(M_N(\C))$ as the universal unital $\C$-algebra generated by the matrix entries of $X = \sum_{ij} e_{ij} \otimes X_{ij}$ with universal relations
\begin{equation}\label{EqCommRX}
 \hat{R}_{12} X_{13}X_{23} = X_{13}X_{23} \hat{R}_{12}. 
\end{equation}
It is a bialgebra by the coproduct and counit
\[
(\id\otimes \Delta)X = X_{12}X_{13},\qquad (\id\otimes \varepsilon)(X) = I_N.
\]
\end{Def}
Writing out  \eqref{EqCommRX}, we find the following commutation relations between the $X_{ij}$,
\begin{equation}\label{EqCommHol}
\left\{\begin{array}{lll} 
X_{ij} X_{kj} = q X_{kj}X_{ij} & \textrm{for} & i<k,\\
X_{ij} X_{ik} = qX_{ik}X_{ij} &\textrm{for}&j<k,\\
X_{ij} X_{kl} = X_{kl} X_{ij}&\textrm{for}& i< k \textrm{ and } j>l,\\
X_{ij}X_{kl} =X_{kl}X_{ij} + (q-q^{-1})X_{il}X_{kj} & \textrm{for}&i<k\textrm{ and } j<l.
\end{array}\right.
\end{equation}

As $\widehat{R}$ is an intertwiner for the tensor product $V_{[1]}^{\otimes2}$ of the vector module of $U_q(\mfgl(N,\C))$ with itself, it follows that there is a pairing of bialgebras between $\mcO_q(M_N(\C))$ and $U_q(\mfgl(N,\C))$, given by 
\begin{equation}\label{EqDualOqUq}
\tau(X_{ij},x) = \theta(x)_{ij},\qquad x \in U_q(\mfgl(N,\C)).
\end{equation}
Stating that this is a pairing simply means that $\tau$ is a bilinear map 
\[
\tau: \mcO_q(M_N(\C))\times U_q(\mfgl(N,\C)) \rightarrow \C
\]
with
\[
\tau(a,xy) = \tau(a_{(1)},x)\tau(a_{(2)},y),\quad \tau(a,1) = \varepsilon(a),\quad \tau(ab,x)= \tau(a,x_{(1)})\tau(b,x_{(2)}),\quad \tau(1,x) = \varepsilon(x). 
\]
This duality is non-degenerate \cite[Chapter 11, Corollary 22]{KS97},\footnote{The statement loc. cit. is for $SL(N,\C)$, but can easily be shown also for $M_N(\C)$, or for $GL(N,\C)$ as introduced later.} so $\tau$ gives injective homomorphisms into the respective vector space duals,
\[
U_q(\mfgl(N,\C))\rightarrow \Lin(\mcO_q(M_N(\C)),\C),\qquad \mcO_q(M_N(\C)) \rightarrow \Lin(U_q(\mfgl(N,\C)),\C),
\]
where the right hand sides are endowed with the convolution algebra structure.

By the universal property, we also see that the comodule structure
\[
\delta: \C^N \rightarrow \C^N\otimes \mcO_q(M_N(\C)),\qquad e_j \mapsto \sum_i e_i \otimes X_{ij}
\]
lifts to a comodule algebra map 
\begin{equation}\label{EqComodAlgMapqExt}
\wedge \delta: \wedge_q(\C^N) \rightarrow \wedge_q(\C^N)\otimes \mcO_q(M_N(\C)),
\end{equation}
which in turn restricts to comodule structures 
\[
\wedge^k\delta: \wedge_q^k(\C^N)  \rightarrow \wedge_q^k(\C^N) \otimes \mcO_q(M_N(\C)). 
\]
It follows that the $U_q(\mfgl(N,\C))$-modules $V_{[k]}$ can be `integrated' to $\mcO_q(M_N(\C))$-comodules, i.e.\ 
\[
xv = (\id\otimes \tau(-,x))\wedge^k\!\delta (v),\qquad x\in U_q(\mfgl(N,\C)),v\in V_{[k]} = \wedge_q^k(\C^N).
\]

\begin{Def}
For $I,J\in \binom{\brN}{k}$ we define the \emph{quantum minors}
\[
X_{IJ} \in \mcO_q(M_N(\C))
\]
as the matrix coefficients of $\wedge^k\delta$ with respect to the canonical basis, so 
\[
\wedge^k\delta(e_J) = \sum_J e_I \otimes X_{IJ}.  
\]
We then put 
\[
X^{[k]} = \sum_{IJ} e_{IJ} \otimes X_{IJ} \in \End(\Lambda_q^k(\C^N))\otimes \mcO_q(M_N(\C)), 
\]
and we call $X^{[k]}$ the \emph{$k$-th exterior power} of $X$.
\end{Def}
For brevity, we write 
\[
X_I = X_{I,I}.
\]
In particular, we denote
\begin{equation}\label{EqQuantDet}
\Det_q(X) := X_{[N]} = X_{[N],[N]} = \sum_{\sigma \in S_N} (-q)^{l(\sigma)} X_{1,\sigma(1)}X_{2,\sigma(2)}\cdots X_{N,\sigma(N)}.
\end{equation}
It can be shown that $\Det_q(X)$ is central \cite[Theorem 4.6.1]{PW91}. More generally, we have
\[
X_{[k]} = X_{[k],[k]} = \sum_{\sigma \in S_k} (-q)^{l(\sigma)} X_{1,\sigma(1)}X_{2,\sigma(2)}\cdots X_{k,\sigma(k)}.
\]

\begin{Def}
We define $\mcO_q(GL(N,\C))$ as the localisation 
\[
\mcO_q(GL(N,\C)) = \mcO_q(M_N(\C))[\Det_q(X)^{-1}]. 
\]
\end{Def}
As $\mcO_q(M_N(\C))$ has no zero divisors \cite[Theorem 3.5.1]{PW91}, it embeds in its localisation. Moreover, $X$ becomes invertible in $\mcO_q(GL(N,\C))$. This implies that $\mcO_q(GL(N,\C))$ becomes a Hopf algebra.

The pairing \eqref{EqDualOqUq} extends uniquely to a non-degenerate pairing of Hopf algebras 
\begin{equation}\label{EqDualOqUqGL}
\tau: \mcO_q(GL(N,\C)) \otimes U_q(\mfgl(N,\C)) \rightarrow \C. 
\end{equation}
It follows that $\mcO_q(GL(N,\C))$ is co-semisimple, and that an $U_q(\mfgl(N,\C))$-module lifts to a $\mcO_q(GL(N,\C))$-comodule through the pairing \eqref{EqDualOqUqGL} if and only if it is integral. In particular, we can identify the coalgebras 
\[
\mcO_q(GL(N,\C)) \cong \oplus_{\lambda \in P_{\inte}^{\dom}} \Lin(\End(V_{\lambda}),\C),
\]
and hence have a concrete model for the convolution algebra of $\mcO_q(GL(N,\C))$,
\begin{equation}\label{EqDualOqGL}
\msU_q(\mfgl(N,\C)) := \Lin(\mcO_q(GL(N,\C)),\C) \cong \prod_{\lambda \in P_{\inte}^{\dom}} \End(V_{\lambda}).
\end{equation}
Note in particular the following corollary.
\begin{Cor}
A finite-dimensional integral $U_q(\mfgl(N,\C))$-module extends uniquely to a $\msU_q(\mfgl(N,\C))$-module. 
\end{Cor}
Similar statements hold for tensor products. For example, any tensor product of finite-dimensional integral $U_q(\mfgl(N,\C))$-modules uniquely extends to a module of
\[
\msU_q(\mfgl(N,\C)) \widehat{\otimes}\msU_q(\mfgl(N,\C)) := \Lin(\mcO_q(GL(N,\C))\otimes \mcO_q(GL(N,\C)),\C),
\]
and we have an identification 
\begin{equation}\label{EqDualOqGL2}
\msU_q(\mfgl(N,\C)) \widehat{\otimes}\msU_q(\mfgl(N,\C)) \cong \prod_{\lambda,\lambda' \in P_{\inte}^{\dom}} \End(V_{\lambda})\otimes \End(V_{\lambda'}). 
\end{equation}

Going back to $\mcO_q(M_N(\C))$, also the following is well-known. Let us give a brief proof.  
\begin{Lem}\label{LemPosIntWellDef}
A finite-dimensional $U_q(\mfgl(N,\C))$-module lifts to a $\mcO_q(M_N(\C))$-comodule if and only if it is positive integral. 
\end{Lem} 
\begin{proof}
As the $X_{ij}$ generate $\mcO_q(M_N(\C))$, any comodule with coefficients in $\mcO_q(M_N(\C))$ must be a direct summand of $\theta^{* m}$ for some $m$. As these are positive integral modules, we see that the condition in the lemma is necessary. The condition is also sufficient since any comodule with dominant weight $\lambda \in \Z_{\geq 0}^{N}$ is clearly a direct summand of a tensor product of the $V_{[k]}$. 
\end{proof}
We can hence denote the dual convolution algebra of $\mcO_q(M_N(\C))$ as 
\begin{equation}\label{EqDualOqMN}
\msU_q^{\post}(\mfgl(N,\C)) := \Lin(\mcO_q(M_N(\C)),\C) \cong \prod_{\lambda \in P_{\posint}^{\dom}} \End(V_{\lambda}).
\end{equation}
Since the positive integral representations of $U_q(\mfgl(N,\C))$ are stable under tensor products, $\msU_q^{\post}(\mfgl(N,\C))$ is still endowed with a coproduct
\[
\Delta^{\post}:  \msU_q^{\post}(\mfgl(N,\C)) \rightarrow \msU_q^{\post}(\mfgl(N,\C)) \widehat{\otimes} \msU_q^{\post}(\mfgl(N,\C)).
\]

\subsection{The W$^*$-category $\mbU_q(N)$}

The Hopf algebra $\mcO_q(GL(N,\C))$ gives a quantization of $GL(N,\C)$ as a complex group variety. To obtain its real form $U(N)$, we endow $\mcO_q(GL(N,\C))$ with the following $*$-structure. 

\begin{Def}
The Hopf $*$-algebra $\mcO_q(U(N)) = (\mcO_q(GL(N,\C)),*)$ is determined uniquely by
\[
X^* = X^{-1}. 
\]
\end{Def}


One can note that this is well-defined as the composition $S(-)^*$ is well-defined, given that $S(X_{ij})^* = X_{ji}$. This $*$-structure is compatible with the real form $U_q(\mfu(N))$ under the pairing $\tau$ of \eqref{EqDualOqUqGL}, in the sense that 
\begin{equation}\label{EqPairingStar}
\tau(f,x^*) = \overline{\tau(S(f)^*,x)},\qquad \tau(f^*,x) = \overline{\tau(f,S(x)^*)},\qquad f\in \mcO_q(U(N)),x \in U_q(\mfu(N)).
\end{equation}
When considering $\mcO_q(U(N))$, we will denote the generating matrix as $U$ rather than $X$.

\begin{Def}
We let $\mbU_q(N)$ be the W$^*$-category of representations of $\mcO_q(U(N))$. 
\end{Def}
Equivalently, we may see this as the full W$^*$-subcategory of $\mbGL_q(N,\C)$ (defined above \eqref{EqOtherEq}) given by couples $(\mcH,U)$ with $U$ \emph{unitary}. This is a tensor W$^*$-subcategory, and the tensor product coincides with the one on $\mcO_q(U(N))$-representations given by the coproduct, 
\begin{equation}\label{EqNotTens}
\pi*\pi' := (\pi\otimes \pi')\Delta.
\end{equation}

It follows by the Tannaka-Krein-reconstruction theorem \cite{Wor88}, together with the fact that $\mcO_q(U(N))$ is non-degenerately paired with $U_q(\mfu(N))$, that $\mcO_q(U(N))$ is C$^*$-faithful. This implies that the Hopf $*$-algebra $\mcO_q(U(N))$ defines a \emph{compact quantum group} \cite{Wor87,LS91,DK94}: there exists on $\mcO_q(U(N))$ a (necessarily faithful, necessarily unique) invariant state
\[
\int_{U_q(N)}: \mcO_q(U(N)) \rightarrow \C,
\] 
called the \emph{Haar integral}. The above conditions mean that, for all $f\in \mcO_q(U(N))$, 
\[
\int_{U_q(N)} f^*f \geq 0,\qquad \int_{U_q(N)}1 = 1,\qquad \left(\int_{U_q(N)}\otimes \id\right)\Delta(f) = \int_{U_q(N)}f = \left(\id\otimes \int_{U_q(N)}\right)\Delta(f).
\]
We can then consider the universal C$^*$-envelope $C_q(U(N))$ of $\mcO_q(U(N))$. On the other hand, one can also consider the associated reduced C$^*$-envelope, given by completing $\mcO_q(U(N))$ in its natural representation $\lambda_{\standard}$ on the GNS-space $L^2_q(U(N))$ with respect to $\int_{U_q(N)}$.

\begin{Prop}\cite[Theorem 2.7.14]{NT13}
The compact quantum group $U_q(N)$ is \emph{coamenable}:  the GNS-representation $\lambda_{\standard}$ lifts to a \emph{faithful} representation of $C_q(U(N))$.
\end{Prop}

Let us look at the finer structure of $\mbU_q(N)$. 

First, recall that the Hopf $*$-algebra $\mcO_q(SU(2))$ is the universal unital Hopf $*$-algebra generated by $a,c$ with the requirement that $U=\begin{pmatrix} a & -qc^* \\ c& a^*\end{pmatrix}$ is a unitary corepresentation matrix.  Any irreducible representation of $\mcO_q(SU(2))$ is equivalent to one of the following list of mutually inequivalent irreducible representations: the family of $*$-characters $\chi_{\theta}$ and the family of representations $\mbs*\chi_{\theta}$, where $\chi_{\theta}$ is the $*$-character 
 \[
\chi_{\theta}(U) = \begin{pmatrix} e^{2\pi i\theta} & 0 \\0 & e^{-2\pi i\theta}\end{pmatrix},\qquad \theta \in \R/\Z,
\] 
and where $\mbs$ is the representation on $l^2(\N)$ given by 
\[
\mbs(a)e_n = (1-q^{2n})^{1/2}e_{n-1},\qquad \mbs(c)e_n = q^{n} e_{n}. 
\]
Now $\mcO_q(SU(2))$ is dual to $U_q(\mfsu(2))$, defined as the $*$-subalgebra of $U_q(\mfu(2))$ generated by $E,F,\hat{K}^{\pm 1}$. Since $U_q(\mfsu(2))$ has for each $1\leq i < N$ a copy in $U_q(\mfu(N))$, generated by $E_i, F_i$ and $\hat{K}_i$, we obtain a corresponding restriction map 
\[
\res_i: \mcO_q(U(N)) \rightarrow \mcO_q(SU(2)),\qquad 1\leq i <N.
\]
We denote $\mbs_i = \mbs \circ \res_i$. On the other hand, we also have for $\theta \in (\R/\Z)^N$ the $*$-character
\[
\chi_{\theta}: \mcO_q(U(N)) \rightarrow \C,\quad  U_{kl} \mapsto \delta_{kl} e^{2\pi i \theta_k}.
\]


\begin{Theorem}[\cite{Koe91,LS91}]
The C$^*$-algebra $C_q(U(N))$ is type $I$.  Any irreducible representation of $\mcO_q(U(N))$ is isomorphic to a representation $\mbs_w* \chi_{\theta}$, where $w \in \Sym(\brN)$ and $\mbs_w = \mbs_{i_1}*\ldots *\mbs_{i_k}$ for $w = s_{i_1}s_{i_2}\ldots s_{i_k}$ a reduced expression in terms of the elementary transpositions $s_i = (i,i+1)$. Any two of these irreducible representations are unitarily equivalent if and only if the corresponding elements in $\Sym(\brN)$ and $\R^N/\Z^N$ are equal. 
\end{Theorem}
In particular, $\mbs_w$ only depends on the reduced decomposition of $w$ up to unitary equivalence.


We also have the following theorem which is a direct consequence of \cite[Theorem 5.2]{RY01}. We call a representation of $\mcO_q(U(N))$ \emph{normal} if it lifts to a normal representation of the von Neumann algebra
\[
L^{\infty}_q(U(N)) := \lambda_{\standard}(\mcO_q(U(N)))''. 
\]

\begin{Theorem}
Let $M=\binom{N}{2}$, and let $s_{i_1}\ldots s_{i_{M}}$ be a fixed reduced expression for the longest word $w_0$ in $\Sym(\brN)$. Let $\mbs_{w_0} =  \mbs_{i_1}*\ldots *\mbs_{i_{M}}$. Then 
\[
\lambda_0 = \int_{(\R/\Z)^N} \mbs_{w_0}*\chi_{\theta}\rd \theta
\]  
is a faithful normal representation of  $\mcO_q(U(N))$. Moreover,
\[
\lambda_{\standard} \cong  \int_{(\R/\Z)^N}\mbs_{w_0}*\mbs_{w_0}*\chi_{\theta}\rd \theta  \cong \oplus_{\N} \lambda_0.
\] 
\end{Theorem}

\subsection{The W$^*$-category $\mbT_q(N)$}

\begin{Def}
We define $\mcO_q(B(N))$ as the quotient algebra of $\mcO_q(GL(N,\C))$ by putting $X_{ij} =0$ for $i>j$. We denote by 
\[
\pi_B: \mcO_q(GL(N,\C)) \rightarrow \mcO_q(B(N))
\] 
the quotient map. We write the resulting images of $X_{ij}$ as $T_{ij}$ (when not zero).
\end{Def}

In the following, we put $T_{ii} = T_i$. Then $\mcO_q(B(N))$ is the universal algebra generated by invertible elements $T_{i}$ for $1\leq i\leq N$ and elements $T_{ij}$ with $1\leq i<j\leq N$ such that the following relations hold:
\[
T_i T_{kl} = q^{\delta_{ik}-\delta_{il}}T_{kl}T_i,\qquad k\leq l, 
\]
\begin{equation}\label{EqCommHolT}
\left\{\begin{array}{lll} 
T_{ij} T_{kj} = q T_{kj}T_{ij} & \textrm{for} & i<k<j,\\
T_{ij} T_{ik} = qT_{ik}T_{ij} &\textrm{for}&i<j<k,\\
T_{ij} T_{kl} = T_{kl} T_{ij}&\textrm{for}& i< k \textrm{ and } j>l,\\
T_{ij}T_{kl} =T_{kl}T_{ij} + (q-q^{-1})T_{il}T_{kj} & \textrm{for}&i<k\textrm{ and } j<l.
\end{array}\right.
\end{equation}



The pairing \eqref{EqDualOqUqGL} restricts to a non-degenerate pairing of Hopf algebras
\begin{equation}\label{EqNonDegPair}
\tau: \mcO_q(B(N))\times U_q(\mfb(N))\rightarrow \C.
\end{equation} 
On the other hand, we have the following result which is for example contained in \cite[Theorem 8.33]{KS97}. 
\begin{Lem}
There is an isomorphism of Hopf algebras
\begin{equation}\label{EqUnivRelInv}
\kappa: U_q(\mfb^-(N))^{\cop}\rightarrow \mcO_q(B(N)),\qquad K_i \mapsto T_{ii}^{-1},\quad F_i \mapsto (q^{-1}-q)^{-1} T_{i,i+1}T_{i+1,i+1}^{-1}. 
\end{equation}
We denote the inverse as 
\begin{equation}\label{EqUnivRelAlt}
\iota: \mcO_q(B(N))\rightarrow U_q(\mfb^-(N))^{\cop},
\end{equation}
so in particular
\[
\iota(T_{i}) = K_i^{-1},\qquad \iota(T_{i,i+1}) = (q^{-1}-q)F_iK_{i+1}^{-1}.
\]
\end{Lem}
The inverse map $\iota$ has a more conceptual description as follows. With $\Sigma$ the flip map, consider the \emph{Yang-Baxter matrix}
\[
R = \Sigma \circ \hat{R} = \sum_{ij} q^{-\delta_{ij}}e_{ii}\otimes e_{jj} + (q^{-1} -q)\sum_{i<j} e_{ij} \otimes e_{ji}.
\]
From $R$ one obtains the \emph{skew bicharacter} $\mbr: \mcO_q(M_N(\C))\otimes \mcO_q(M_N(\C)) \rightarrow \C$, uniquely determined by 
\begin{equation}\label{EqSkewBichar}
(\id\otimes \id\otimes \mbr)(X_{13}X_{24}) = R. 
\end{equation}
This follows from the relations of $\mcO_q(M_n(\C))$, together with the defining properties of a skew bicharacter: 
\begin{equation}\label{EqDefProp}
\mbr(ab,c) = \mbr(a,c_{(1)})\mbr(b,c_{(2)}),\quad \mbr(a,bc) = \mbr(a_{(2)},b)\mbr(a_{(1)},c),\qquad \mbr(a,1) = \varepsilon(a),\quad \mbr(1,b) = \varepsilon(b).
\end{equation}
This skew bicharacter automatically extends to $\mcO_q(GL(N,\C))$.

Through \eqref{EqDualOqGL2}, we can view $\mbr$ as being induced by an element
\[
\msR \in \msU_q(\mfgl(N,\C)) \widehat{\otimes} \msU_q(\mfgl(N,\C)) = \prod_{\lambda,\kappa\in P_{\inte}^{\dom}} \End(V_{\lambda}) \otimes \End(V_{\kappa}),
\]
the \emph{universal $R$-matrix}, which then satisfies
\begin{equation}\label{EqPropR}
(\Delta\otimes \id)\msR = \msR_{13}\msR_{23},\quad (\id\otimes \Delta)\msR = \msR_{13}\msR_{12},
\end{equation}
\begin{equation}\label{EqPropRReversing}
 \msR \Delta(-) = \Delta^{\opp}(-)\msR.
\end{equation}
These identities can be interpreted within the dual vector space of $\otimes_{i=1}^n \mcO_q(GL(N,\C))$ for appropriate $n$. We further have that
\begin{equation}\label{EqPropR2}
\msR(\xi\otimes \eta) = q^{-(\wt(\xi),\wt(\eta))}\xi\otimes \eta
\end{equation}
for a highest weight vector $\xi$ and lowest weight vector $\eta$. Together with \eqref{EqPropRReversing} this uniquely determines $\msR$. 

We will need the following information on $\msR$. For $\alpha \in Q^+ \subseteq \R^N$, consider the finite-dimensional spaces
\begin{equation}\label{EqPosRootAdj}
U_q(\mfn(N))_{\alpha} = \{X \in U_q(\mfn(N))\mid K_i XK_i^{-1} = q^{\alpha_i}X\},
\end{equation}
\begin{equation}\label{EqNegRootAdj}
U_q(\mfn^-(N))_{-\alpha} = \{X \in U_q(\mfn^-(N))\mid K_i XK_i^{-1} = q^{-\alpha_i}X\}.
\end{equation}
Then one can factorize 
\begin{equation}\label{EqFactR}
\msR = \widetilde{\msR} \msQ,
\end{equation}
where for general weight vectors $\xi,\eta$
\begin{equation}\label{EqPropR3}
\msQ(\xi\otimes \eta) = q^{-(\wt(\xi),\wt(\eta))}\xi\otimes \eta 
\end{equation}
and
\begin{equation}\label{EqPropR5}
\widetilde{\msR} = \sum_{\alpha \in Q^+} \widetilde{\msR}_{\alpha},\qquad \widetilde{\msR}_{\alpha} \in U_q(\mfn(N))_{\alpha} \otimes U_q(\mfn^-(N))_{-\alpha}.
\end{equation}
The latter sum converges weakly as a functional on $\mcO_q(GL(N,\C))^{\otimes 2}$: for any $f,g\in \mcO_q(GL(N,\C))$, the sum 
\[
\sum_{\alpha} \tau(f\otimes g,\widetilde{\msR}_{\alpha})
\]
contains only finitely many non-zero terms. We have for example, for $\mbalpha_r = e_r - e_{r+1}$, 
\begin{equation}\label{EqPropR4}
\widetilde{\msR}_{0} = 1\otimes 1,\qquad \widetilde{\msR}_{\mbalpha_r} = (q^{-1}-q) E_r \otimes F_r.
\end{equation}

It can now be checked that the map 
\begin{equation}\label{EqiotaGL}
\mcO_q(GL(N,\C)) \rightarrow \msU_q(\mfgl(N,\C)),\qquad f\mapsto (\tau(f,-)\otimes \id)\msR
\end{equation}
factors through a map $\mcO_q(B(N)) \rightarrow U_q(\mfb^-(N))$, which is precisely the map $\iota$:
\begin{equation}\label{EqUnivRel}
\iota: \mcO_q(B(N)) \rightarrow U_q(\mfb^-(N))^{\cop},\quad f\mapsto (\tau(f,-)\otimes \id)\msR.
\end{equation}
%



The above gives a quantization of the group of invertible upper triangular matrices as a complex algebraic group. We now want to quantize the underlying real space. 

If $A$ is a complex algebra, we denote by $A^{*}$ its formal anti-linear, anti-isomorphic copy: it is the unique structure of complex algebra on the set of formal symbols $\{a^*\mid a\in A\}$ such that 
\[
A \rightarrow A^*,\qquad a \mapsto a^*
\]
is an anti-linear, anti-multiplicative map. 

\begin{Def}
The unital $*$-algebra $\mcO_q^{\R}(B(N))$ is generated by $\mcO_q(B(N))$ and $\mcO_q(B(N))^*$, with relations 
\begin{equation}\label{EqUniRelMN}
T_{23}\hat{R}_{12}T^*_{23} = T^*_{13}\hat{R}_{12}T_{13}.
\end{equation}
We define $\mcO_q(T(N))$ as the quotient of $\mcO_q^{\R}(B(N))$ by putting $T_{i}^* = T_i$ for all $i$. 
\end{Def}

We have that $\mcO_q(T(N))$ is the universal unital $*$-algebra generated by invertible self-adjoint elements $T_i=T_{ii}$  for $1\leq i \leq N$ and elements $T_{ij}$ with $1\leq i<j\leq N$ such that \eqref{EqCommHolT} holds, as well as
\begin{equation}\label{EqXComStarClassT}
\left\{\begin{array}{lllll}
T_{kj} T_{li}^{*} - T_{li}^{*}T_{kj}&=&0,&& i\neq j, k\neq l,\\
T_{kj}T_{lj}^{*} - qT_{lj}^{*}T_{kj}&=& -(1-q^2)\sum_{i<j} T_{ki}T_{li}^{*} ,&&k\neq l,\\
qT_{kj}T_{ki}^{*}- T_{ki}^{*}T_{kj} &=&  (1-q^2)\sum_{k<l\leq\mathrm{min}\{i,j\}} T_{li}^{*}T_{lj},&& i\neq j, \\
T_{kj}T_{kj}^{*}- T_{kj}^{*}T_{kj} &=& (1-q^2)\left(\sum_{k<l\leq j} T_{lj}^{*}T_{lj} - \sum_{k\leq l<j} T_{kl}T_{kl}^{*}\right). &&
\end{array}\right. 
\end{equation}

Then $\mcO_q(T(N))$ is a Hopf $*$-algebra for the unique coproduct such that 
\begin{equation}\label{EqCoprodT}
(\id\otimes \Delta)(T) = T_{12}T_{13}. 
\end{equation}

The Hopf algebra isomorphism \eqref{EqUnivRelInv} extends uniquely to a Hopf $*$-algebra isomorphism 
\begin{equation}
\kappa: U_q(\mfu(N))^{\cop} \rightarrow \mcO_q(T(N)).
\end{equation}
whose inverse we again denote by 
\begin{equation}\label{EqUnivRelReal}
\iota: \mcO_q(T(N)) \rightarrow U_q(\mfu(N))^{\cop}  .
\end{equation}

Consider now the W$^*$-category $\mbT_q(N)$ of representations $\lambda$ of $\mcO_q(T(N))$ for which the $\lambda(T_{i})$ are positive operators. Equivalently, we can view $\mbT_q(N)$ as the monoidal W$^*$-subcategory of $\mbGL_q(N,\C)$ consisting of all $(\Hsp,T)$ for which $T$ is upper triangular, with positive, invertible operators on the diagonal. Again, the monoidal structure is the same as the one induced by the coproduct of $\mcO_q(T(N))$.  

By the $*$-isomorphism \eqref{EqUnivRelReal}, a finite-dimensional admissible $U_q(\mfu(N))$-representation becomes a $\mcO_q(T(N))$-representation in $\mbT_q(N)$, and we can transfer the unitary irreducible admissible representations $V_{\lambda}$ for $\lambda\in P_{\pos}^{\dom}$ to representations of $\mcO_q(T(N))$. We then have the following. 

\begin{Theorem}
Every representation in $\mbT_q(N)$ is a direct integral (with multiplicities) of irreducible representations. The irreducible representations in $\mbT_q(N)$ are finite-dimensional, and parametrized by $P_{\pos}^{\dom}$. 
\end{Theorem}
\begin{proof} We only sketch a possible proof. The boundedness of $T_i$ and $T_i^{-1}$, together with their $q$-commutation relations with $T_{i,i+1}$ and $T_{i,i+1}^*$,  implies that each irreducible representation in $\mbT_q(N)$ must be of the form $V_{\lambda}$ for $\lambda \in P_{\pos}^{\dom}$. It then follows from general C$^*$-algebra theory that if we have a representation in $\mbT_q(N)$ for which the central element $\Det_q(T)=T_1\ldots T_n$ acts as a scalar, the representation must be a direct sum of joint eigenspaces for the $T_i$, and hence must be a (possibly infinite) direct sum of $V_{\lambda}$'s. On the other hand, any representation in $\mbT_q(N)$ can be written as a direct integral of representations for which the central element $\Det_q(T)$ is a positive scalar. Combining these two facts, we obtain the result in the theorem.
\end{proof} 


\subsection{The W$^*$-category $\mbGL_q(N,\C)$}

\begin{Def}
We define the \emph{FRT double} $\mcO_{q}^{(2)}(M_N(\C))$ as the universal $\C$-algebra generated by the matrix entries of $X = \sum_{ij} e_{ij} \otimes X_{ij}$ and $Y =  \sum_{ij} e_{ij} \otimes Y_{ij}$ with universal relations
\begin{equation}\label{EqUniRelMN}
\hat{R}_{12}X_{13}X_{23} = X_{13}X_{23}\hat{R}_{12},\qquad 
\hat{R}_{12}Y_{23}Y_{13} = Y_ {23}Y_{13}\hat{R}_{12},\qquad 
X_{23}\hat{R}_{12}Y_{23} = Y_{13}\hat{R}_{12}X_{13}.
\end{equation}
We define $\mcO_{q}^{\R}(M_N(\C))$ as $\mcO_{q}^{(2)}(M_N(\C))$ endowed with the $*$-structure
\[
X^{*} = Y,\qquad Y^{*} = X.
\]
\end{Def}
The commutation relations can be written out as follows: apart from the relations in \eqref{EqCommHol}, we also have
\begin{equation}\label{EqXComStarClass}
\left\{\begin{array}{lllll}
X_{kj} X_{li}^{*} - X_{li}^{*}X_{kj}&=&0,&& i\neq j, k\neq l,\\
X_{kj}X_{lj}^{*} - qX_{lj}^{*}X_{kj}&=& -(1-q^2)\sum_{i<j} X_{ki}X_{li}^{*} ,&&k\neq l,\\
qX_{kj}X_{ki}^{*}- X_{ki}^{*}X_{kj} &=&  (1-q^2)\sum_{k<l} X_{li}^{*}X_{lj},&& i\neq j, \\
X_{kj}X_{kj}^{*}- X_{kj}^{*}X_{kj} &=& (1-q^2)\left(\sum_{k<l} X_{lj}^{*}X_{lj} - \sum_{l<j} X_{kl}X_{kl}^{*}\right). &&
\end{array}\right. 
\end{equation}

The element $\Det_q(X)$ is central in $\mcO_{q}^{\R}(M_N(\C))$ \cite[Lemma 3.5]{DCF19}, so with $\Det_q(X^{*}) := \Det_q(X)^{*}$, we can define the Hopf $*$-algebra 
$\mcO_q^{\R}(GL(N,\C))$ as the localisation
\[
\mcO_{q}^{\R}(GL(N,\C)) = \mcO_{q}^{\R}(M_N(\C))[\Det_q(X)^{-1},\Det_q(X^{*})^{-1}].
\]

The W$^*$-category $\mbGL_q(N,\C)$, defined below \eqref{EqHqNAlt}, can be identified with the representation category of $\mcO_q^{\R}(GL(N,\C))$: we only need to observe that since $\Det_q(X)$ arises as a direct summand of $X^{\otimes n}$, any representation of $\mcO_q^{\R}(M_N(\C))$ for which the image of $X$ is invertible, extends to a $\mcO_q^{\R}(\GL(N,\C))$-representation. 

We have natural quotient maps of Hopf $*$-algebras
\[
\pi_U: \mcO_q^{\R}(GL(N,\C)) \rightarrow \mcO_q(U(N)),\qquad X \mapsto U,
\]
\[
\pi_T: \mcO_q^{\R}(GL(N,\C))\rightarrow \mcO_q(T(N)),\quad X\mapsto T,
\]
leading to W$^*$-functors  
\[
\mbU_q(N) \times \mbT_q(N) \rightarrow \mbGL_q(N,\C),\qquad (\lambda,\lambda')\mapsto \lambda \oast \lambda' =  (\lambda \circ \pi_U)*(\lambda'\circ \pi_T),
\]
\[
\mbT_q(N) \times \mbU_q(N) \rightarrow \mbGL_q(N,\C),\qquad (\lambda',\lambda)\mapsto \lambda' \oast \lambda = (\lambda'\circ \pi_T)*(\lambda\circ \pi_U).
\]
The following is a categorical analogue of the Gauss decomposition.
\begin{Theorem}\label{TheoRepGLN}
For any irreducible representation $\lambda$ in  $\mbGL_q(N,\C)$, there exist unique (up to unitary equivalence) irreducible representations $\lambda_U,\lambda_U'$ in $\mbU_q(N)$ and unique (up to unitary equivalence) irreducible representations $\lambda_T,\lambda_T'$ in $\mbT_q(N)$ such that 
\[
\lambda \cong \lambda_U \oast \lambda_T \cong \lambda_T' \oast \lambda_U'. 
\]
\end{Theorem}
\begin{proof}
The proof of \cite[Proposition 4.18]{DCF19} can be immediately adapted. 
\end{proof}

More generally, \cite[Proposition 4.18]{DCF19} implies that any representation of $\mcO_q^{\R}(GL(N,\C))$ is weakly contained in a representation of the form $\lambda \oast \lambda'$ for $\lambda \in  \mbU_q(N)$ and $\lambda'\in\mbT_q(N)$.

\section{The W$^*$-category $\mbH_q(N)$}\label{section: REA intro}

\subsection{Quantum hermitian matrices}

Consider $\mcO_q(H(N))$ as defined via \eqref{EqUnivRelZ}. These relations can be written out entrywise as 
\begin{multline}\label{EqCommOqHN}
q^{-\delta_{ik} -\delta_{jk}} Z_{ij} Z_{kl} + \delta_{k<i}(q^{-1}-q) q^{-\delta_{ij}} Z_{kj}Z_{il} +\delta_{jk} (q^{-1}-q)q^{-\delta_{ij}}\sum_{p<j} Z_{ip} Z_{pl} + \delta_{ij} \delta_{k<i}(q^{-1}-q)^2  \sum_{p<i} Z_{kp}Z_{pl} \\
= q^{-\delta_{il}-\delta_{jl}} Z_{kl} Z_{ij} + \delta_{l<j}(q^{-1}-q)q^{-\delta_{ij}}  Z_{kj}Z_{il} + \delta_{il}(q^{-1}-q) q^{-\delta_{ij}} \sum_{p<i} Z_{kp}Z_{pj} +\delta_{ij}\delta_{l<j} (q^{-1}-q)^2  \sum_{p<j} Z_{kp}Z_{pl}. 
\end{multline}


In \cite[Theorem 2.2]{Mud02}, using just the above relations, Mudrov was able to classify all characters of the reflection equation algebra. From this, we can easily obtain also a classification of all \emph{$*$-characters}\footnote{Note that the convention for the $R$-matrix is different in \cite{Mud02}, but we can easily pass to our convention by replacing $q \leftrightarrow q^{-1}$ and $k \leftrightarrow N+1-k$.}
\[
\chi: \mcO_q(H(N)) \rightarrow \C.
\]

\begin{Theorem}\label{Character classification theorem}
Each $*$-character $\chi$ of $\mcO_q(H(N))$ is of the form $\chi = \chi_{k,l,a,c,y}$ where 
\begin{itemize}
\item $k,l\in \Z_{\geq0}$ with $k+l \leq N-l$,
\item $a>0$ and $c\in \R^{\times}$, 
\item $y = (y_{0},y_{1},\ldots,y_{l-1})$ with $|y_i| = 1$, 
\end{itemize}
such that
\[
\chi(Z) = c\left(a \sum_{i = k+l+1}^{N} e_{ii} - a^{-1} \sum_{i=N-l+1}^N e_{ii} +   \sum_{i=0}^{l-1} (y_{i} e_{k+i+1,N-i} + \overline{y_{i}} e_{N-i,k+i+1})\right).
\] 
\end{Theorem} 
\begin{proof}
It follows immediately from \cite[Theorem 2.2]{Mud02} upon checking self-adjointness of the solutions presented there. 
\end{proof}

For example, for $N= 4$ we have the following types of solutions, where empty entries signify $0$:
\[
\begin{pmatrix} &&& \\ &&& \\ &&& \\ &&& \end{pmatrix},\quad \begin{pmatrix} &&& \\ &&& \\ &&& \\ &&&  c \end{pmatrix},\quad \begin{pmatrix} &&& \\ &&& \\ &&  c& \\ &&&  c \end{pmatrix}, \quad \begin{pmatrix} &&& \\ &  c&& \\ && c& \\ &&&  c \end{pmatrix},\quad \begin{pmatrix}  c &&& \\ &  c&&\\ && c& \\ &&&  c \end{pmatrix}, 
\]
\[
 c \begin{pmatrix} &&& \\ &&& \\ &&& y_0 \\ && \overline{y_0} & a-a^{-1}\end{pmatrix},\quad  c\begin{pmatrix} &&& \\ &&& y_0 \\ && a& \\ &\overline{y_0} && a-a^{-1}\end{pmatrix},\quad  c\begin{pmatrix} &&& y_0 \\ & a && \\ && a & \\ \overline{y_0} &&& a-a^{-1}\end{pmatrix},\]
\[
 c\begin{pmatrix} &&& y_0  \\ &&y_1 & \\ &\overline{y_1} & a-a^{-1} & \\ \overline{y_0} &&&a-a^{-1}\end{pmatrix}.
\]

\subsection{Classification of irreducible representations for $N=2$}\label{SecExH2}

For $\mcO_q(H(2))$, one can work directly from the defining relations to obtain a classification of the irreducible representations. 

Fix in the following $N=2$, and write
\[
Z =  \begin{pmatrix} z & w \\ v & u\end{pmatrix} \in M_2(\mcO_q(H(2))).
\]
Define the \emph{quantum trace} $T$ and \emph{quantum determinant} $D$ to be 
\[
T = \Tr_q(Z) = qz + q^{-1}u,\qquad D = \Det_q(Z) = uz -q^{-2}vw.
\]
We find by a tedious but direct comparison with \eqref{EqCommOqHN} that the universal relations of $\mcO_q(H(2))$ are centrality and self-adjointness of $T,D$ together with the relations 
\[
z^* = z,\quad w^* = v,\quad zw = q^2wz,\qquad vz = q^2zv
\]
and
\[
q^{-2}vw = -D+qTz - q^{2}z^2,\qquad q^{-2}wv = -D+q^{-1}Tz -q^{-2}z^2.
\]
For $A$ a unital commutative $*$-algebra, write $\Spec_*(A)$ for the space of unital $*$-characters. For $A$ a general $*$-algebra, write $\msZ(A)$ for the center of $A$. 

Exploiting the fact that $z$ either commutes or $q^{\pm 2}$-commutes with all generators, one sees that $\pi(z)$ must be diagonalizable for any irreducible representation $\pi$, and that moreover $\pi(v)$ must vanish on some eigenvector of $\pi(z)$ if $\pi(z)$ is not identically zero. The reader will then have no difficulty in providing the details for the following proposition.
\begin{Prop}\label{PropRank1Tensor}
We have $\msZ(\mcO_q(H(2))) \cong \C[D,T]$. Identifying $\Spec_*(\C[D,T])\cong \R^2/\Sym(2)$ via 
\[
\chi \mapsto \{\textrm{roots of }\lambda^2 -\chi(T)\lambda + \chi(D)\},
\] 
we have that there exists a (non-zero) irreducible representation of $\mcO_q(H(2))$ fibering over $\lambda\in \R^2/\Sym(2)$ if and only if 
\[
\lambda \in  \Lambda_q = \Lambda_- \cup \Lambda_0 \cup \Lambda_+
\] 
with 
\begin{eqnarray*}
\Lambda_0 &=& \{\{0,\lambda\}\mid \lambda \in \R\},\\
\Lambda_+ &=& \{\{cq^{n+1},cq^{-n-1}\}\mid c\in \R^{\times}, n\in \Z_{\geq0}\},
\\
\Lambda_- &=& \{\{\lambda_+,\lambda_-\}\mid \lambda_+ \in \R_{>0},\lambda_- \in \R_{<0}\}.
\end{eqnarray*}
More precisely: 
\begin{enumerate}
\item[(1)]\label{PropRank1Tensor1}
There is a unique irreducible representation fibering over $\{cq^{n+1},cq^{-n-1}\}\in \Lambda_+$, namely 
\[
\mcS_{c^2,n}\cong \C^{n+1},\qquad ze_k= cq^{-n+2k}e_k,\quad ve_0 = 0.
\]
We say that it has rank $2$ and shape $\begin{pmatrix} \sgn(c) & 0 \\ 0 & \sgn(c) \end{pmatrix}$.
\item[(2)]\label{PropRank1Tensor2}
For $\{0,\lambda\} \in \Lambda_0$ we have the $*$-character $\chi(z) = \chi(v) = \chi(w) =0$. When $\lambda=0$ this is the only irreducible representation, and we say that it has rank $0$. If $\lambda\neq 0$ the only other irreducible representation is given by 
\[
\mcS_{0,\lambda} \cong l^2(\N),\quad ze_n = q^{2n+1}\lambda e_n,\quad ve_0 = 0,
\] 
and we say that it has rank $1$ and shape $\begin{pmatrix} 1 & 0 \\ 0 & 0 \end{pmatrix}$.
\item[(3)]\label{PropRank1Tensor3}
For $\{ca,-ca^{-1}\} \in \Lambda_-$ with $a,c>0$, we have as complete list of non-equivalent irreducible representations the following.
\begin{enumerate}
\item The $*$-characters $\chi_{\theta}$ for $\theta \in [0,1)$, with $\chi_{\theta}(z)= 0$ and $\chi_{\theta}(v) = qce^{2\pi i\theta}$. We say that they have rank $2$ and shape $\begin{pmatrix} 0 & e^{-2\pi i\theta} \\ e^{2\pi i\theta} & 0 \end{pmatrix}$.
\item The representations
\[
\mcS^{\pm}_{-c^2,a} \cong l^2(\N),\qquad ze_n = \pm ca^{\pm 1} q^{2n+1}e_n,\qquad ve_0 =0.
\]
We say that they have rank $2$ and shape $\begin{pmatrix} \pm 1 & 0 \\ 0 & \mp 1\end{pmatrix}$.
\end{enumerate}
\end{enumerate}
\end{Prop}

A general definition of the rank of a representation will be given in Section \ref{section: rank of rep}. The terminology `shape' will be further motivated and generalized in \cite{DCMo24}. 

\begin{Cor}\label{CorTypeIH2}
The $*$-algebra $\mcO_q(H(2))$ is type $I$.
\end{Cor}
\begin{proof}
One immediately checks that $\pi(z)$ is always a non-zero compact operator, for any of the non-scalar irreducible representations $\pi$ in Proposition \ref{PropRank1Tensor}.
\end{proof}

\subsection{Covariantisation}

For the following, see e.g.~ \cite{Maj95}, \cite[Chapter 8-10]{KS97} or \cite{DKM03}.

Recall the skew bicharacter $\mbr$ from \eqref{EqSkewBichar}. This defines a \emph{coquasitriangular structure} on $\mcO_q(GL(N,\C))$, in the sense that we also have the `braided commutativity' property
\begin{equation}\label{EqBraidComM}
\mbr(a_{(1)},b_{(1)})a_{(2)}b_{(2)} = b_{(1)}a_{(1)} \mbr(a_{(2)},b_{(2)}).
\end{equation}


Since $R$ is invertible, the functional $\mbr: A\otimes A \rightarrow \C$ has a convolution inverse $\mbr^{-1}$, where we endow $A\otimes A$ with the tensor product coalgebra structure. Then \eqref{EqBraidComM} can alternatively be written as 
\begin{equation}\label{EqBraidComMInv}
\mbr^{-1}(b_{(1)},a_{(1)})a_{(2)}b_{(2)} = b_{(1)}a_{(1)} \mbr^{-1}(b_{(2)},a_{(2)}),
\end{equation}
or as 
\begin{equation}\label{EqBraidComMInvAlt}
ab = \mbr^{-1}(a_{(1)},b_{(1)})b_{(2)}a_{(2)} \mbr(a_{(3)},b_{(3)}) = \mbr(b_{(1)},a_{(1)})b_{(2)}a_{(2)} \mbr^{-1}(b_{(3)},a_{(3)}).
\end{equation}

However, $\mbr$ also has a convolution inverse $\mbr'$ when considered as a functional on $(A,\Delta)\otimes (A,\Delta^{\opp})$. Then $\mbr'$ is uniquely determined by 
\[
\mbr'(ab,c) = \mbr'(a,c_{(2)})\mbr'(b,c_{(1)}),\quad \mbr'(a,bc) = \mbr'(a_{(1)},b)\mbr'(a_{(2)},c),\qquad \mbr'(a,1) = \varepsilon(a),\quad \mbr'(1,b) = \varepsilon(b)
\]
and, with $t$ the ordinary transpose of a matrix,
\[
(\id\otimes \id\otimes \mbr')(X_{13}X_{24}) = R' := ((R^{\id\otimes t})^{-1})^{\id\otimes t} = \sum_{ij} q^{\delta_{ij}}e_{ii}\otimes e_{jj} + (q-q^{-1})\sum_{i<j} q^{2(j-i)} e_{ij}\otimes e_{ji}. 
\]

These properties of $\mbr$ allow to endow $\mcO_q(M_N(\C))$ with the new product 
\begin{equation}\label{EqBraidProd}
f*g = \mbr(f_{(1)},g_{(2)})f_{(2)}g_{(3)} \mbr'(f_{(3)},g_{(1)}), 
\end{equation}
which can be checked to be associative and unital (with the same unit). Inverting \eqref{EqBraidProd}, we obtain that the original product can be obtained from the deformed product via
\begin{equation}\label{EqRevBraid}
fg = \mbr^{-1}(f_{(1)},g_{(1)})f_{(2)}*g_{(3)} \mbr(f_{(3)},g_{(2)}). 
\end{equation}

One has the following result, see e.g. \cite[Section 7.4]{Maj95} or \cite[Example 10.18]{KS97}. 
\begin{Theorem}\label{LemCommBraid}
There exists a (unique) isomorphism of algebras
\begin{equation}\label{EqDefPhi}
\Phi: \mcO_q(H(N)) \rightarrow (\mcO_q(M_N(\C)),*),\qquad Z \mapsto X. 
\end{equation}
\end{Theorem}

The $*$-structure of $\mcO_q(H(N))$ can be transported along $\Phi$ to a $*$-structure on $(\mcO_q(M_N(\C)),*)$, which we denote by $f \mapsto f^{\#}$. Concretely, we have 
\begin{equation}\label{EqBraidStar}
X_{ij}^{\#} = X_{ji}.
\end{equation}

The above isomorphism allows to define quantum minors inside $\mcO_q(H(N))$. 

\begin{Def}
For $I,J \in \binom{[N]}{k}$, we define the \emph{quantum minors}
\begin{equation}\label{EqBraidMinor}
Z_{IJ} = \Phi^{-1}(X_{IJ}). 
\end{equation}
\end{Def}

More generally, if $V$ is a $\mcO_q(M_N(\C))$-comodule, we denote by 
\begin{equation}\label{EqGenComodNot}
X_V(\xi,\eta) \in \mcO_q(M_N(\C)),\qquad \xi,\eta\in V
\end{equation}
the corresponding matrix units, and we similarly write 
\[
Z_V(\xi,\eta) = \Phi^{-1}(X_V(\xi,\eta)) \in \mcO_q(H(N)). 
\]
Then the quantum minors $Z_{IJ}$ arise as matrix coefficients of 
\begin{equation}\label{EqBraidMinorMatrix}
Z^{[k]} := (\id\otimes \Phi^{-1})X^{[k]}. 
\end{equation}

\begin{Lem}
The $*$-structure on $\mcO_q(H(N))$ satisfies 
\begin{equation}\label{EqStarHerm}
Z_{IJ}^* = Z_{JI},\qquad Z_V(\xi,\eta)^* = Z_V(\eta,\xi),
\end{equation}
\end{Lem}
\begin{proof}
First note that $\#$, as defined by \eqref{EqBraidStar}, extends uniquely to an anti-linear \emph{automorphism} on $\mcO_q(M_N(\C))$ with the original product. This automorphism will then be anti-comultiplicative: 
\[
\Delta(f^{\#}) =f_{(2)}^{\;\#}\otimes f_{(1)}^{\;\#},\qquad f\in \mcO_q(M_N(\C)). 
\] 
As then 
\begin{equation}\label{EqProprSharp}
\mbr(f^{\#},g^{\#}) = \overline{\mbr(g,f)},\qquad f,g\in \{X_{ij}\mid 1\leq i,j\leq N\},
\end{equation}
it follows by \eqref{EqDefProp} that \eqref{EqProprSharp} holds for all $f,g\in \mcO_q(M_N(\C))$. The same condition must then hold for $\mbr'$. But this is easily seen to imply (using \eqref{EqBraidProd} and \eqref{EqBraidComM}) that 
\[
(f*g)^{\#} = g^{\#}*f^{\#},\qquad f,g\in \mcO_q(M_N(\C)). 
\]
It follows that $z \mapsto \Phi^{-1}(\Phi(z)^{\#})$ must coincide with the original $*$-structure on $\mcO_q(H(N))$, as it coincides on the generators $Z_{ij}$. 

The conclusion now follows by the easy verification that 
\[
X_V(\xi,\eta)^{\#} = X_V(\eta,\xi),
\]
using that $\#$ is multiplicative on $\mcO_q(M_N(\C))$ with the original product. 
\end{proof}


\subsection{Commutation relations and Laplace expansions}

For $I,J\in \binom{[N]}{k}$ and $I',J'\in \binom{[N]}{l}$, denote
\begin{equation}\label{EqCoeffsRBraid}
\hat{R}^{IJ}_{I'J'} := \mbr(X_{JI},X_{I'J'})
\end{equation}
and 
\[
\hat{R}^{[k],[l]} = \sum_{I,J\in \binom{\brN}{k},I',J' \in \binom{\brN}{l}}\hat{R}^{IJ}_{I'J'} e_{I'I}\otimes e_{JJ'}. 
\]
By the braided commutativity \eqref{EqBraidComM}, one gets that 
\begin{equation}\label{EqBraidComm}
\hat{R}^{[k],[l]}_{12}X^{[k]}_{13} X^{[l]}_{23} = X^{[l]}_{13} X^{[k]}_{23} \hat{R}^{[k],[l]}_{12}.
\end{equation}

\begin{Lem}
We have
\begin{equation}\label{EqCondR}
\hat{R}^{IJ}_{I'J'} \neq 0 \quad \textrm{only if}\quad  J\preceq I,J'\preceq I'\quad \textrm{and}\quad J \setminus I = J'\setminus I',\;I\setminus J = I'\setminus J'.
\end{equation}
Moreover,
\[
\hat{R}^{II}_{I'I'} = q^{- |I\cap I'|}. 
\] 
\end{Lem}
\begin{proof}
This follows from \eqref{EqFactR}, \eqref{EqPropR3}, \eqref{EqPropR5} and the defining equation \eqref{EqCoeffsRBraid}.
\end{proof}

We also write $\hat{R}_{[k],[l]}^{-1} = (\hat{R}^{[k],[l]})^{-1}$ and
\[
\hat{R}_{[k],[l]}^{-1} = \sum_{I,J\in \binom{\brN}{k},I',J' \in \binom{\brN}{l}} (\hat{R}^{-1})_{I'J'}^{IJ} e_{JJ'} \otimes e_{I'I},\qquad (\hat{R}^{-1})_{I'J'}^{IJ}  = \mbr^{-1}(X_{JI},X_{I'J'}).
\]
The $(\hat{R}^{-1})_{I'J'}^{IJ}$ satisfy the same conditions \eqref{EqCondR} except that now $(\hat{R}^{-1})^{II}_{I'I'} = q^{|I\cap I'|}$. 

We have the following $q$-analogue of the Laplace expansion \cite[Theorem 4.4.3]{PW91}, which can be obtained by using the comodule algebra $\wedge_q(\C^N)$. We use the notation as introduced in \eqref{EqSetUnderUpper}.

\begin{Prop}\label{PropLaplace}
For $I,J\in \binom{\brN}{k}$ and $K,K'\in \binom{\brk}{l}$ the following identities hold in $\mcO_q(M_N(\C))$,
\begin{eqnarray}
\delta_{K,K'}X_{I,J} &=& \sum_{P \in \binom{\brk}{l}} (-q)^{\wt(P) -\wt(K)}X_{I_K,J_P}\; X_{I^{K'},J^P} \label{EqLaplaceRow}\\
&=& \sum_{P \in \binom{\brk}{l}} (-q)^{\wt(P) -\wt(K)}X_{I_P,J_K}\; X_{I^P,J^{K'}}.  \label{EqLaplaceColumn}
\end{eqnarray}
\end{Prop}

Using Theorem \ref{LemCommBraid} and \eqref{EqBraidComm}, we see that the quantum minors from \eqref{EqBraidMinorMatrix} satisfy
\begin{equation}\label{EqCommZGen}
\hat{R}^{[l],[k]}_{12} Z^{[k]}_{23} \hat{R}^{[k],[l]}_{12} Z^{[l]}_{23} = Z^{[l]}_{23} \hat{R}^{[l],[k]}_{12} Z^{[k]}_{23} \hat{R}^{[k],[l]}_{12}.
\end{equation}

By \eqref{EqCommZGen} we obtain the following commutation relations between the $Z_{I,J}$: for all $I,J,I',J'$
\begin{equation}\label{EqGenCommRel}
\sum_{K,L,L'} \left(\sum_{P'}\hat{R}_{JK}^{P'I'}\hat{R}_{P'L'}^{IL}\right)  Z_{K,L} Z_{L',J'} =   \sum_{K,L,L'}  \left(\sum_{P'}\hat{R}_{JK}^{P'L'}\hat{R}_{P'J'}^{IL}\right) Z_{I',L'} Z_{K,L}.
\end{equation}

Putting $I= J = [k]$ and using \eqref{EqCondR}, we obtain from \eqref{EqGenCommRel} that the \emph{leading quantum minors} 
\[
Z_{\brk} = Z_{\brk,\brk}
\] 
are mutually commuting self-adjoint elements which satisfy $q$-commutation relations with the $Z_{I,J}$, namely 
\begin{equation}\label{EqCommDs}
Z_{\brk} Z_{I,J} = q^{2|I\cap [k]| - 2|J\cap [k]|}Z_{I,J}Z_{\brk}.
\end{equation}
In particular, we can consider the localisation 
\begin{equation}\label{EqLocHN}
\mcO_q^{\loc}(H(N)) = \mcO_q(H(N))[Z_{[1]}^{-1},\ldots,Z_{\brN}^{-1}].
\end{equation}

Also the Laplace expansion has an analogue in $\mcO_q(H(N))$.

\begin{Prop}
For $I,J \in \binom{\brN}{k}$ and $K\in \binom{\brk}{m}$ with $m\leq k \leq N$ we have
\begin{equation}\label{EqLaplExp1}
Z_{I,J}= \sum_{P \in \binom{\brk}{m}}\sum_{S,T\in \binom{\brN}{m}}\sum_{S',T'\in \binom{\brN}{k-m}}(-q)^{\wt(P) -\wt(K)} (\hat{R}^{-1})^{S,I_K}_{I^K,T'}  \hat{R}^{J_P,T}_{T',S'} Z_{S,T} Z_{S',J^P}
\end{equation}
and
\begin{equation}\label{EqLaplExp2}
Z_{I,J} =  \sum_{P \in \binom{\brk}{m}}\sum_{S,T\in \binom{\brN}{m}}\sum_{S',T'\in \binom{\brN}{k-m}}(-q)^{\wt(P) -\wt(K)} (\hat{R}^{-1})^{T,J_K}_{J^K,S'}  \hat{R}^{I_P,S}_{S',T'} Z_{I^P,T'} Z_{S,T}.
\end{equation}
\end{Prop}
\begin{proof} 
Using \eqref{EqRevBraid}, we see that 
\begin{equation}\label{EqRevBraidConc}
X_{I,J}X_{I',J'} = \sum_{A,B,C,D} (\hat{R}^{-1})_{I',B}^{A,I} \hat{R}_{B,D}^{J,C} X_{A,C}*X_{D,J'}.
\end{equation}
Applying \eqref{EqRevBraidConc} to \eqref{EqLaplaceRow} and invoking Theorem \ref{LemCommBraid} and the concrete formulas \eqref{EqBraidProd}, we obtain \eqref{EqLaplExp1}. Applying $*$ we obtain \eqref{EqLaplExp2}. 
\end{proof}

\subsection{Actions by quantum groups}

The right action \eqref{EqDefAdGLq} of $\mbGL_q(N,\C)$ on $\mbH_q(N)$ is implemented by a coaction (= comodule $*$-algebra structure)
\begin{equation}\label{EqActAdqGL}
\Ad^{GL}_q: \mcO_q(H(N)) \rightarrow \mcO_q(H(N)) \otimes \mcO_q^{\R}(GL(N,\C)),\qquad Z \mapsto X_{13}^{*}Z_{12}X_{13}.
\end{equation}
The well-definedness of $\Ad^{GL}_q$ is immediate from the defining relations.

Similarly, the restrictions of the above action to actions of $\mbU_q(N)$ and $\mbT_q(N)$ on $\mbH_q(N)$, arise from coactions
\begin{equation}\label{EqActAdqU}
\Ad^{U}_q: \mcO_q(H(N)) \rightarrow \mcO_q(H(N)) \otimes \mcO_q(U(N)),\qquad Z \mapsto U_{13}^{*}Z_{12}U_{13}.
\end{equation}
\begin{equation}\label{EqActAdqT}
\Ad^{T}_q: \mcO_q(H(N)) \rightarrow \mcO_q(H(N)) \otimes \mcO_q(T(N)),\qquad Z \mapsto T_{13}^{*}Z_{12}T_{13}.
\end{equation}

Using the pairing \eqref{EqDualOqUqGL} between $\mcO_q(U(N))$ and $U_q(\mfu(N))$, the $*$-algebra $\mcO_q(H(N))$ becomes a left $U_q(\mfu(N))$-module $*$-algebra by considering the infinitesimal action dual to $\Ad^U_q$, 
\begin{equation}\label{EqInfAdj}
x\rhd z = (\id\otimes \tau(-,x))\Ad^U_q(z),\qquad x\in U_q(\mfu(N)),z\in\mcO_q(H(N)).
\end{equation}
One has the global formula
\begin{equation}\label{EqGlobForm}
x\rhd Z_{V}(\xi,\eta)= Z_{V}(S(x_{(1)})^*\xi,x_{(2)}\eta).
\end{equation}
Indeed, one notes that the right hand side also defines a module $*$-algebra structure on $\mcO_q(H(N))$, so that the identity only needs to be checked on the generating matrix, where it is obvious. 

Classically, the $U(N)$-orbit of the identity matrix in $H(N)$ is of course trivial under the adjoint action. On the other hand, by the Cholesky decomposition one has that the $T(N)$-orbit of the identity matrix under the adjoint action gives the space of all positive-definite matrices. As the latter are Zariski dense in $H(N)$, the following proposition is not surprising. We will use in the statement the map $\iota$ from \eqref{EqiotaGL}, the infinitesimal adjoint action \eqref{EqInfAdj} as well as the right adjoint action $\lhdb$ of $\mcO_q(T(N))$ on $\mcO_{q}(T(N))$, given by  
\begin{equation}\label{EqAdjT}
x \lhdb g = S(g_{(1)})xg_{(2)},\qquad x,g\in \mcO_q(T(N)),
\end{equation}
with $S$ denoting the antipode. 

\begin{Prop}\label{PropInclChol}
The $*$-homomorphism 
\[
i_T:\mcO_q(H(N)) \rightarrow \mcO_q(T(N)),\qquad Z \mapsto T^*T 
\]
is injective, and 
\[
i_{T}(\iota(S(g))\rhd f) = i_{T}(f) \lhdb g,\qquad f\in \mcO_q(H(N)), g\in \mcO_q(T(N)).
\]
\end{Prop}
\begin{proof}
This is well-known, cf.~\cite[Theorem 3]{Bau00}.
\end{proof} 


More generally, we have for any positive integral $U_q(\mfu(N))$-representation $V$ that
\begin{equation}\label{EqFormGlobClas}
(\id\otimes i_T) Z_V = T_V^*T_V,
\end{equation}
writing $T_V$ for the image of $X_V$ in $\End(V)\otimes \mcO_q(T(N))$.  See the proof of Lemma \ref{LemGlobalFormMapi} for a more general statement. In particular, 
\begin{equation}\label{EqDetFormT}
i_T(Z_{[k]}) =T_1^2\ldots T_k^2,\qquad 1 \leq k \leq N.
\end{equation}

\subsection{Inclusions and quotients between the $\mcO_q(H(N))$}

One can map $\mcO_q(H(K))$ into $\mcO_q(H(N))$ by placement in the upper left corner: 
\begin{equation}\label{EqInclCorn}
\mcO_q(H(K)) \rightarrow \mcO_q(H(N)),\qquad Z_{ij}\mapsto Z_{ij}. 
\end{equation}

Similarly, we can also project down to the right bottom corner:

\begin{Lem}\label{Lem OqHn quotient} Let $J_M$ be the 2-sided $*$-ideal generated by the $Z_{ij}$ with $1\leq i \leq M$ and $1\leq j \leq N$. Then 
\[
\mcO_q(H(N))/J_M \cong \mcO_q(H(N-M)),\qquad (Z_{ij})_{1\leq i,j\leq N} \mapsto \begin{pmatrix} 0_M & 0_{M,N-M} \\ 0_{N-M,M} & (Z_{ij})_{1\leq i,j\leq N-M}\end{pmatrix}
\]
is a $*$-homomorphism. 
\end{Lem}
\begin{proof}
This follows immediately from the universal relations \eqref{EqCommOqHN}.
\end{proof}

We will now show that, upon a small modification, the corner \eqref{EqInclCorn} can also be embedded in a different way. Namely, let $1\leq k \leq l \leq N$, and put $M = l-k+1$. Let $U_q(\mfu(M))_{k,l} \subseteq U_q(\mfu(N))$ be generated by the $E_i,F_i$ with $k\leq i <l$ and the $K_i^{\pm 1}$ with $k \leq i \leq l$. Then we can consider the restriction of the infinitesimal adjoint action $\rhd$ to $U_q(\mfu(M))$ via $U_q(\mfu(M)) \cong U_q(\mfu(M))_{k,l}$. 

\begin{Prop}\label{PropEquik}
There is a unique $U_q(\mfu(M))$-equivariant $*$-homomorphism
\begin{equation}\label{EqMapk}
\rho_{k,l}: \mcO_q(H(M)) \rightarrow \mcO_q(H(N))
\end{equation}
such that $Z_{\brone}\mapsto Z_{\brk}$. Moreover, we then have $\rho_{k,l}(Z_{\brM}) = Z_{[k-1]}^{M-1}Z_{\brl}$.
\end{Prop}

\begin{proof}
By \eqref{EqXComStarClassT} and \eqref{EqCoprodT}, it follows immediately that we have a Hopf $*$-algebra homomorphism
\[
\kappa_{k,l}: \mcO_q(T(M)) \rightarrow \mcO_q(T(N)),\qquad T_{ij} \mapsto T_{i+k-1,j+k-1}.
\]
Then as the adjoint action $\lhdb$, defined in \eqref{EqAdjT}, is inner, we have
\[
\kappa_{k,l}(x \lhdb y) = \kappa_{k,l}(x) \lhdb \kappa_{k,l}(y),\qquad x,y\in \mcO_q(T(M)).
\]
Consider now the modified $*$-homomorphism 
\[
\widetilde{\kappa}_{k,l}: \mcO_q(T(M)) \rightarrow \mcO_q(T(N)),\quad T_{ij} \mapsto T_1\ldots T_{k-1} \kappa_{k,l}(T_{ij}),
\]
using that $T_i$ for $i<k$ commutes with all elements of $\kappa_{k,l}(\mcO_q(T(M)))$. Then we clearly still have 
\[
\widetilde{\kappa}_{k,l}(x \lhdb y) = \widetilde{\kappa}_{k,l}(x) \lhdb \kappa_{k,l}(y), \qquad x,y\in \mcO_q(T(M)).
\]
Pulling $\widetilde{\kappa}_{k,l}$ back through the respective embeddings $i_T$, we hence obtain from Proposition \ref{PropInclChol} a $U_q(\mfu(M))$-equivariant $*$-homomorphism \eqref{EqMapk}, which must be unique since $Z_{11}$ is a generator for $\mcO_q(H(M))$ as a module $*$-algebra. It is also immediate from the above construction and \eqref{EqDetFormT} that $\rho_{k,l}(Z_{\brM}) = Z_{[k-1]}^{M-1}Z_{\brl}$. 
\end{proof}

\subsection{Rank of a representation}\label{section: rank of rep}

\begin{Def}\label{DefQuotRankM}
Let $1\leq M \leq N$. We define $I_M$ to be the 2-sided $*$-ideal of $\mcO_q(H(N))$ generated by all $Z_{I,J}$ with $I,J \in \binom{\brN}{k}$ for $k\geq M$.  We define 
\[
\mcO_q^{\leq M}(H(N)) = \mcO_q(H(N))/I_{M+1}.
\]
\end{Def}

By convention, we put $I_{N+1} =\{0\}$.

\begin{Lem}
The ideal $I_M$ is generated as a 2-sided $*$-ideal by the $Z_{I,J}$ with $I,J\in \binom{\brN}{M}$. 
\end{Lem} 
\begin{proof}
This follows directly from the Laplace expansion \eqref{EqLaplExp1}.
\end{proof}

\begin{Def}
We say that a representation $\pi$ of $\mcO_q(H(N))$ is of rank $M$ if $I_{M+1} \subseteq \Ker(\pi)$ but $I_{M} \nsubseteq \Ker(\pi)$. 
\end{Def}

\section{Spectral weight}\label{section: spectral weight}


\subsection{Spectral weight}\label{SecSpectWeight}


Let $\msZ(\mcO_q(H(N)))$ be the center of $\mcO_q(H(N))$, and let $\mcO_q(H(N))^{\Ad^U_q}$ be the $*$-algebra of coinvariants for the coaction \eqref{EqActAdqU},
\[
\mcO_q(H(N))^{\Ad^U_q} = \{z\in \mcO_q(H(N)) \mid \Ad^U_q(z) = z\otimes 1\}.
\]

The following result is well-known. For the convenience of a direct reference for the $GL$-setting, we refer to \cite[Lemma 3.30]{DCF19}.
\begin{Theorem}\label{TheoCenter}
An element of $\mcO_q(H(N))$ is central if and only if it is $\Ad^U_q$-coinvariant, i.e.
\[
\msZ(\mcO_q(H(N))) = \mcO_q(H(N))^{\Ad^U_q}.
\]
\end{Theorem} 

The center of $\mcO_q(H(N))$ can further be described explicitly, see \cite[Theorem 1.3]{JW20} or also \cite[(1.12)]{Flo20}.

\begin{Theorem}\label{TheoCentrPol}
We have $\msZ(\mcO_q(H(N))) \cong \C[\sigma_1,\ldots,\sigma_N]$, where the $\sigma_k$ are self-adjoint, algebraically independent elements given by\footnote{We have rescaled the elements with respect to the above references.} 
\[
\sigma_k =  q^{2Nk} \sum_{I\in \binom{\brN}{k}} \sum_{\sigma \in \Sym(I)}q^{-2\wt(I)} (-q)^{-l(\sigma)} q^{-a(\sigma)} Z_{i_k\sigma(i_k)}\ldots Z_{i_1 \sigma(i_1)},
\] 
where $\Sym(I)$ is the group of permutations of $\brN$ leaving the complement of $I = \{i_1<\ldots < i_k\}$ pointwise fixed, where $\wt(I)$ is as in \eqref{EqWeightVecto}, and where $a(\sigma) = |\{l\mid \sigma(l)<l\}|$ (called the \emph{anti-exceedance}). 

Moreover, the following \emph{quantum Cayley-Hamilton} identity holds: with $\sigma_0 = 1$, we have
\begin{equation}\label{EqCHQuant}
\sum_{k=0}^N (-1)^k \sigma_k Z ^{N-k} = 0.
\end{equation} 
\end{Theorem}

In particular, we define the \emph{quantum trace} and \emph{quantum determinant}
\[
\Tr_q(Z) =  q^{-(N-1)}\sigma_1 = q^{N+1}\sum_{i=1}^N q^{-2i}Z_{ii},
\]
\[
\Det_q(Z) = q^{-N(N-1)}\sigma_N = \sum_{\sigma \in \Sym([N])} (-q)^{-l(\sigma)} q^{-a(\sigma)} Z_{N,\sigma(N)}\ldots Z_{1,\sigma(1)},
\]
It can be shown \cite{JW20} that  
\[
Z_{\brN} = \Det_q(Z) .
\]
As we have a unital $*$-homomorphism $\mcO_q(H(K)) \rightarrow \mcO_q(H(N))$ by $Z \mapsto (Z_{ij})_{1\leq i,j\leq K}$, we  obtain
\begin{equation}\label{EqFormMinorZ}
Z_{\brk} = \sum_{\sigma \in \Sym([k])} (-q)^{-l(\sigma)} q^{-a(\sigma)} Z_{k,\sigma(k)}\ldots Z_{1,\sigma(1)},\qquad 1\leq k \leq N.
\end{equation}
Applying $*$, we also obtain the alternative version 
\begin{equation}\label{EqFormMinorZAlt}
Z_{\brk} = \sum_{\sigma \in \Sym([k])} (-q)^{-l(\sigma)} q^{-a(\sigma)} Z_{\sigma(k),k}\ldots Z_{\sigma(1),1},\qquad 1\leq k \leq N.
\end{equation}

If $\pi$ is a factor representation of $\mcO_q(H(N))$, then $\pi$ restricts to a $*$-character on $\msZ(\mcO_q(H(N))$. 

\begin{Def}
If $\pi: \mcO_q(H(N)) \rightarrow B(\Hsp)$ is a factor representation, we denote 
\[
\chi_{\pi}^{\msZ}: \msZ(\mcO_q(H(N))) \rightarrow \C,\qquad \sigma_k \mapsto \sigma_k^{\pi}
\]
for the associated $*$-character on $\msZ(\mcO_q(H(N))$, and call it the \emph{central $*$-character} of $\pi$.

We call $s\in \R^N$ \emph{centrally admissible} if there exists a factor representation $\pi$ such that 
\[
\sigma_k^{\pi} = s_k,\qquad  1\leq k \leq N.
\]  
We call $\lambda \in \R^N/\Sym(N)$ a \emph{spectral weight} if there exists a centrally admissible $s \in \R^N$ for which $\lambda$ is the multiset of roots of the polynomial
\[
P_s(x) = \sum_k (-1)^k s_k x^{k},
\]
and we denote 
\[
\Lambda_q = \{ \textrm{spectral weights} \}.
\]
For $\pi$ a factor representation of $\mcO_q(H(N))$, we denote by $\lambda_{\pi}$ the spectral weight of $\pi$, given as the multiset of roots of $P_s$ for $s_k = \sigma_k^{\pi}$. 
\end{Def}

Note that Theorem \ref{TheoEigZ} concretely identifies the set $\Lambda_q$. For now, we will content ourselves to show in the next section that $\Lambda_q$ parametrises $\mbU_q(N)$-orbits.

\subsection{Orbits under $\mbU_q(N)$}\label{SecApCQG}

The results in this section will be a bit more general than needed for our immediate purposes. We start with the following well-known general fact.

\begin{Prop}\label{PropAlwaysBounded}
Let $\mathbb{G}$ be a compact quantum group with associated Hopf $*$-algebra $H = \mcO(\mathbb{G})$, and let $A$ be a unital $*$-algebra with coaction 
\[
\alpha: A \rightarrow A\otimes H.
\]
Assume that the algebra of coinvariants $A^{\alpha} = \{a \in A \mid \alpha(a) = a\otimes 1\}$ is trivial, 
\[
A^{\alpha} = \C. 
\]
Then all representations of $A$ on pre-Hilbert spaces are bounded, and $A$ admits a universal C$^*$-envelope.
\end{Prop}
\begin{proof}
Any element $v\in A$ lies in a finite dimensional comodule $V$, which can be given an invariant Hilbert space structure. Denoting $b_1,\ldots,b_n$ for an orthonormal basis with respect to this Hilbert space structure, we have that $\sum_i b_i^*b_i$ is a coinvariant vector, hence equal to a scalar $C$. Then all $b_i$, and in particular $v$, are (uniformly) bounded under any representation of $A$ on a pre-Hilbert space. 
\end{proof}

To apply this to the representation theory of $\mcO_q(H(N))$, note that, by Theorem \ref{TheoCentrPol}, any $s\in \R^N$ determines a $*$-character $\chi_s: \msZ(\mcO_q(H(N))) \rightarrow \C$ such that $\chi_{s}(\sigma_k) = s_k$ for all $1\leq k\leq N$. 

\begin{Def}\label{DefQuantOrb}
For $s\in \R^N$, we define the $*$-algebra $\mcO_q(O_s)$ as the quotient of $\mcO_q(H(N))$ by the extra relations 
\[
\sigma_k = s_k.
\]
We call $\mcO_q(O_s)$ the \emph{quantum orbit $*$-algebra} at parameter $s$. 
\end{Def}
These quantum orbit algebras were studied from the viewpoint of deformation quantization in \cite{DM02b}. 

The following is clear by Theorem \ref{TheoCenter}.
\begin{Prop}\label{PropTrivializesCent}
The coaction $\Ad^U_q$ descends to a coaction of $\mcO_q(U(N))$ on $\mcO_q(O_s)$, and 
\begin{equation}\label{EqTrivCoinv}
\mcO_q(O_s)^{\Ad^U_q} = \C. 
\end{equation}
\end{Prop}

By Proposition \ref{PropAlwaysBounded}, we now obtain:
\begin{Cor}\label{CorUniversalEnv}
Each $\mcO_q(O_s)$ admits a universal C$^*$-envelope $C_q(O_s)$. More generally, a $*$-representation of $\mcO_q(H(N))$ on a pre-Hilbert space is bounded as soon as it is bounded on its center.
\end{Cor} 

Note that the canonical map $\mcO_q(O_s) \rightarrow C_q(O_s)$ need not be injective, even when $s$ is centrally admissible.






Let us return now to a general compact quantum group $\mathbb{G}$ with associated Hopf $*$-algebra $H = \mcO(\mathbb{G})$. Let $L^{\infty}(\mathbb{G})$ be its associated von Neumann algebra, obtained as the double commutant of $\mcO(\mathbb{G})$ in its GNS-representation $\lambda_{\standard}$ with respect to its invariant state $\int_{\mathbb{G}}$. 

\begin{Def}
Let $A$ be a unital $*$-algebra with a coaction 
\[
\alpha: A \rightarrow A\otimes \mcO(\mathbb{G}).
\]
We say that a representation $\pi$ of $A$ is $\alpha$-compatible if $\alpha$ descends to a coaction on $\pi(A)$. We say that moreover $\pi$ is \emph{normal $\alpha$-compatible} if $\alpha$ extends to a normal $*$-homomorphism $\pi(A)'' \rightarrow \pi(A)''\overline{\otimes} L^{\infty}(\mathbb{G})$. 
\end{Def}

When $\pi_1,\pi_2$ are representations of a unital $*$-algebra $A$, we say that $\pi_1$ and $\pi_2$ are \emph{stably unitarily equivalent} if $\oplus_{i\in I} \pi_1$ is unitarily equivalent to $\oplus_{j\in J} \pi_2$ for some index sets $I,J$. 

\begin{Prop}\label{PropNorComp}
Let $A$ be a unital $*$-algebra with coaction $\alpha: A \rightarrow A\otimes \mcO(\mathbb{G})$.
\begin{enumerate}
\item\label{EqFirstPoint} If $\pi$ is any representation of $A$, then $\pi *\lambda_{\standard}$ is a normal $\alpha$-compatible representation, where $\pi*\lambda_{\standard} = (\pi\otimes \lambda_{\standard})\circ \alpha$. Conversely, any normal $\alpha$-compatible representation $\pi$ is stably unitarily equivalent to $\pi * \lambda_{\standard}$. 
\item\label{EqSecondPoint} If $\pi$ is a normal $\alpha$-compatible representation of $A$, and $\pi_{\mathbb{G},\mathbb{H}}: \mcO(\mathbb{G}) \rightarrow \mcO(\mathbb{H})$ defines a compact quantum subgroup, then $\alpha_{\mid \mathbb{H}} = (\id\otimes \pi_{\mathbb{G},\mathbb{H}})\alpha$ is a coaction by $\mathcal{O}(\mathbb{H})$ and $\pi$ is a normal $\alpha_{\mid \mathbb{H}}$-compatible representation.
\end{enumerate}
\end{Prop}
\begin{proof}
Let $V$ be the right regular unitary corepresentation
\[
L^2(\G) \otimes L^2(\G) \rightarrow L^2(\G) \otimes L^2(\G), \quad f\otimes g \mapsto \Delta(f)(1\otimes g),\qquad f,g\in\mcO(\G).
\]
If then $\pi$ is a representation of $A$, we have that 
\[
(\pi*\lambda_{\standard}\otimes \lambda_{\standard})(\alpha(x))= V_{23} (\pi*\lambda_{\standard}(x)\otimes 1)V_{23}^*,
\]
proving that $\pi*\lambda_{\standard}$ is normal $\alpha$-compatible. This proves the first half of \eqref{EqFirstPoint}. 

Conversely, if $\pi$ is normal  $\alpha$-compatible, let $N = \pi(A)''$ and consider the normal coaction $\alpha_{\pi}: N \rightarrow N \overline{\otimes}L^{\infty}(\G)$. Let $E: N \rightarrow N^{\alpha}$ be the faithful normal conditional expectation 
\[
E(n) = \left(\id\otimes \int_{\G}\right)\alpha_{\pi}(n) \in N^{\alpha},\qquad n \in N.
\]
Then on $L^2(N)$, as obtained from $L^2(N^{\alpha})$ through induction by $E$, we obtain the unitary 
\[
V_{\pi,\alpha}: L^2(N) \otimes L^2(\G) \rightarrow L^2(N) \otimes L^2(\G),\qquad n\xi\otimes \eta \mapsto \alpha(n)(\xi\otimes \eta),\qquad n\in N,\xi\in L^2(N^{\alpha}),\eta\in L^2(\G).
\]
Then, with $\pi_N$ the GNS-representation of $N$, 
\[
\pi_N*\lambda_{\standard}(n) = V_{\pi,\alpha}(\pi_N(n)\otimes 1)V_{\pi,\alpha}^*.
\]
As $\pi$ and $\pi_N$ are stably unitarily equivalent (being faithful normal representations of $N$), this finishes the proof of \eqref{EqFirstPoint}. Moreover, if $\mathbb{H}$ is a subgroup of $\G$, consider 
\[
V_{\G,\mathbb{H}}: L^2(\G) \otimes L^2(\mathbb{H})\rightarrow L^2(\G) \otimes L^2(\mathbb{H}),\quad f\otimes h \mapsto (\id\otimes \pi_{\G,\mathbb{H}})\Delta(f)(1\otimes h),\qquad f\in \mcO(\G),h\in \mcO(\mathbb{H}). 
\]
Then by \cite[Theorem 2.7]{Vae01}, we obtain a unique normal coaction $\alpha_{\pi,\mathbb{H}}$ of $L^{\infty}(\mathbb{H})$ on $N$ such that 
\[
(V_{\G,\mathbb{H}})_{23}\alpha_{\pi}(n)_{12}(V_{\G,\mathbb{H}})_{23}^* = (V_{\pi,\alpha})_{12}\alpha_{\pi,\mathbb{H}}(n)_{13}(V_{\pi,\alpha})_{12}^*.
\]
It is then easily verified that 
\[
\alpha_{\pi,\mathbb{H}}(\pi(a)) = (\pi \otimes \pi_{\mathbb{G},\mathbb{H}})\alpha(a),\qquad a\in A,
\]
proving \eqref{EqSecondPoint}.
\end{proof}

The first point of the previous proposition can be strengthened under the condition that $A$ has trivial coinvariants.
\begin{Prop}\label{PropNorCompErg}
Let $A$ be a unital $*$-algebra with coaction $\alpha: A \rightarrow A\otimes \mcO(\mathbb{G})$, and assume $A^{\alpha} = \C$. Then any two normal $\alpha$-compatible representations are stably unitarily equivalent. 
\end{Prop} 
\begin{proof}
Reprising the notation of the proof of the previous proposition, the conditional expectation $E$ now becomes a faithful normal state $\varphi$. It is then clear that $N$ must be the von Neumann algebra obtained from the GNS-representation of $A$ with respect to $\varphi$. But as this means $N$ is independent of $\pi$, this proves the proposition.   
\end{proof}

\begin{Lem}\label{LemWeakContGen}
Let $\mathbb{G}$ be a coamenable compact quantum group, and let $\alpha:A \rightarrow A \otimes \mcO(\mathbb{G})$ be a coaction on a unital $*$-algebra $A$. Let $\pi$ be a representation of $A$. Then $\pi$ is weakly contained in $\pi*\lambda_{\standard} = (\pi \otimes \lambda_{\standard})\circ \alpha$. 
\end{Lem}
\begin{proof}
Let $C(\mathbb{G})$ be the closure of $\mcO(\mathbb{G})$ in its standard representation. As $\mathbb{G}$ is coamenable, the C$^*$-algebra $C(\mathbb{G})$ admits the counit $*$-character $\varepsilon$. So $\pi$ factors through $\pi *\lambda_{\standard}$ by applying $\iota \otimes \varepsilon$. 
\end{proof}

\begin{Cor}\label{CorWeakContErg}
Let $\mathbb{G}$ be a coamenable compact quantum group, and let $A \rightarrow A \otimes \mcO(\mathbb{G})$ be a coaction on a unital $*$-algebra $A$ with trivial coinvariants. Let $\pi,\pi'$ be representations of $A$. Then $\pi'\preccurlyeq  \pi *\lambda_{\standard}$. 
\end{Cor} 
\begin{proof}
Let $A_u$ be the universal C$^*$-envelope of $A$, and let $\pi_u$ be some faithful representation of $A_u$. Then $\alpha$ extends to a coaction $\alpha_u: A_u \rightarrow A_u \otimes C(\mathbb{G})$ which still has trivial coinvariants. In particular, we obtain on $A_u$ a unique state $\varphi$ determined by 
\[
\varphi(x)1_{A_u} = \left(\id\otimes \int_{\mathbb{G}}\right)\alpha_u(x). 
\]
Since $\int_{\mathbb{G}}$ is faithful on $C(\mathbb{G})$, and $\alpha_u$ is faithful, it follows that $\varphi$ is faithful. Hence $A_u \cong C^*(\pi*\lambda_{\standard})$, and in particular $\pi' \preccurlyeq  \pi*\lambda_{\standard}$. 
\end{proof}

Returning now to $\mathbb{G} = U_q(N)$, let us say that two factor representations $\pi,\pi'$ of $\mcO_q(H(N))$ lie in the same $\mbU_q(N)$-orbit if there exist $\lambda,\lambda'\in \mbU_q(N)$ such that 
\[
\pi \preccurlyeq \pi'*\lambda,\qquad \pi'\preccurlyeq \pi*\lambda.
\]

The following theorem is a quantum analogue of the spectral theorem for finite-dimensional self-adjoint matrices.
\begin{Theorem}
Two factor representations in $\mbH_q(N)$ lie in the same $\mbU_q(N)$-orbit if and only if they have the same spectral weight. 
\end{Theorem}
\begin{proof}
Assume that $\pi,\pi'$ have the same spectral weight. Then they factor through a representation of $\mcO_q(O_s)$ for some $s$, and so $\pi,\pi'$ lie in the same $\mbU_q(N)$-orbit by Corollary \ref{CorWeakContErg}.

Conversely, assume $\pi,\pi'$ lie in the same $\mbU_q(N)$-orbit, and choose a representation $\rho \in \mbU_q(N)$ such that 
\[
\pi' \preccurlyeq \pi*\rho.
\]
Since $\mcO_q(H(N))^{\Ad_q} = \msZ(\mcO_q(H(N)))$, we find that 
\[
(\pi*\rho)(\sigma_k) = \pi(\sigma_k)\otimes 1,\qquad 1\leq k \leq N
\]
so that $\pi*\rho$ factors over $\mcO_q(O_s)$. As $\pi' \preccurlyeq \pi*\rho$, the same must be true for $\pi'$. This implies that the spectral weights of $\pi$ and $\pi'$ must coincide.
\end{proof}








\section{Big cell representations}\label{section: big cell representations}

\subsection{Big cell representations}

\begin{Def}
We define a representation $\pi$ of $\mcO_q(H(N))$ to be a \emph{big cell representation of rank $M$} if $\pi$ is of rank $M$ and $\Ker(\pi(Z_{\brk})) = 0$ for all $k\leq M$. 

We define a representation to be a (generic) big cell representation if it is a direct sum of big cell representations (of possibly varying rank).    
\end{Def}



To understand big cell representations, it is convenient to consider a deformation of $\mcO_q(T(N))$. This is what we turn to next. 


\subsection{Deformed quantized function algebras $\mcO_q^{\epsilon}(T(N))$}

For $\epsilon \in\R^N$, let 
\begin{equation}\label{EqEtaEps}
\epsilon_{(i,j]} = \epsilon_{i+1}\ldots \epsilon_j,\qquad (\eta_{\epsilon})_j = \epsilon_{[j]} = \epsilon_{(0,j]}.
\end{equation}
Define, for $\epsilon \in \R^N$, the \emph{$\epsilon$-deformed braid operator} by
\[
\hat{R}_{\epsilon} = \sum_{ij} q^{-\delta_{ij}} e_{ji}\otimes e_{ij} +(q^{-1}-q)\sum_{i<j} \epsilon_{(i,j]} e_{jj}\otimes e_{ii},
\]
When $\epsilon_i = 1$ for all $i$, we write $\epsilon = +$, so e.g.\ $\hat{R}_+ = \hat{R}$.  

\begin{Def}
We define $\mcO_{q,\epsilon}^{\R}(B(N))$ to be the algebra generated by $\mcO_q(B(N))$ and an independent copy of $\mcO_q(B(N))^*$, together with the extra relations 
\begin{equation}\label{EqUniRelMN}
T_{23}\hat{R}_{12}T^*_{23} = T^*_{13}\hat{R}_{\epsilon,12}T_{13}.
\end{equation}
We define $\mcO_q^{\epsilon}(T(N))$ as the quotient of $\mcO_{q,\epsilon}^{\R}(B(N))$ by the further relations $T_{i}^* = T_i$ for all $i$. 
\end{Def}

We then have that $\mcO_q^{\epsilon}(T(N))$ is the universal unital $*$-algebra generated by invertible elements $T_i$ for $1\leq i \leq N$ and elements $T_{ij}$ with $1\leq i<j\leq N$ such that \eqref{EqCommHolT} hold, as well as
\begin{equation}\label{EqXComStarClassT}
\left\{\begin{array}{lllll}
T_{kj} T_{li}^{*} - T_{li}^{*}T_{kj}&=&0,&& i\neq j, k\neq l,\\
T_{kj}T_{lj}^{*} - qT_{lj}^{*}T_{kj}&=& -(1-q^2)\sum_{i<j} T_{ki}T_{li}^{*} ,&&k\neq l,\\
qT_{kj}T_{ki}^{*}- T_{ki}^{*}T_{kj} &=&  (1-q^2)\sum_{k<l\leq\mathrm{min}\{i,j\}} \epsilon_{(k,l]}T_{li}^{*}T_{lj},&& i\neq j, \\
T_{kj}T_{kj}^{*}- T_{kj}^{*}T_{kj} &=& (1-q^2)\left(\sum_{k<l\leq j} \epsilon_{(k,l]} T_{lj}^{*}T_{lj} - \sum_{k\leq l<j} T_{kl}T_{kl}^{*}\right). &&
\end{array}\right. 
\end{equation}

When $\epsilon = +  = (+,+,\ldots,+)$, we find back $\mcO_q(T(N))$. 

In what follows we need a slightly more general class of $*$-algebras $\mcO_q^{\epsilon}(T(N))$ parametrized by $\epsilon \in \sqcup_{M=0}^N \R^M$.

\begin{Def}
Assume that $M \leq N$, and let $\epsilon \in \R^M$. We define $\mcO_q^{\leq M}(B(N))\subseteq \mcO_q(B(N))$ as the unital subalgebra generated by the $T_i^{\pm 1}$ and $T_{ij}$ with $i\leq M$ and $j\leq N$. We define $\mcO_{q,\epsilon}^{\R}(B(N))$  to be the $*$-subalgebra of $\mcO_{q,\widetilde{\epsilon}}^{\R}(B(N))$ generated by $\mcO_q^{\leq M}(B(N))$, where 
\begin{equation}\label{EqWideEps}
\widetilde{\epsilon}  = (\epsilon,0,\ldots,0)\in \R^N.
\end{equation}
Similarly, we define $\mcO_q^{\epsilon}(T(N))$ as the $*$-subalgebra of $\mcO_q^{\widetilde{\epsilon}}(T(N))$ generated by the $T_i^{\pm1}$ and $T_{ij}$ with $i\leq M$ and $j\leq N$. 
\end{Def} 


We denote  
\[
\pi_T: \mcO_{q,\epsilon}^{\R}(B(N)) \rightarrow \mcO_q^{\epsilon}(T(N))
\]
for the natural quotient map. 

\begin{Lem}
The multiplication maps
\begin{equation}\label{EqBijMult3}
\mcO_q^{\leq M}(B(N))^* \otimes \mcO_q^{\leq M}(B(N)) \rightarrow \mcO_{q,\epsilon}^{\R}(B(N)),\quad x^*\otimes y\mapsto x^*y
\end{equation}
\begin{equation}\label{EqBijMult4}
\mcO_q^{\leq M}(B(N)) \otimes \mcO_q^{\leq M}(B(N))^* \rightarrow \mcO_{q,\epsilon}^{\R}(B(N)),\quad x\otimes y^*\mapsto xy^*
\end{equation}
are bijective. 
\end{Lem}
Here we note that, on the left hand side, we consider $\mcO_q^{\leq M}(B(N))^*$ to be a formal anti-linear, anti-isomorphic copy of $\mcO_q^{\leq M}(B(N))$. 
\begin{proof}
Let us first prove the result for $\epsilon \in \R^N$. From the concrete commutation relations \eqref{EqXComStarClassT} and an easy induction argument, it is already clear that the maps will be surjective.

Let us show that \eqref{EqBijMult4} is injective, the proof for \eqref{EqBijMult3} is similar. Note first that there exists a unique skew bicharacter
\begin{equation}\label{EqPairingreps}
\mbr_{\epsilon}: \mcO_q(GL(N,\C)) \otimes \mcO_q(GL(N,\C)) \rightarrow \C
\end{equation}
such that, with $R_{\epsilon} = \Sigma \circ \hat{R}_{\epsilon}$, 
\[
(\id\otimes \id\otimes \mbr_{\epsilon})X_{13}X_{24} = R_{\epsilon}.
\]
In fact, $\mbr_{\epsilon}$ is obtained by pairing with $(\nu_{\epsilon}\otimes \id)\msR = (\id\otimes \nu_{\epsilon})\msR$, where $\nu_{\epsilon}$ is (the completion of) the Hopf algebra endomorphism of $U_q(\mfb(N))$, resp.~ $U_q(\mfb^-(N))$ with 
\begin{equation}\label{Eqnu}
\nu_{\epsilon}(K_i) = K_i,\qquad \nu_{\epsilon}(E_i) = \epsilon_{i+1}E_i,\qquad \nu_{\epsilon}(F_i) = \epsilon_{i+1}F_i.
\end{equation}
This $\mbr_{\epsilon}$ descends to a skew bicharacter
\[
\mbr_{\epsilon}: \mcO_q(B(N)) \otimes \mcO_q(B^-(N)) \rightarrow \C.
\]
We can then define on $\mcO_q^{\leq M}(B(N))^* \otimes \mcO_q^{\leq M}(B(N))$ a unique associative product such that 
\[
(z^*\otimes x)(y^*\otimes w) := \mbr_{\epsilon}(S(x_{(1)}),y_{(1)}^*)z^*y_{(2)}^{*}\otimes x_{(2)}w \mbr(x_{(3)},y_{(3)}^*),
\]
where we view $\mcO_q(B^-(N)) = \mcO_q(B(N))^*$ with respect to the compact real form on $\mcO_q(GL(N,\C))$. The universal relations for $\mcO_{q,\epsilon}^{\R}(B(N))$ now allow to construct a splitting map 
\[
 \mcO_{q,\epsilon}^{\R}(B(N)) \rightarrow \mcO_q^{\leq M}(B(N)) \otimes \mcO_q^{\leq M}(B(N))^* 
\]
for \eqref{EqBijMult4}, from which the injectivity of \eqref{EqBijMult4} directly follows. 

Assume now in general that $\epsilon \in \R^M$. Clearly, the multiplication maps \eqref{EqBijMult3} and \eqref{EqBijMult4} will then be injective by the first part of the proof. But they are also still surjective: all the right hand terms of \eqref{EqXComStarClassT} involving $T_{ij}$ for $i >M$ involve a factor $\widetilde{\epsilon}_i = 0$, from which it follows directly that the ranges of our multiplication maps are already closed under multiplication. 
\end{proof}



Fix $\epsilon \in \R^N$, and recall the notation \eqref{EqDualOqMN} and \eqref{EqEtaEps}. Define $\msE_{\epsilon} \in \msU_q^{\post}(\mfgl(N,\C))$ by 
\[
\msE_{\epsilon}\xi = (\eta_{\epsilon})_{\wt(\xi)}\xi,
\]
which is meaningful since $\xi \in V_{\lambda}$ has $\wt(\xi) \in P_{\posint}$ for $\lambda \in P_{\posint}^{\dom}$. In the above, we put $0^0=1$ by convention. For example, for the vector module we get 
\[
E_{\epsilon} = \theta(\msE_{\epsilon}) = \sum_i \epsilon_{[i]}e_{ii}.
\]



\begin{Lem}
There is a unique  $*$-homomorphism
\begin{equation}\label{EqiB}
i_{B}^{\epsilon}: \mcO_q(H(N)) \rightarrow \mcO_{q,\epsilon}^{\R}(B(N)),\quad Z \mapsto T^{*}(E_{\epsilon}\otimes 1) T.
\end{equation}
\end{Lem}
\begin{proof}
This follows directly from the defining relations, using the easily computed identity
\[
(1\otimes E_{\epsilon})\hat{R}_{\epsilon}(1\otimes E_{\epsilon})\hat{R} = \hat{R}(1\otimes E_{\epsilon})\hat{R}_{\epsilon}(1\otimes E_{\epsilon}).
\]
\end{proof}
We denote then also 
\begin{equation}\label{EqiT}
i_{T}^{\epsilon}: \mcO_q(H(N)) \rightarrow \mcO_{q}^{\epsilon}(T(N)),\quad Z \mapsto T^{*}(E_{\epsilon}\otimes 1) T.
\end{equation}

We can also give again a more global form of these maps as follows, using notation as in \eqref{EqFormGlobClas} and viewing $T_V \in \End(V)\otimes \mcO_q(B(N))$:
\begin{Lem}\label{LemGlobalFormMapi}
The map $i_{B}^{\epsilon}$ from \eqref{EqiB} is given by 
\begin{equation}\label{EqFormGlob}
i_{B}^{\epsilon}:   \mcO_q(H(N)) \rightarrow \mcO_{q,\epsilon}^{\R}(B(N)),\quad Z_V(\xi,\eta) \mapsto \sum_{ij} T_V(e_i,\xi)^* \langle e_i,\msE_{\epsilon}e_j\rangle T_V(e_j,\eta),
\end{equation}
where $\{e_i\}$ is an arbitrary orthonormal basis of $V$. 
\end{Lem} 
\begin{proof}
Viewing $U_q(\mfgl(N,\C)) \subseteq \msU_q^{\post}(\mfgl(N,\C))$, it is easily seen that $\epsilon_{i+1} \msE_{\epsilon} E_i = E_i \msE_{\epsilon}$. Hence, with $\nu_{\epsilon}$ as in \eqref{Eqnu}, we have $\msE_{\epsilon} \nu_{\epsilon}(X) = X\msE_{\epsilon}$  for $X \in U_q(\mfb(N))$, and so 
\[
(\msE_{\epsilon}\otimes 1)\msR_{\epsilon}  = \msR(\msE_{\epsilon}\otimes 1),\qquad \textrm{with } \msR_{\epsilon} = (\nu_{\epsilon}\otimes \id)\msR.
\] 
Since also $\msE_{\epsilon}$ is grouplike, i.e.\ $\Delta^{\post}(\msE_{\epsilon}) = \msE_{\epsilon}\otimes \msE_{\epsilon}$, we find that 
\[
\msR\Delta^{\post}(\msE_{\epsilon}) = (\msE_{\epsilon} \otimes 1)\msR_{\epsilon}(1\otimes \msE_{\epsilon}).
\]
Noting that $(\id\otimes S)\msR_{\epsilon}$ is the inverse of $\msR_{\epsilon}$ in $\msU_q(\mfgl(N,\C))\widehat{\otimes} \msU_q(\mfgl(N,\C))^{\opp}$, it is now easily verified through the global formula \eqref{EqDefPhi} for the product of $\mcO_q(H(N))$, that the map defined by \eqref{EqFormGlob} is a $*$-homomorphism. To show that it coincides with $i_{B}^{\epsilon}$ it is sufficient to verify equality on the generating matrix $Z$, which is immediate. 
\end{proof}

Recall now the localisation $\mcO_q^{\loc}(H(N))$ from \eqref{EqLocHN}. More generally, recalling Definition \ref{DefQuotRankM}, we can consider the localisation 
\[
\mcO_{q,\leq M}^{\loc}(H(N)) = \mcO_{q}^{\leq M}(H(N))[Z_1^{-1},\ldots,Z_M^{-1}].
\] 
Our aim is to give a different description of  $\mcO_{q,\leq M}^{\loc}(H(N))$. We start with the following lemma.

\begin{Lem}\label{LemImZ}
We have
\[
i_{B}^{\epsilon}(Z_{\brk}) = (\eta_{\epsilon})_1 T_1^*T_1\ldots (\eta_{\epsilon})_k T_k^*T_k,\qquad i_{T}^{\epsilon}(Z_{\brk}) = (\eta_{\epsilon})_1 T_1^2\ldots (\eta_{\epsilon})_k T_k^2.
\]
\end{Lem}
\begin{proof}
This follows immediately upon applying the formula \eqref{EqFormGlob}, the fact that $T_{[k]}(\eta,\xi_{\varpi_{\brk}}) = 0$ for $\eta \perp \xi_{\varpi_{\brk}}$ by (for example) the non-degeneracy of \eqref{EqNonDegPair}, and the fact that $T_{\brk} = T_1\ldots  T_k$.
\end{proof}

\begin{Prop}\label{PropLocHT}
Let $\epsilon \in (\R\setminus\{0\})^M$. Then the map $i_T^{\epsilon}$ from \eqref{EqiT} factors through an injective $*$-homomorphism
\[
j_T^{\epsilon}: \mcO_{q,\leq M}^{\loc}(H(N)) \rightarrow \mcO_q^{\epsilon}(T(N)).
\]
Moreover, the range of $j_T^{\epsilon}$ is the unital $*$-algebra generated by all $T_i^{\pm 2}$ for $1\leq i \leq M$ and $T_iT_{ij}$ with $1\leq i\leq M$ and $i<j\leq N$. 
\end{Prop} 

Note that the proof will actually show that the restriction of $j_T^{\epsilon}$ to $\mcO_{q}^{\leq M}(H(N))$ is injective, so that it indirectly also shows that $\mcO_{q}^{\leq M}(H(N))$ embeds in its localisation. 
\begin{proof}
Let $\widetilde{\epsilon}$ be as in \eqref{EqWideEps}. 
It is clear by the formula for \eqref{EqFormGlob}, the definition of $I_{M+1}$ in Definition \ref{DefQuotRankM} and Lemma \ref{LemImZ} that $i_B^{\epsilon}$ factors through a $*$-homomorphism 
\begin{equation}\label{EqFactB}
j_B^{\epsilon}: \mcO_{q,\leq M}^{\loc}(H(N)) \rightarrow \mcO_{q,\epsilon}^{\R}(B(N))\subseteq \mcO_{q,\widetilde{\epsilon}}^{\R}(B(N))
\end{equation}

Choose now for each $\lambda \in P_{\posint}^{\dom}$ an orthonormal basis of weight vectors $\{\xi_{\lambda,i}\mid i \in I_{\lambda}\}$ for $V_{\lambda}$, and put $\omega_{\lambda,i} = \wt(\xi_{\lambda,i})$ and $\alpha_{\lambda,i} = \lambda - \omega_{\lambda,i}$. Recalling the notation \eqref{EqPosRootAdj}, we can then choose $x_{\lambda,i} \in U_q(\mfn)_{\alpha_{\lambda,i}}$ with $x_{\lambda,i}\xi_{\lambda,j} = \delta_{ij}\xi_{\lambda}$ whenever $\alpha_{\lambda,i} = \alpha_{\lambda,j}$. 

Assume then that we have a non-trivial finite linear combination 
\[
C = \sum_{\lambda\in P_{\posint}^{\dom}}\sum_{k,l\in I_{\lambda}} a_{\lambda kl} Z_{\lambda}(\xi_{\lambda,k},\xi_{\lambda,l}) \in \Ker(i_B^{\epsilon}),
\] 
with $a_{ikl} =0$ for all but finitely many indices. We will show that $C\in I_{M+1}$. 

Note first that $i_B^{\epsilon}(C)=0$ gives
\[
\sum_{\lambda}\sum_{k,l,r\in I_{\lambda}} a_{\lambda k l} T_{\lambda}(\xi_{\lambda,r},\xi_{\lambda,k})^* \langle \xi_{\lambda,r},\msE_{\widetilde{\epsilon}}\xi_{\lambda,r}\rangle T_{\lambda}(\xi_{\lambda,r},\xi_{\lambda,l}) =0. 
\] 
From \eqref{EqBijMult3} we deduce that
\begin{equation}\label{EqLinInd}
\sum_{\lambda}\sum_{k,l,r\in I_{\lambda}}  \langle \xi_{\lambda,r},\msE_{\widetilde{\epsilon}}\xi_{\lambda,r}\rangle a_{\lambda kl} T_{\lambda}(\xi_{\lambda,r},\xi_{\lambda,k})^* \otimes T_{\lambda}(\xi_{\lambda,r},\xi_{\lambda,l}) =0. 
\end{equation}
Let now $K_{\alpha,\beta}$, for $\alpha,\beta \in Q^+$, be the set of $(\lambda,k,l)$ with 
\[
a_{\lambda,k,l}\neq0,\qquad   \wt(\xi_{\lambda}) - \wt(\xi_{\lambda,k})= \alpha,\qquad \wt(\xi_{\lambda}) - \wt(\xi_{\lambda_i,l})= \beta.
\] 
Let $K$ be the set of $(\alpha,\beta)$ with $K_{\alpha,\beta}\neq \emptyset$, and pick $(\alpha_0,\beta_0)\in K$ maximal with respect to the coordinatewise partial order $(\alpha,\beta)\preccurlyeq (\alpha',\beta') \Leftrightarrow \alpha \preccurlyeq \alpha'$ and $\beta\preccurlyeq \beta'$. Pick $(\lambda_0,k_0,l_0) \in K_{\alpha_0,\beta_0}$. Then $x_{\lambda_0,k_0}\xi_{\lambda,k}\otimes x_{\lambda_0,l_0}\xi_{\lambda,l}$ will be a multiple of $\xi_{\lambda}\otimes \xi_{\lambda}$ for each $(\lambda,k,l) \in K_{\alpha_0,\beta_0}$, and zero for all other $(\lambda,k,l)\in K_{\alpha,\beta}$ with $(\alpha,\beta)\neq (\alpha_0,\beta_0)$. Hence upon applying $(\id\otimes \tau(-,x_{\lambda_{i_0},l_0}))\Delta$ to the second tensorand and $(\id\otimes \overline{\tau(x_{\lambda_{i_0},k_0},-*)})\Delta$ to the first tensorand of \eqref{EqLinInd}, we see that there exist $c_{\lambda kl} \in \C$ with $c_{\lambda_0,k_0,l_0} = 1$, with all other $c_{\lambda_0,k,l}= 0$, and  
\[
\sum_{(\lambda,k,l) \in K_{\alpha_0,\beta_0}}\sum_{r\in I_{\lambda}}  \langle \xi_{\lambda,r},\msE_{\widetilde{\epsilon}}\xi_{\lambda,r}\rangle a_{\lambda kl}c_{\lambda kl} T_{\lambda}(\xi_{\lambda,r},\xi_{\lambda})^* \otimes T_{\lambda}(\xi_{\lambda,r},\xi_{\lambda}) =0, 
\]
which simplifies to 
\[ 
\sum_{(\lambda,k,l) \in K_{\alpha_0,\beta_0}}  \langle \xi_{\lambda},\msE_{\widetilde{\epsilon}}\xi_{\lambda}\rangle a_{\lambda kl}c_{\lambda kl} T_{\lambda}(\xi_{\lambda},\xi_{\lambda})^* \otimes T_{\lambda}(\xi_{\lambda},\xi_{\lambda}) =0. 
\]
Since the $T_{\lambda}(\xi_{\lambda},\xi_{\lambda})$ are linearly independent, we deduce that 
\[
\langle \xi_{\lambda_{0}},\msE_{\widetilde{\epsilon}}\xi_{\lambda_{0}} \rangle = (\eta_{\widetilde{\epsilon}})_{\lambda_0} = 0.
\] 
By our assumption on $\epsilon$, this can only happen if $(\lambda_0)_k >0$ for some $k\geq M+1$. 

Now the linear span of all the $X_{\lambda}(\xi,\eta)$ with $\lambda_k>0$ for some $k\geq M+1$ equals the $2$-sided ideal generated by all $X_{[k]}(\xi,\eta)$ for $k\geq M+1$ in $\mcO_q(M_N(\C))$, since any $V_{\lambda}$ can be identified with the representation spanned by the tensor product of the highest weight vectors in $V_{[1]}^{\otimes\lambda_1}\otimes \ldots \otimes V_{[N]}^{\otimes \lambda_N}$. It follows that $I_{M+1}$ is the linear span of all the $Z_{\lambda}(\xi,\eta)$ with $\lambda_k>0$ for some $k\geq M+1$. In particular, $Z_{\lambda_0}(\xi_{\lambda_0,k_0},\xi_{\lambda_0,l_0}) \in I_{M+1}$. By induction on the number of non-zero terms in $C$, we now find that $C \in I_{M+1}$. It follows that \eqref{EqFactB} is injective.

Let us now determine the range of \eqref{EqFactB}. Let $\msA$ be the $*$-subalgebra of $\mcO_{q,\widetilde{\epsilon}}^{\R}(B(N))$ generated by the $(T_{ii}^{*}T_{ii})^{-1}$ and $T_{ii}^*T_{ij}$ for $1\leq i \leq M$ and $1\leq j \leq N$. Since 
\begin{equation}\label{EqAcB}
(\id\otimes i^{\widetilde{\epsilon}}_B)(Z) =  T^*\begin{pmatrix} \widetilde{\epsilon}_{[1]} & & \\ & \ddots & \\ & & \widetilde{\epsilon}_{[N]}\end{pmatrix} T,
\end{equation}
it follows by induction that $j^{\epsilon}_B(\mcO_{q,\leq M}^{\loc}(H(N)))$ contains all $\widetilde{\epsilon}_{[i]}T_{ii}^*T_{ij}$, and hence  $\msA \subseteq j^{\epsilon}_B(\mcO_{q,\leq M}^{\loc}(H(N))$. But \eqref{EqAcB} also shows that this inclusion holds in the other direction, so $\msA = j^{\epsilon}_B(\mcO_{q,\leq M}^{\loc}(H(N))$. 

The proposition now follows since the map $\pi_T: \msA \rightarrow \mcO_q^{\epsilon}(T(N))$ is injective.
\end{proof}

We have for $\mcO_q^{\epsilon}(T(N))$ a triangular decomposition: with $\mcO_q^{\leq M}(\msN(N)) \subseteq \mcO_q^{\leq M}(B(N))$ the unital algebra generated by the $T_{ij}$ with $i<j$, and with $\mcO_q(D(M))$ the algebra of Laurent polynomials in the $T_i$ for $1\leq i \leq M$, we have that multiplication provides an isomorphism 
\begin{equation}\label{EqTriangDecomp}
\mcO_q^{\leq M}(\msN(N)) \otimes \mcO_q(D(M)) \otimes \mcO_q^{\leq M}(\msN(N))^* \rightarrow \mcO_q^{\epsilon}(T(N)). 
\end{equation}
Consider the associated projection map 
\[
\msP_{\epsilon}: \mcO_q^{\epsilon}(T(N)) \rightarrow \mcO_q(D(M)),\quad XYZ \mapsto \varepsilon(X)Y\varepsilon(Z). 
\]
Then with $\mcO_q^{\epsilon}(T(N))_0$ the centralizer of $\mcO_q(D(M))$ in $\mcO_q^{\epsilon}(T(N))$, we obtain that $\msP_{\epsilon}:\mcO_q^{\epsilon}(T(N))_0 \rightarrow \mcO_q(D(M))$ is a $*$-homomorphism. In particular, we obtain a $*$-homomorphism
\[
\chi_{\HC}^{\epsilon} = \msP_{\epsilon} \circ i_T^{\epsilon}: \msZ(\mcO_q(H(N))) \rightarrow \mcO_q(D(M)),
\]
called the \emph{Harish-Chandra map}. 
\begin{Prop}\label{PropHC}
With $e_k$ the $k$-th elementary symmetric polynomial in $N$ variables, we have
\begin{equation}\label{EqHC}
\chi_{\HC}^{\epsilon}(\sigma_k) =  e_k(\widetilde{\epsilon}_{[1]}T_1^2,\widetilde{\epsilon}_{[2]}q^2T_2^2,\ldots, \widetilde{\epsilon}_{[N]}q^{2N-2}T_N^2). 
\end{equation}
\end{Prop}
Note that the right hand side indeed lies in $\mcO_q(D(M))$. 
\begin{proof}
If $\epsilon = +$ this is well-known, see e.g.~ \cite[Corollary 3.34]{DCF19}. To obtain the general case, let $\mcO_q(\widetilde{D}(N)) \subseteq \mcO_q^{\loc}(H(N))$ be generated by the $Z_{[k]}^{\pm 1}$, and let $\mcO_q(\widetilde{\msN}(N))$ be generated by the inverse images of the $T_iT_{ij}$ under $i_T$. Then it is clear from Proposition \ref{PropLocHT} that we obtain a triangular decomposition 
\[
\mcO_q(\widetilde{\msN}(N)) \otimes \mcO_q(\widetilde{D}(N)) \otimes \mcO_q(\widetilde{\msN}(N))^* \rightarrow \mcO_q^{\loc}(H(N)). 
\]
with associated projection map 
\[
\widetilde{\msP}: \mcO_q^{\loc}(H(N)) \rightarrow \mcO_q(\widetilde{D}(N)),\quad XYZ \mapsto \varepsilon(i_T(X))Y\varepsilon(i_T(Z)). 
\]
This triangular decomposition is then however seen to be compatible with any of the embeddings $i_T^{\epsilon}$ for $\epsilon \in (\R\setminus \{0\})^N$, i.e.~
\[
\msP_{\epsilon} \circ i_T^{\epsilon} = i_T^{\epsilon}\circ \widetilde{\msP}.
\]
The identity \eqref{EqHC} now follows in this case from Lemma \ref{LemImZ}. The general case follows since both sides of \eqref{EqHC}  are continuous in $\epsilon$.
\end{proof}

\subsection{Big cell representations of $\mcO_q^{\epsilon}(T(N))$}

To classify the irreducible big cell representations of $\mcO_q(H(N))$, we will relate them to the representation theory of the $\mcO_q^{\epsilon}(T(N))$.

We first introduce some terminology. 

\begin{Def}\label{def: T(N) highest weight}
Let $V$ be an $\mcO_q^{\epsilon}(T(N))$-module, where $\epsilon \in \R^M$. 

We call a vector $\xi \in V$ a \emph{weight vector} if it is a joint eigenvector of all $T_i$, for $1\leq i \leq M$. We call $\xi \in V$ a \emph{highest weight vector} if $\xi$ is a weight vector and $\mcO_q^{\leq M}(\msN(N))^*\xi = 0$. We call a weight vector $\xi \in V$ \emph{admissible} if the associated eigenvalues of the $T_i$ are strictly positive. We then say that $\xi$ has \emph{weight $r \in \R^M$} if $T_{i} \xi =q^{r_i} \xi$ for all $1\leq i\leq M$.  

We call $V$ \emph{admissible} if it is finitely generated, has a vector basis of admissible weight vectors, and is locally finite as an $\mcO_q^{\leq M}(\msN(N))^*$-module, i.e.~ the $\mcO_q^{\leq M}(\msN(N))^*\xi$ are finite dimensional for each $\xi\in V$. 

We call $V$ an \emph{admissible unitary module} if moreover $V$ is a pre-Hilbert space for which the $*$-structure is compatible with the inner product,
\[
\langle \xi,x\eta\rangle = \langle x^*\xi,\eta\rangle,\qquad \forall \xi,\eta \in V,x\in \mcO_q^{\epsilon}(T(N)).
\]
\end{Def}

Recall the triangular decomposition \eqref{EqTriangDecomp}. It then follows immediately from this decomposition that the weight spaces of admissible $\mcO_q^{\epsilon}(T(N))$-modules are finite dimensional. It is then also easily concluded that any admissible unitary module is a direct sum of \emph{admissible unitary highest weight modules}, i.e.~ admissible unitary modules generated by a highest weight vector.

More generally, we call an $\mcO_q^{\epsilon}(T(N))$-module $V$ a \emph{highest weight module} if it is generated by a highest weight vector. For example, using again the triangular decomposition \eqref{EqTriangDecomp}, we can define a notion of Verma module $M_r$ at a given highest weight $r$, uniquely determined by the fact that, with $\xi_r$ the highest weight vector, the map 
\[
\mcO_q^{\leq M}(\msN(N)) \rightarrow M_r,\quad x \mapsto x\xi_r
\]
is a linear isomorphism. Then, up to a scalar, there exists a unique $*$-invariant Hermitian scalar product $\langle -,-\rangle$ on $M_r$. Following for example the discussion in \cite[Section 5]{JL92}, we also see that if $V$ is an arbitrary highest weight module with highest weight $r$, then from the fact that the kernel of the natural map $M_r \rightarrow V$ must necessarily be in the kernel of $\langle -,-\rangle$, it follows that (up to a scalar) there exists a unique $*$-invariant Hermitian scalar product on $V$. We call $V$ \emph{unitarizable} if this scalar product is positive-definite. Note that if the Hermitian scalar product is only semi-positive-definite, then its kernel will be $\mcO_q^{\epsilon}(T(N))$-invariant, and the quotient by this kernel will be unitarizable.  In any case, we always normalize the scalar product such that the highest weight vector has norm $1$. We see in particular that there is a one-to-one correspondence between unitarizable highest weight modules and admissible unitary highest weight modules (with their normalized highest weight vectors fixed). 

We summarize the above results in the following proposition.

\begin{Prop}\label{PropUniqueHW}
For $\epsilon \in \R^M$, there exists an admissible unitary highest weight-module $V_r$ for  $\mcO_q^{\epsilon}(T(N))$ with highest weight $r\in \R^M$ if and only if the invariant form on the Verma module $M_r$ is positive semi-definite. Moreover, $V_r$ is then unique up to unitary equivalence.
\end{Prop}

\begin{Def}
We call a vector $r\in \R^M$ \emph{$\epsilon$-admissible} if there exists an admissible unitary highest weight module for $\mcO_q^{\epsilon}(T(N))$ with highest weight $r$.  
\end{Def}

Let us now return to $\mcO_q(H(N))$. Assume that $\pi$ is an \emph{irreducible} big cell representation of rank $M$ on a Hilbert space $\Hsp_{\pi}$. By irreducibility, it will necessarily factor over some $\mcO_q(O_s)$ (Definition \ref{DefQuantOrb}), and we then also call $\pi$ a big cell representation of $\mcO_q(O_s)$. Furthermore, there exist unique signs $\eta_k \in \{\pm1\}$ such that the $\eta_k \pi(Z_{[k]})$ for $k \leq M$ are positive operators. We will then refer to $\eta$ or $\widetilde{\eta} = (\eta,0,0,\ldots,0) \in \{-1,0,1\}^N$ as the \emph{signature} of $\pi$ (we will show in Theorem \ref{TheoUnique} that there is no clash with the notion of signature from the Introduction). In the following we will write again
\[
\eta_{I} = \eta_{i_1}\eta_{i_2}\ldots\eta_{i_k},\qquad I = \{i_1,\ldots,i_k\} \in \binom{\brN}{k}. 
\]
By irreducibility, it is also clear that the $Z_{[k]}$ will have a joint eigenvector basis, so that $V_{\pi}$, the sum of their weight spaces, is dense in $\Hsp_{\pi}$. The subspace $V_{\pi}$ will clearly be invariant under $\mcO_q(H(N))$. 

In the following proposition, we reprise the notation from \eqref{EqiT} and  Proposition \ref{PropLocHT}.

\begin{Prop}\label{PropCorrHT}
Let $\pi$ be an irreducible big cell representation of $\mcO_q(H(N))$ of signature $\eta\in \{\pm 1\}^M$, and let $\epsilon \in \{\pm 1\}^M$ be the unique sign vector such that 
\[
\eta_i = (\eta_{\epsilon})_i =  \epsilon_{[i]},\qquad 1\leq i \leq M.
\]
\begin{itemize}
\item The restriction of $\pi$ to $V_{\pi}$ can be uniquely extended to an admissible unitary highest weight module $\pi'$ of $\mcO_q^{\epsilon}(T(N))$ on $V_{\pi}$ through the map $i_{T}^{\epsilon}$.
\item Conversely, any admissible unitary highest weight module of $\mcO_q^{\epsilon}(T(N))$ restricts and completes to an irreducible big cell representation of rank $M$ for $\mcO_q(H(N))$. 
\end{itemize}
\end{Prop}
\begin{proof}
As the $Z_{[k]}$ are diagonalizable with spectrum of constant sign $\eta_{[k]}$, it follows from Lemma \ref{LemImZ} and Proposition \ref{PropLocHT} that we can extend the restriction of $\pi$ to $V_{\pi}$ uniquely to an admissible unitary highest weight module $\pi'$ of $\mcO_q^{\epsilon}(T(N))$, and $V_{\pi}$ is then the linear span of weight vectors for $\mcO_q^{\epsilon}(T(N))$. As each $\pi(Z_{[k]})$ for $1\leq k\leq M$ is bounded, it also follows that any vector in $V_{\pi}$ must be locally $\mcO_q^{\leq M}(\msN(N))^*$-finite. In particular, there must exist a highest weight vector $\xi$ for $\pi'$. Since $\pi(\mcO_q(H(N)))\xi = \pi'(\mcO_q^{\epsilon}(T(N)))\xi$, it follows that $\pi(\mcO_q(H(N)))\xi$ is a direct sum of finite dimensional weight spaces for $\mcO_q(H(N))$.  Since the closure of $\pi(\mcO_q(H(N)))\xi$ equals $\Hsp_{\pi}$ by irreducibility of $\pi$, this implies that $V_{\pi} = \mcO_q(H(N)) \xi$, and hence $\xi$ is a cyclic highest weight vector for $\pi'$. 

We have now shown that $\pi'$ is an admissible unitary highest weight module. To see conversely that any admissible unitary highest weight module $(V_{\pi'},\pi')$ of $\mcO_q^{\epsilon}(T(N))$ restricts to an irreducible big cell representation of $\mcO_q(H(N))$ with signature $\eta$, let us first verify that $\pi = \pi'\circ i_T^{\epsilon}$ determines a bounded $*$-representation of $\mcO_q(H(N))$ on the pre-Hilbert space $V_{\pi'}$. Indeed, it is clear that $\pi$ will factor over a $*$-character on the center of $\mcO_q(H(N))$. Boundedness then follows from Corollary \ref{CorUniversalEnv}. To see that $\pi$ is irreducible, note that any bounded intertwiner $T$ for $\pi$ must preserve the weight spaces of the $Z_{[k]}$. Then $T$ preserves $V_{\pi'}$ and is also an intertwiner for the $\mcO_q^{\epsilon}(T(N))$-representation, hence constant. 
\end{proof}


The classification of the irreducible big cell representations thus follows from the classification of the admissible unitary highest weight modules of the $\mcO_q^{\epsilon}(T(N))$. To state the latter classification result, we first introduce the following auxiliary terminology.

\begin{Def}\label{DefEpsAda}
Let $\epsilon \in \{-1,0,1\}^{M}$. We say that $r \in \R^M$ is \emph{$\epsilon$-adapted} if and only if the following condition holds:  for any $1\leq s<t\leq M$ with $\epsilon_{(s,t]} = 1$ one has 
\begin{equation}\label{EqIneqPos}
(r_t+t)-(r_s+s)  \in \Z_{>0}.
\end{equation}
\end{Def}

It is easily seen that when $\epsilon \in \{\pm 1\}^M$, then \eqref{EqIneqPos} is satisfied if and only if there exist $\alpha_+,\alpha_- \in \R$ and $N_+,N_- \in \Z_{\geq 0}$ with $N_++ N_- = M$, a subset $I = \{i_1<\ldots<i_{N_+}\} \in \binom{[M]}{N_+}$ and integers $m_{i_1}<m_{i_2} < \ldots <m_{i_{N_+}}$ and  $m_{i_1'} <m_{i_2'} <\ldots  < m_{i_{N_-}'}$ for $I' = \{i_1'<\ldots< i_{N_-}\} = [M]\setminus I$, such that 
\[
r_s+s = \alpha_{\epsilon_{[s]}} + m_s,\qquad \forall 1\leq s \leq M. 
\]

\begin{Theorem}\label{TheoClassIrrepT}
Let $\epsilon \in \{-1,0,1\}^M$. Then $r$ is $\epsilon$-admissible for $\mcO_q^{\epsilon}(T(N))$ if and only if $r$ is $\epsilon$-adapted.
\end{Theorem} 

We will prove this theorem in the next section. For now, let us deduce the consequences for the representation theory of $\mcO_q(H(N))$. 

\begin{Theorem}\label{TheoBigCellRep}
An irreducible big cell representation $\pi$ of $\mcO_q(H(N))$ of signature $(\eta_1,\ldots,\eta_m)$ factors through a unique admissible unitary highest weight module $V_{r(\pi)}$ of $\mcO_q^{\epsilon}(T(N))$, where $\eta_i = \epsilon_{[i]}$. Moreover, two irreducible big cell representations $\pi_1,\pi_2$ are unitary equivalent if and only if they have the same signature and $r(\pi_1) = r(\pi_2)$. We have $r = r(\pi)$ for some irreducible big cell representation $\pi$ of $\mcO_q(H(N))$ of signature $\eta$ if and only if $r$ is $\epsilon$-adapted for the unique $\epsilon \in \{\pm 1\}^M$ with $\eta_i = \epsilon_{[i]}$. 
\end{Theorem} 
\begin{proof}
The first part of the theorem follows from Proposition \ref{PropCorrHT} and Proposition \ref{PropUniqueHW}. The latter part then follows from Theorem \ref{TheoClassIrrepT}. 
\end{proof}

\begin{Theorem}\label{TheoUnique}
\begin{enumerate}
\item Let $s \in \R^N$. Then the $*$-character $\chi_s$ on $\msZ(\mcO_q(H(N)))$ arises as the central $*$-character of an irreducible big cell representation if and only if the roots of 
\[
P_s(x) = \sum_{k=0}^N (-1)^k s_k x^{N-k}
\] 
are of the form 
\begin{equation}\label{EqFormRoots}
\{ q^{2\alpha+2m_1},\ldots,q^{2\alpha + 2m_{N_+}},-q^{2\beta + 2n_1},\ldots,-q^{2\beta + 2n_{N_-}},\underbrace{0,\ldots, 0}_{N_0}\}
\end{equation}
with $\alpha,\beta\in \R$ and $m_i, n_i \in \Z$ with $m_i\neq m_j$ and $n_i \neq n_j$ for all $i\neq j$. Moreover, if $\chi_s$ arises from the irreducible big cell representation $\pi$ and $P_s$ is of the above form, then necessarily $\pi$ has rank $N-N_0$ and there are exactly $N_+$ values $\eta_k=+$, with $\eta \in \{\pm 1\}^M$ the signature of $\pi$.
\item There is at most one irreducible big cell representation of a given rank and signature with fixed central $*$-character. 
\end{enumerate}
\end{Theorem}
\begin{proof}
Let $\pi$ be an irreducible big cell representation of rank $M$ with signature $\eta \in \{\pm 1\}^M$ and highest weight $r$ for the associated $\mcO_q^{\epsilon}(T(N))$-module, where $\epsilon \in \{\pm 1\}^M$ with $\eta_k = \epsilon_{[k]}$. By applying $i_T^{\epsilon}(\sigma_k)$ on a highest weight vector, we see from Proposition \ref{PropHC} that we will have 
\[
\sigma_k^{\pi} = e_k(\eta_1q^{2r_1},\eta_2q^{2r_2+2},\ldots,\eta_N q^{2r_N+2N-2}),
\]
where we put $\eta_m= 0$ for $m>M$ (and take the values $r_m$ for $m>M$ arbitrary). It follows that the roots of $P_s(x)$ for $s_k = \sigma_k^{\pi}$ are given by the $\eta_k q^{2(r_k+k)-2}$. From Theorem \ref{TheoBigCellRep} and the discussion following Definition \ref{DefEpsAda}, we see that then necessarily these roots must be of the form stated in the first item of the theorem, with $N_0 = N-M$ and $N_+$ the number of positive values of $\eta$. 

On the other hand, given an arbitrary $\epsilon \in \{\pm 1\}^M$ for $M = N- N_0$, any set of roots of the form \eqref{EqFormRoots} for which $N_+$ equals the number of positive $\epsilon_{[k]}$ can be ordered in a unique way as to have them satisfy the $\epsilon$-adaptedness condition. Then again by Theorem \ref{TheoBigCellRep}, we obtain the `if'-part of the first item as well as the second item of the theorem.
\end{proof}

As an immediate corollary, we have the following. 

\begin{Cor}\label{CorFinBigCellOs}
For a given $s\in \R^N$, there are only finitely many irreducible big cell representations of $\mcO_q(O_s)$.
\end{Cor}
\begin{proof}
This follows from the second part of Theorem \ref{TheoUnique}. We moreover get the explicit bound $\binom{N_++N_-}{N_+}$, with $N_++N_-$ the number of non-zero entries of $s$ and $N_+$ the number of strictly positive entries of $s$. 
\end{proof}

\subsection{Classification of highest weights for $\mcO_q^{\epsilon}(T(N))$}\label{SecClassHW}

We now prove Theorem \ref{TheoClassIrrepT}. First, we show that we can restrict to the case $M= N$. 

\begin{Lem}
Assume Theorem \ref{TheoClassIrrepT} holds when $M=N$. Then it also holds for $M\leq N$ arbitrary.
\end{Lem}
\begin{proof}
Let $M\leq N$, and let $\pi$ be an admissible unitary highest weight module of $\mcO_q^{\epsilon}(T(N))$ for $\epsilon \in \{-1,0,1\}^M$. By restricting $\pi$ to $\mcO_q^{\epsilon}(T(M))$, we obtain that $r$ needs to be $\epsilon$-adapted by the hypothesis. To see that this condition is also sufficient, let $r$ be $\epsilon$-adapted and put $\tilde{r} = (r,1,\ldots,1)$ and $\tilde{\epsilon} = (\epsilon,0,\ldots,0)$. Then $\tilde{r}$ is $\tilde{\epsilon}$-adapted and hence $\tilde{\epsilon}$-admissible, again by the hypothesis. There hence exists an admissible unitary highest module $V_{\tilde{r}}$ of $\mcO_q^{\tilde{\epsilon}}(T(N))$ with highest weight $\tilde{r}$. If $\xi$ is its highest weight vector, we then immediately obtain that $\mcO_q^{\epsilon}(T(N))\xi$ is a highest weight $*$-representation of $\mcO_q^{\epsilon}(T(N))$ at highest weight $r$, so $r$ is $\epsilon$-admissible. 
\end{proof}

The classification in Theorem \ref{TheoClassIrrepT} for the case $N=M$ can be done straightforwardly using \emph{Gelfand-Tsetlin bases}. We recall the specifics \cite{STU90} (see also \cite[Section 7.3.3]{KS97}), adapted to the $\epsilon$-deformed setting.  
The following proposition holds for arbitrary $\epsilon \in \R^N$. 
\begin{Prop}\label{PropScalProd}
Let $r\in \R^N$. Let $M_{r}$ be the Verma module for $\mcO_{q}^{\epsilon}(T(N))$ at highest weight $r$, and let $V_r$ be the corresponding simple quotient. Write 
\[
w = (w_1,\ldots,w_N) = (q^{r_1},\ldots,q^{r_N}).
\]
Then $V_r$ has a spanning set $\xi(P)$ (as a vector space) labeled by symbols $P$ where $P=(P_1,\ldots, P_{N-1})$ for 
\[
P_k = (P_{1k},\ldots,P_{kk}) \in \Z_{\geq 0}^{k},
\] 
where $\xi(P)$ is a weight vector with
\[
T_k \xi(P) = w_k q^{\sum_{i=k}^{N-1} P_{ki}- \sum_{i=1}^{k-1} P_{i,k-1}},
\] 
and where 
\begin{equation}\label{EqInnProdxi}
\langle \xi(P),\xi(P')\rangle = \delta_{P,P'} c_P
\end{equation}
for 
\begin{equation}\label{EqExprC}
c_P = \tau_1(P)\ldots\tau_{N-1}(P)c_P' 
\end{equation}
with
\[
0< c_P' =  \prod_{k=1}^{N-1} \prod_{1\leq i \leq j \leq k} (q^{-1}-q)^{-2P_{ik}} q^{-P_{ik}((r_j+j-r_i-i)    + \sum_{l=k}^{N-1} (P_{jl}-P_{il}) + (r_{j+1}+j+1-r_i-i)+ \sum_{l=k+1}^{N-1} (P_{j+1,l}-P_{il}) )} 
\]
and with 
\begin{multline*}
\tau_{k}(P) = \prod_{1\leq i \leq j \leq k} \left(\epsilon_{(i,j]}\frac{w_j^2}{w_i^2} q^{2(j-i+1)+2 \sum_{l=k}^{N-1}(P_{jl} - P_{il})};q^2\right)_{P_{ik}} \\ \qquad \times \left(\epsilon_{(i,j+1]}\frac{w_{j+1}^2}{w_{i}^2} q^{2(j-i+1) -2P_{ik}+2 \sum_{l=k+1}^{N-1}(P_{j+1,l} -P_{il})};q^2\right)_{P_{ik}} 
\end{multline*}
where we use the notation 
\[
(a;q^2)_m = (1-a)(1-q^2a)\ldots (1-q^{2m-2}a).
\]
\end{Prop} 
\begin{proof}
Put $U_q^{\epsilon}(\mfu(N))$ the $*$-algebra as defined in Definition \ref{DefUqglN}, except that the relation \eqref{EqFundCommEF} has been changed to 
\begin{equation}\label{EqFundCommEFEpsilon}
E_iF_j - F_jE_i = \delta_{ij} \frac{\epsilon_i\widehat{K}_i - \widehat{K}_i^{-1}}{q-q^{-1}}.
\end{equation}
We keep the same $*$-structure \eqref{EqDefStar}. 

The $*$-algebra isomorphism \eqref{EqUnivRelReal} deforms to an isomorphism of $*$-algebras $\iota_{\epsilon}:O_q^{\epsilon}(T(N)) \cong U_q^{\epsilon}(\mfu(N))$, in such a way that we have 
\[
\iota_{\epsilon}(T_i) = K_i^{-1},\qquad \iota_{\epsilon}(T_{i,i+1}) = (q^{-1}-q)F_iK_{i+1}^{-1}. 
\]
We will in the following view this isomorphism as an identification. 

It will be convenient to add self-adjoint square roots $K_i^{1/2}$ of the $K_i$ to $U_q^{\epsilon}(\mfu(N))$. We denote the resulting $*$-algebra as $\breve{U}_q(\mfu(N))$. Clearly any admissible $\mcO_q^{\epsilon}(T(N))$-module extends uniquely to a $\breve{U}_q(\mfu(N))$-module for which the $K_i^{1/2}$ have positive eigenvalues. 

Consider then the following elements in $\breve{U}_q(\mfu(N))$: 
\[
d_{kk} = 1,\qquad  c_{kk} = 1,
\]
\begin{equation}\label{EqDefAltGenUq}
d_{k,k-1} = f_k = q^{-1/2}F_k \hat{K}_k^{1/2},\qquad c_{k-1,k} = e_k =f_k^*  = q^{1/2} \hat{K}_i^{-1/2}E_i,
\end{equation}
and inductively 
\[
d_{ki} = \langle K_{i+1}K_k^{-1}q^{k-i}\rangle_{i+1,k} f_kd_{k-1,i} - \langle K_{i+1} K_k^{-1}q^{k-i-1}\rangle_{i+1,k} d_{k-1,i}f_k,
\]
\[
c_{ik} = d_{ki}^* = c_{i,k-1}e_k \langle K_{i+1}K_k^{-1}q^{k-i}\rangle_{i+1,k} - e_k c_{i,k-1}  \langle K_{i+1} K_k^{-1}q^{k-i-1}\rangle_{i+1,k},
\]
where $0 \leq i <k \leq N-1$ and 
\[
\langle a\rangle_{i,k} = \frac{\epsilon_{(i,k]}a -a^{-1}}{q-q^{-1}}. 
\]
For $P = (P_1,\ldots, P_{N-1})$ with $P_k = (P_{1k},\ldots,P_{kk}) \in \Z_{\geq 0}^{k}$ a symbol as above, write
\[
d_{k}^{P_k}  = d_{k,0}^{P_{1k}}\ldots d_{k,k-1}^{P_{kk}},\quad  d^P = d_1^{P_1}\ldots d_{N-1}^{P_{N-1}}, \qquad c_k^{P_k} = (d_k^{P_k})^*,\quad c^P = (d^P)^*. 
\]
With $\xi_r$ the highest weight vector of $M_r$, put 
\[
\xi(P) = d^P \xi_r \in M_r. 
\]
Then we claim that \eqref{EqInnProdxi} holds (in $M_r$ and hence in $V_r$). To see this, one can either repeat the arguments of \cite{STU90}, or argue as follows. Assume we are in the case $\epsilon_i>0$ for all $i$, and write
\[
\widetilde{e}_i = \epsilon_{i+1}^{-1/4} e_i,\qquad \widetilde{f}_i = \epsilon_{i+1}^{-1/4} f_i,\qquad \widetilde{K}_i^{1/2} = \epsilon_{(i,N]}^{1/4}K_i^{1/2}.
\]
Then the $\widetilde{e}_i,\widetilde{f}_i,\widetilde{K}_i^{1/2}$ establish an isomorphic copy of $\breve{U}_q(\mfu(N))$. Up to reindexing, the elements 
\[
\widetilde{d}_{k,k-1} = \widetilde{f}_k,\quad \widetilde{d}_{ki} = \langle \widetilde{K}_{i+1}\widetilde{K}_k^{-1}q^{k-i}\rangle \widetilde{f}_k\widetilde{d}_{k-1,i} - \langle \widetilde{K}_{i+1} \widetilde{K}_k^{-1}q^{k-i+1}\rangle \widetilde{d}_{k-1,i}\widetilde{f}_k
\]
and their adjoints then correspond to the elements as defined in \cite{STU90}. Upon noticing that 
\[
\widetilde{K}_i \xi_r = q^{-r_i'} \xi_r,\qquad q^{-r_i'} = \epsilon_{(i,N]}^{1/2} q^{-r_i}, 
\]
and
\[
d^P  = \left(\prod_{1\leq i \leq k< N} \left(\epsilon_{(i,k+1]}^{1/4} \prod_{i< j \leq k} \epsilon_{(i,j]}^{1/2}\right)^{P_{i,k}}\right)\widetilde{d}^P,
\]
one can easily deduce \eqref{EqInnProdxi} directly from \cite{STU90}. This identity however continues to hold true for general $\epsilon$ by continuity of the multiplication structure in the parameter $\epsilon$.

To finish the proof, we need to show that the $\xi(P)$ span $V_r$. To see this, let $\widetilde{M}_r$ be a vector with basis elements $\xi'(P)$ labeled by the $P$ as above. Define on $\widetilde{M}_r$ operators $K_i,e_i,f_i$ by the following formulas: 
\[
K_i \xi'(P) = q^{-r_i} q^{\sum_{j=1}^{i-1} P_{j,i-1}-\sum_{j=i}^{N-1} P_{ij}}\xi'(P),
\]
\[
e_i \xi'(P) = \sum_{j=1}^{i} a_{j,i}(P) \widetilde{\xi}(P+\delta_{j}(i-1) - \delta_{j}(i)),
\]
\[
f_i \xi'(P) = \sum_{j=1}^{i} b_{j,i}(P) \widetilde{\xi}(P-\delta_{j}(i-1) + \delta_{j}(i)),
\]
where  $\widetilde{\xi}(P\pm \delta_{j}(i-1) \mp \delta_{j}(i))$ means replacing $P_{j,i-1}$ by $P_{j,i-1}\pm 1$ and $P_{j,i}$ by $P_{j,i}\mp 1$, and where, with 
\[
[x]_\epsilon = \frac{\epsilon q^x - q^{-x}}{q-q^{-1}},\qquad [x]^{\epsilon} = \frac{q^x - \epsilon q^{-x}}{q-q^{-1}},
\]
\begin{multline*}
a_{j,i}(P) = - \frac{\prod_{1\leq k\leq j} [-r_{k}+r_{j} -\sum_{l=i+1}^{N-1} P_{k,l} + \sum_{l=i}^{N-1} P_{j,l}+j-k]_{\epsilon_{(k,j]}}}
{\prod_{1\leq k< j} [-r_{k}+r_{j} -\sum_{l=i+1}^{N-1} P_{k,l} + \sum_{l=i}^{N-1} P_{j,l}+j-k]_{\epsilon_{(k,j]}}} \\
 \times \frac{\prod_{j<k\leq i+1} [-r_{k}+r_{j} -\sum_{l=i+1}^{N-1} P_{k,l} + \sum_{l=i}^{N-1} P_{j,l}+j-k]^{\epsilon_{(j,k]}}}{\prod_{j<k\leq i} [-r_{k}+r_{j} -\sum_{l=i+1}^{N-1} P_{k,l} + \sum_{l=i}^{N-1} P_{j,l}+j-k]^{\epsilon_{(j,k]}}},
\end{multline*}
\begin{multline*}
b_{j,i}(P) = \frac{\prod_{1\leq k\leq j} [-r_{k}+r_{j} -\sum_{l=i+1}^{N-1} P_{k,l} + \sum_{l=i}^{N-1} P_{j,l}+j-k]_{\epsilon_{(k,j]}}}
{\prod_{1\leq k< j} [-r_{k}+r_{j} -\sum_{l=i+1}^{N-1} P_{k,l} + \sum_{l=i}^{N-1} P_{j,l}+j-k]_{\epsilon_{(k,j]}}} \\
 \times \frac{\prod_{j<k\leq i-1} [-r_{k}+r_{j} -\sum_{l=i-1}^{N-1} P_{k,l} + \sum_{l=i}^{N-1} P_{j,l}+j-k]^{\epsilon_{(j,k]}}}{\prod_{j<k\leq i} [-r_{k}+r_{j} -\sum_{l=i+1}^{N-1} P_{k,l} + \sum_{l=i}^{N-1} P_{j,l}+j-k]^{\epsilon_{(j,k]}}}
\end{multline*}
It can then be shown that these elements satisfy the defining relations for $\breve{U}_q^{\epsilon}(\mfu(N))$ (using the alternative generators in \eqref{EqDefAltGenUq}). Namely, we use \cite[Theorem 2.11]{STU90} to see that the defining relations hold when evaluated on a given $\xi'(P)$ when $\epsilon = +$ and $r$ is any large enough integral weight, after which one can conclude by a rescaling argument together with a Zariski density argument. 

Similarly, one concludes that 
\[
\xi'(P) = d^P \xi_r',
\]
where $\xi_r' = \xi'(0)$ corresponds to the diagram $0_{ik}=0$ for all $i,k$. It follows that $\widetilde{M}_r$ is a highest weight representation. By the universal property of $M_r$, we obtain a (necessarily surjective) $\breve{U}_q^{\epsilon}(\mfu(N))$-linear map $M_r\rightarrow\widetilde{M}_r$ sending $\xi(P)$ to $\xi'(P)$. But since the quotient map $M_r \rightarrow V_r$ must factor through $\widetilde{M}_r$, it follows that the $\xi(P)$ span $V_r$. 
\end{proof}

We can now prove Theorem \ref{TheoClassIrrepT}. We restrict now to the case of $\epsilon \in \{-1,0,1\}^{N}$. 

Assume first that $r$ is $\epsilon$-admissible. Assume that $s<t$ with $\epsilon_{(s,t]} = 1$. We can find a subdivision $s = t_0 <t_1< \ldots < t_n = t$ such that the following holds: each $\epsilon_{(t_i,t_{i+1}]} =1$ and $\epsilon_{(t_i,p]} <0$ for $p<t_{i+1}$. If the condition \eqref{EqIneqPos} holds for all $t_i < t_{i+1}$, it clearly also holds for $s<t$. We may hence assume already that $\epsilon_{(s,p]} <0$ for $p<t$. 

Consider the symbol $P$ with $P_{ik} = 0$ unless $i=s$ and $k = t-1$, in which case we write $P_{s,t-1} = m$. Then one easily sees that the sign of \eqref{EqExprC} equals the sign of 
\[
\left(\frac{w_{t}^2}{w_s^2} q^{2(t-s)-2m};q^2\right)_m = (q^{2((r_t+t)-(r_s +s)-m)};q^2)_m.
\]
The latter will be $\geq 0$ for all $m\in \Z_{\geq 0}$ if and only if $(r_t+t) - (r_s+s) -1 \in \Z_{\geq 0}$. Hence $r$ is $\epsilon$-adapted.

Conversely, assume that $r$ is $\epsilon$-adapted.  By Proposition \ref{PropUniqueHW} we only need to show that the scalar product in Proposition \ref{PropScalProd} is positive semi-definite, which will hold if and only if we have $c_M \geq0$ for all $M$, or equivalently, $\prod_{k=1}^{N-1}\tau_k(P) \geq 0$. 
 
Now we can reorder
\[
\prod_{k=1}^{N-1}\tau_k(P) = \prod_{i\leq j} \kappa_{ij}(P)
\]
with $\kappa_{ii}(P) = \prod_{k\geq i} (q^2;q^2)_{P_{ik}}>0$ and, for $i<j$, 
\begin{multline*}
\kappa_{ij}(P) = \prod_{k\geq j} \left(\epsilon_{(i,j]}\frac{w_j^2}{w_i^2} q^{2(j-i+1)+2 \sum_{l=k}^{N-1}(P_{jl} - P_{il})};q^2\right)_{P_{ik}} \\ \qquad \times \prod_{k\geq j+1}  \left(\epsilon_{(i,j]}\frac{w_{j}^2}{w_{i}^2} q^{2(j-i) -2P_{ik}+2 \sum_{l=k+1}^{N-1}(P_{jl} -P_{il})};q^2\right)_{P_{ik}}.
\end{multline*}
Clearly $\kappa_{ij}(P) \geq 0$ if $\epsilon_{(i,j]} \leq 0$. On the other hand, if $i<j$ with $\epsilon_{(i,j]} = 1$, we may by $\epsilon$-adaptedness write $r_j - r_i = m_{ij}+i-j+1$ for $m_{ij}\geq 0$. Then we can rewrite 
\[
\kappa_{ij}(P) = \prod_{k\geq j} \alpha_{ijk} \prod_{k\geq j+1} \beta_{ijk} 
\]
with 
\[
\alpha_{ijk} = \left(q^{2(m_{ij} + \sum_{l=k+2} P_{jl} - \sum_{l=k+1} P_{il}) + 2 -2P_{ik} + 2P_{j,k+1} +2P_{jk} + 2};q^2\right)_{P_{ik}},
\]
\[
\beta_{ijk} =  \left(q^{2(m_{ij} + \sum_{l=k+2} P_{jl} - \sum_{l=k+1} P_{il}) + 2 -2P_{ik} + 2P_{j,k+1}};q^2\right)_{P_{ik}}.
\]
Assume that $\kappa_{ij}(P) \neq 0$. Then as $m_{ij} \geq 0$, we see by considering $\beta_{ij,N-1}$ that necessarily $m_{ij} - P_{i,N-1}\geq 0$, and hence $\beta_{ij,N-1}>0$. By induction, we then see that 
\begin{equation}\label{EqPosm}
m_{ij} + \sum_{l=k+1} P_{jl} - \sum_{l=k} P_{il} \geq 0,\qquad \textrm{for all }k\geq j+1,
\end{equation}
and so all $\beta_{ijk} >0$ for $k\geq j+1$. It is then clear from the form of the $\alpha_{ijk}$ and \eqref{EqPosm} that also $\alpha_{ijk}>0$ for all $k\geq j$. 

It follows that $\kappa_{ij}(P) \geq 0$ for all $i\leq j$, and hence $r$ is $\epsilon$-admissble. 

This finishes the proof of Theorem \ref{TheoClassIrrepT}.

\section{A quantized version of Sylvester's law of inertia}\label{section: sylvester's law}

We now use the classification of irreducible  big cell representations for $\mcO_q(H(N))$ to prove our main results. 

\subsection{Weak containment} 

Our aim is to show that any representation of $\mcO_q(H(N))$ is weakly contained in a big cell representation. We will need some preparatory lemmas. Recall that 
\[
\Ad_{\lambda}(\pi) = (\pi \otimes \lambda)\circ \Ad_q^U
\]
for $\pi$, resp.~ $\lambda$ a representation of $\mcO_q(H(N))$, resp.~ $\mcO_q(U(N))$. For flexibility, we also use the notation
\[
\pi *\lambda = \Ad_{\lambda}(\pi) 
\] 

Let us restate Lemma \ref{LemWeakContGen} for our particular case of interest. 

\begin{Lem}\label{LemWeakCont}
Let $\pi$ be a representation of $\mcO_q(H(N))$. Then $\pi$ is weakly contained in $\Ad_{\lambda_{\standard}}(\pi)$. 
\end{Lem}

\begin{Lem}\label{LemDirInt2}  Any big cell representation $\pi$ of $\mcO_q(H(N))$ is a direct integral of irreducible big cell representations. Moreover, if $\pi$ factors over some $\mcO_q(O_s)$, then $\pi$ is a direct sum of irreducible big cell representations.
\end{Lem}
\begin{proof}
It is sufficient to consider big cell representations of rank $M$. However, such a representation will clearly disintegrate into a direct integral of factorial big cell representations $\pi$ of rank $M$, i.e.~ with $\pi(\mcO_q(H(N))''$ a factor. It is thus sufficient to show that any factorial big cell representation $\pi$ is a direct sum of irreducibles. This will then automatically prove also the moreover part of the lemma, by Corollary \ref{CorFinBigCellOs}.

Note now that a factorial big cell representation $\pi$ automatically has a fixed signature $\eta$, and that moreover $\Hsp_{\pi}$ has an orthonormal basis of weight vectors for the $Z_{\brk}$ with $k\leq M$. Let $V_{\pi}$ be the algebraic sum of these eigenspaces. Then with $\epsilon \in \{\pm 1\}^M$ such that $\eta = \eta_{\epsilon}$, we have that $\pi$ extends on $V_{\pi}$ to a $^*$-representation of $\mcO_{q}^{\epsilon}(T(N))$ through $i_{\epsilon}^T$. Let $V_h \subseteq V_{\pi}$ be the subspace of vectors $\xi$ with $\mcO_q^{\leq M}(\msN(N))^* \xi = 0$. Then $V_h$ will be invariant under the $Z_{\brk}$, and hence the linear span of the highest weight vectors for $\pi$. As $i_{\epsilon}^T(Z_{\brk})T_{ij}^* = q^{-2 \delta_{k \in [i,j)}}T_{ij}^*i_{\epsilon}^T(Z_{\brk})$, we have by boundedness of the operators $Z_{\brk}$ that the vector space $V_h$ will be non-empty. 

Now any element in $\pi(\mcO_q(H(N)))'$ will preserve the homogenous components of $V_h$. On the other hand, if $\xi,\eta \in V_h$ are orthogonal vectors, it follows from the triangular decomposition \eqref{EqTriangDecomp} that $\mcO_q(H(N))\xi$ and $\mcO_q(H(N))\eta$ will be orthogonal.  By factoriality of $\pi$, we hence deduce that $V_h$ must consist of a single homogeneous component, which moreover must then generate the representation $\pi$. If then $\{\xi_k\}$ is an orthonormal basis of $V_h$, it follows that $\pi = \oplus_k \pi_k$ with $\pi_k$ the restriction of $\pi$ to the closure of $\mcO_q(H(N)) \xi_k$. Since clearly each $\pi_k$ is an irreducible big cell representation, we are done. 
\end{proof}

In the next lemma, we will use that for $A$ a unital $*$-algebra with coaction $\alpha: A \rightarrow A\otimes \mcO(\mathbb{G})$ by a compact quantum group $\mathbb{G}$, and $\pi \cong \int_X^{\oplus} \pi_x \rd \mu(x)$ a direct integral decomposition of a representation $\pi$ of $A$, then
\[
\pi * \lambda \cong \int_X^{\oplus} (\pi_x * \lambda)\rd \mu(x)
\]
for any representation $\lambda$ of $\mcO(\mathbb{G})$. Indeed, the isomorphism is established through the unitary operator 
\[
\int_X^{\oplus} \left(\Hsp_{\pi_x}\otimes \Hsp_{\lambda}\right)\rd \mu(x) \cong \left(\int_X^{\oplus} \Hsp_{\pi_x} \rd \mu(x)\right)\otimes \Hsp_{\lambda}.
\]
We will use this in combination with the fact that any representation of a (separable) type $I$ C$^*$-algebra disintegrates into a direct integral of irreducible representations. We reprise the notation from Section \ref{SecExH2}.

\begin{Lem}\label{LemDichH2}
If a representation $\pi$ of $\mcO_q(H(2))$ is $\Ad^U_q$-compatible and normal, then $\Ker(\pi(z))$ is $\mcO_q(H(2))$-invariant and $\pi(Z)_{\mid \C^2 \otimes \Ker(\pi(z))} = 0$.
\end{Lem}
\begin{proof} 
The $\mcO_q(H(2))$-invariance of $\Ker(\pi(z))$ is immediate as $z$ $q$-commutes with the other generators. Further, since $\pi$ is assumed normal, we have by Proposition \ref{PropNorComp} that $\Ad_{\lambda_{\standard}}(\pi)$ is unitarily equivalent to a direct sum of copies of $\pi$. By Corollary \ref{CorTypeIH2} and the remark preceding the lemma, it is then sufficient to prove that the kernel of $\pi(z)$ is zero for $\pi  = \Ad{\lambda_{\standard}}(\pi')$ with $\pi'$ an irreducible representation of $\mcO_q(H(2))$ distinct from the zero $*$-character  $Z \mapsto 0$. 

If $\pi'$ is as in case (1) of  Proposition \ref{PropRank1Tensor}, then there exists $m>0$ with $m\leq \pi'(Z)$, hence also $m\leq (\Ad_{\lambda_{\standard}}(\pi'))(z) = (\langle e_1,-\rangle e_1\rangle \pi' \otimes \lambda_{\standard})(U_{13}^*Z_{12}U_{13})$.

If we are in the case (2) or case (3), we may by Proposition \ref{PropNorCompErg} assume that $\pi'$ is any non-zero $*$-character $\chi$ of $\mcO_q(H(2))$, say $\chi(Z) = \begin{pmatrix} 0 & x \\ x & y\end{pmatrix}$ for $x,y\in \R$. But it can be checked by a direct spectral computation that 
\[
(\Ad_{\lambda_{\standard}}(\chi))(z) = xa^*c + yc^*c + xc^*a  \in B(L^2(SU_q(2)))
\] 
does not have $0$ in its pointspectrum if $(x,y)\neq (0,0)$, cf.~ \cite{MNW91}.
\end{proof}

\begin{Lem}\label{LemDecompSign}
Let $\pi$ be a big cell representation of $\mcO_q(H(2))$ of signature $(\eta_1,\eta_2)$ with $\eta_1\eta_2 = -1$. Then $\pi*\mbs = \pi_1 \oplus \pi_2$ with $\pi_1$ of signature $(\eta_1,\eta_2)$ and $\pi_2$ of signature $(\eta_2,\eta_1)$. 
\end{Lem}  
\begin{proof}
By Corollary \ref{CorTypeIH2} and the discussion above Lemma \ref{LemDichH2}, it is sufficient to prove the lemma for $\pi$ irreducible. Now the conditions imply that we are in the case $(3,b)$ of Proposition \ref{PropRank1Tensor}, so $\Hsp_{\pi} \cong \mcS_{-c^2,a}^{\pm}$ for some choice of sign and $a,c>0$.  But then 
\[
\mcS^{\pm}_{-c^2,a}*\mbs \cong \Hsp \otimes (\mcS^+_{-c^2,a} \oplus \mcS^-_{-c^2,a})
\]
with $\Hsp$ some countable multiplicity Hilbert space, see e.g.~ \cite[Proposition 1.14]{DeC20}.
\end{proof}

\begin{Prop}\label{PropBigCell}
Let $\pi$ be a normal $\Ad^U_q$-compatible representation of $\mcO_q(H(N))$. Then $\pi$ is a big cell representation.
\end{Prop}
\begin{proof}
Let us in the proof write $\Ad^U_q= \Ad^{U(N)}_q$. 

We give a proof by induction. 

The lemma holds trivially for $N=1$. 

Assume now that the lemma holds for $N-1$, and let $\pi$ be a normal $\Ad^{U(N)}_q$-compatible representation of $\mcO_q(H(N))$. Let $\Hsp_1 = \Ker(\pi(Z_{11}))$ and $\Hsp_2 = \Hsp^{\perp}_1$. Then $\pi$ restricts to representations $\pi_i$ of $\mcO_q(H(N))$ on $\Hsp_i$. Let us show that these representations are normal $\Ad^{U(N)}_q$-compatible. Indeed, consider $\mcO_q(H(2))$ embedded in the upper left hand corner of $\mcO_q(H(N))$. As the restriction of $\pi$ to $\mcO_q(H(2))$ is normal $\Ad^{U(2)}_q$-compatible by Proposition \ref{PropNorComp}, we have by the previous lemma that $\pi_1(Z_{ij}) = 0$ for $1\leq i,j\leq2$, and hence $(\pi_1 * \mbs_1)(Z_{11}) =0$. Since $(\pi*\mbs_i)(Z_{11}) = \pi(Z_{11})\otimes 1$ for $i\geq 2$, we find by using 
\[
w_0 = (s_1s_2\ldots s_N)(s_2s_3\ldots s_N)\ldots (s_{N-1}s_N)s_N
\]
that $(\pi_1*\lambda_{\standard})(Z_{11}) = 0$. Writing $p_1$ for the projection onto $\Hsp_1$, we have by the $q$-commutation relations for $Z_{11}$ that $p_1 \in \msZ(\pi(\mcO_q(H(N)))'')$ and, by the arguments above, $p_1\otimes 1 \leq \Ad_q^{U(N)}(p_1) $. On the other hand, since $E(z) = \left(\id\otimes \int_{U_q(N)}\right)\Ad_q^{U(N)}(z)$ is a faithful conditional expectation of $\pi(\mcO_q(H(N)))''$ onto $\pi(\msZ(\mcO_q(H(N))))''$, it follows easily that $p_1 = E(p_1)$ and hence in fact $\Ad_q^{U(N)}(p_1) = p_1\otimes 1$. Hence $\pi_1$ is an $\Ad^{U(N)}_q$-compatible normal representation of $\mcO_q(H(N))$, and the same must then hold for $\pi_2$. 

It hence suffices to prove our proposition in two specific cases, namely $\pi(Z_{11}) =0$ and $\Ker(\pi(Z_{11}))= \{0\}$. 

In the first case, it follows by $\Ad^{U(N)}_q$-compatibility that in fact $\pi(U_q(\mfu(N))\rhd Z_{11})=0$, hence, by \eqref{EqGlobForm}, $\pi(Z_{ij})=0$ for all $i,j$, and $\pi$ is up to amplification with some multiplicity Hilbert space the unique big cell representation of rank $0$. 

Consider now the second case. Denoting by $\mcO_q(U(N)) \rightarrow \mcO_q(U(N-1))$ the restriction map to the lower right corner, we obtain by restriction that $\pi$ is an $\Ad^{U(N-1)}_q$-compatible normal representation, using again Proposition \ref{PropNorComp}. Through the map $ \rho_{2,N}$ defined in \eqref{EqMapk}, we obtain an $\Ad^{U(N-1)}_q$-compatible normal representation $\pi$ of $\mcO_q(H(N-1))$. By induction, $\msH_{\pi}$ is a big cell representation of $\mcO_q(H(N-1))$. But since $\Ker(\pi(Z_{11}))= 0$ by assumption, this makes $\pi$ into a big cell representation of $\mcO_q(H(N))$.  
\end{proof}

\begin{Theorem}\label{TheoWeakCont}
Any irreducible representation $\pi$ of $\mcO_q(H(N))$ is weakly contained in an irreducible big cell representation. 
\end{Theorem} 
\begin{proof}
By Lemma \ref{LemWeakCont}, $\pi$ is weakly contained in $\pi*\lambda_{\standard}$, which by Proposition \ref{PropBigCell} is a big cell representation. Now as $\pi*\lambda_{\standard}$ acts as a $*$-character on the center of $\mcO_q(H(N))$, it follows by Lemma \ref{LemDirInt2} and Corollary \ref{CorFinBigCellOs} that $\pi*\lambda_{\standard}$ must be an amplification of a finite number of irreducible big cell representations of $\mcO_q(H(N))$. Hence $\pi$ must be weakly contained in an irreducible big cell representation.
\end{proof}

\subsection{Quantized Sylvester's law of inertia}

We prove our main theorems. We start with Theorem \ref{TheoEigZ}.

\begin{Theorem}\label{TheoEigZRef}
Let $s\in \R^N$. Then $s$ is centrally admissible if and only if all non-zero roots of the polynomial 
\begin{equation}\label{EqCharPolq}
P_s(z) = \sum_{k=0}^N (-1)^k s_k z ^{N-k}
\end{equation}
are real with multiplicity $1$, and moreover the quotients of non-zero roots of the same sign lie in $q^{2\Z}$.  
\end{Theorem} 
In other words, $s$ is admissible if and only if we have a factorisation 
\[
P_s(z) = z^{N_0}(z- q^{2\alpha + 2m_1})\ldots (z- q^{2\alpha + 2m_{N_+}}) (z+ q^{2\beta + 2n_1})\ldots (z+ q^{2\beta + 2n_{N_-}})
\]
where $\alpha,\beta \in \R$ and $m_i,n_i \in \Z_{\geq 0}$ with the $m_i-m_j \neq 0$ and $n_i-n_j\neq 0$ for $i\neq j$, which is the formulation in  Theorem \ref{TheoEigZ}.

\begin{proof}
Assume that $s$ is centrally admissible. By Theorem \ref{TheoWeakCont} it follows that there exists an irreducible big cell representation $\pi$ of $\mcO_q(H(N))$ such that $\sigma_k^{\pi} = s_k$. The theorem now follows from Theorem \ref{TheoUnique}.
\end{proof}

To prove Theorem \ref{TheoQLI}, we start with the following lemmas.

\begin{Lem}\label{LemPermSign}
Let $\pi$ be an irreducible big cell representation of rank $M$ with signature $\eta$, and assume $k < N_++ N_-$ with $\eta_k \eta_{k+1} < 0$. Then there exists an irreducible big cell representation $\pi' \subseteq\pi * \mbs_k$ of rank $M$ such that $\pi'$ has signature 
\[
\eta' = (\eta_1,\ldots,\eta_{k-1},\eta_{k+1},\eta_k,\eta_{k+2},\ldots,\eta_M). 
\]
\end{Lem} 
\begin{proof}
Assume that $\pi$ factors over $\mcO_q(O_s)$. Clearly we will have 
\[
(\pi*\mbs_k)(Z_{[r]}) = \pi(Z_{[r]})\otimes 1
\]
for $r < k$ or $r>k+1$. On the other hand, consider the representation $\pi \circ \rho_{k,k+1}$ of $\mcO_q(H(2))$ with $\rho_{k,k+1}$ as in Proposition \ref{PropEquik}. Then clearly $\pi \circ \rho_{k,k+1}$ has fixed signature $(\eta_{[k]},\eta_{[k-1]}\eta_{k+1})$ with $\eta_{[k]}(\eta_{[k-1]}\eta_{k+1}) = \eta_k\eta_{k+1} = -1$ by the assumption on $k$. By equivariance of $\rho_{k,k+1}$, we have $(\pi * \mbs_k) \circ \rho_{k,k+1} = (\pi \circ \rho_{k,k+1})*\mbs$, hence by Lemma \ref{LemDecompSign} it follows that $\pi * \mbs_k$ contains a subrepresentation of signature $\eta'$. As $\pi * \mbs_k$ splits as a direct sum of irreducible big cell representations of $\mcO_q(O_s)$ by Lemma \ref{LemDirInt2}, it follows that there exists an  irreducible big cell representation $\pi' \subseteq\pi * \mbs_k$ of signature $\eta'$. 
\end{proof}

\begin{Lem}\label{LemWeakStandTens}
Let $s \in \R^N$ be centrally admissible, and let $\pi,\pi'$ be representations of $\mcO_q(O_s)$. Then $\pi' \preccurlyeq  \Ad_{\lambda_{\standard}}(\pi)$.
\end{Lem} 
\begin{proof}
This is a special case of Corollary \ref{CorWeakContErg}.
\end{proof}

For $\lambda \in \mbT_q(N)$ an admissible (hence finite dimensional) irreducible representation, let us denote by $\lambda^c$ the contragredient unitary representation, uniquely determined by the fact that the trivial representation is contained in $\lambda^c * \lambda$. By taking direct sums, one can define $\lambda^c$ for any admissible representation. Then we have the following standard Frobenius reciprocity. 

\begin{Lem}\label{LemFrobRec}
Let $\pi,\pi' \in \mbH_q(N)$ be irreducible, and let $\lambda \in \mbT_q(N)$ be finite dimensional. If $\pi \subseteq \pi'*\lambda$, then $\pi' \subseteq \pi*\lambda^c$. 
\end{Lem}

\begin{Lem}\label{LemMultT}
Let $\eta$ be a signature
\[
\eta = \diag(\underbrace{1,\ldots,1}_{N_+},\underbrace{-1,\ldots,-1}_{N_-}),
\]
and let $\pi,\pi'$ be two associated irreducible big cell representations with respective weights $r$ and $r'$.
\begin{enumerate}
\item If $N_+N_- = 0$, then there exists an admissible $\mcO_q(T(N))$-representation $\lambda$ with $\pi'\subseteq \pi * \lambda$. 
\item If $N_+N_-\neq 0$, then there exists an admissible $\mcO_q(T(N))$-representation $\lambda$ with $\pi'\subseteq \pi * \lambda$ if and only if $r_{N_++N_-} - r_{N_+} \in (r_{N_++N_-}' - r_{N_-}') + \Z$. 
 \end{enumerate} 
\end{Lem} 
\begin{proof}
We know from Theorem \ref{TheoBigCellRep} that 
\[
r = (\alpha + m_1,\ldots,\alpha + m_{N_+},\beta + n_1,\ldots,\beta + n_{N_-}),
\]
with $m_1\leq m_2\leq \ldots \leq m_{N_+}$ and $n_1\leq n_2 \leq \ldots \leq n_{N_-}$ in $\Z_{\geq 0}$, and $\alpha,\beta \in \R$. Similarly 
\[
r' = (\alpha' + m_1',\ldots,\alpha'+ m_{N_+}',\beta' + n_1',\ldots,\beta' + n_{N_-}'),
\]
Now any weight vector for $\mcO_q(T(N))$ has weight of the form $(\gamma + k_1,\ldots,\gamma + k_N)$ for $\gamma \in \R$ and $k_1 \leq \ldots \leq k_N$ in $\Z_{\geq 0}$. This already shows that the conditions in the lemma are necessary. To see that they are sufficient, let us first multiply both representations with a one-dimensional admissible representation $T_{ij} \mapsto c\delta_{ij}$ of $\mcO_q(T(N))$ to ensure $\alpha = \alpha' = 0$. Pick then $p \in \Z_{\geq 0}$ large enough so that 
\[
m_1 \leq m_2 \leq \ldots \leq m_{N_+} \leq n_1 + p \leq \ldots \leq n_{N_-} + p,
\]
\[
m_1' \leq m_2 '\leq \ldots \leq m_{N_+}' \leq n_1' + p \leq \ldots \leq n_{N_-}' + p.
\]
Consider now the finite dimensional admissible representation $\lambda_1 =   \pi_{(m_1,\ldots,m_{N_+},n_1+p,\ldots,n_{N_-}+p,K,\ldots,K)} \in \mbT_q(N)$, where $K \in \Z_{\geq n_{N-}+p}$ is arbitrary. Let  $\pi_0 =  \pi_{S,(0,\ldots,0,\beta -p,\ldots,\beta-p,0,\ldots 0)}$. Then by restricting to the tensor product of the highest weights, we see that
\[
\pi_{S,r}  \subseteq \pi_0 * \lambda_1.
\]
Similarly, there exists a finite dimensional admissible representation $\lambda_2 \in \mbT_q(N)$ with 
\[
\pi_{S,r'} \subseteq \pi_0*\lambda_2.
\]
Hence, by the Frobenius reciprocity of Lemma \ref{LemFrobRec}, 
\[
\pi_{S,r'} \subseteq \pi_0 * \lambda_2 \subseteq \pi_{S,r} * \lambda_1^{c} *\lambda_2. 
\]
\end{proof}

We now prove Theorem \ref{TheoQLI}.

\begin{Theorem}\label{TheoQSLI}
Let $\pi \in \mbH_q(N)$ be a representation of extended signature $([r],N_+,N_-,N_0)$. Then $\Ad_{\lambda}(\pi)$ has the same extended signature for any $\lambda \in \mbGL_q(N,\C)$. Moreover, if $\pi,\pi' \in \mbH_q(N)$ are irreducible representations of the same extended signature, then there exists such a $\lambda\in \mbGL_q(N,\C)$ with 
\[
\pi' \preccurlyeq  \Ad_{\lambda}(\pi).
\]
\end{Theorem}
\begin{proof}
Let us prove the first part of the theorem.

Clearly $\pi_1\preccurlyeq  \pi_2$ and $\pi_2$ of extended signature $([r],N_+,N_-,N_0)$ implies $\pi_1$ of the same extended signature. By the remark following Theorem \ref{TheoRepGLN}, it is hence enough to show that if $\pi$ is of extended signature $([r],N_+,N_-,N_0)$, then also  $\pi * \lambda_{\standard}*\lambda'$ has extended signature $([r],N_+,N_-,N_0)$, for $\lambda_{\standard}$ the standard representation of $\mcO_q(U(N))$ and $\lambda'$ an irreducible admissible representation of $\mcO_q(T(N))$. Now since the center of $\mcO_q(H(N))$ is $Ad_q^U$-coinvariant, clearly $\pi * \lambda_{\standard}$ has the same signature as $\pi$. By Proposition \ref{PropNorComp} and Proposition \ref{PropBigCell} we may hence assume that $\pi$ is an irreducible big cell representation, and we want to show that $\pi*\lambda'$ has the same extended signature. But since $\pi$ factors through $\mcO_q(H(N)) \rightarrow \mcO_q^{\leq M}(H(N)) \rightarrow \mcO_q^{\leq M}(T(N))$, and $\Ad_q^T$ factors through a coaction on $\mcO_q^{\leq M}(T(N))$, it follows that $\pi*\lambda'$ has rank $M$. Moreover, since $\Ad_q^T(Z_{[k]}) = Z_{[k]}\otimes T_1^2\ldots T_k^2$, we see that $\pi*\lambda'$ is a big cell representation. As the $(\pi*\lambda')(Z_{[k]})$ are diagonalizable, $\pi*\lambda'$ must be a direct sum of irreducible big cell representations. But since the extended signature of an irreducible big cell representation can be read of from any of its weight spaces, it follows straightforwardly that $\pi*\lambda'$ has extended signature $([r],N_+,N_-,N_0)$. This proves the first part of the theorem. 

To prove the second part of the theorem, assume that $\pi,\pi'$ are irreducible representations of $\mcO_q(H(N))$ of extended signature $([r],N_+,N_-,N_0)$. We may assume that $\pi,\pi'$ are irreducible. To show that there exists $\lambda \in \mbGL_q(N)$ so that $\pi' \preccurlyeq  \pi*\lambda$, we may by Lemma \ref{LemWeakCont} and Proposition \ref{PropBigCell} assume that $\pi$ is an irreducible big cell representation.  By applying Lemma \ref{LemPermSign} repeatedly and then applying Lemma \ref{LemWeakStandTens}, we may assume that $\pi'$ and $\pi$ have signature $\diag(\underbrace{1,\ldots,1}_{N_+},\underbrace{-1,\ldots,-1}_{N_-},\underbrace{0,\ldots,0}_{N_0})$. But for these the conclusion of the theorem follows immediately from Lemma \ref{LemMultT}.
\end{proof} 

\begin{Theorem}\label{TheoTypeI}
The $*$-algebra $\mcO_q(H(N))$ is C$^*$-faithful and type $I$. 
\end{Theorem}
\begin{proof}
The C$^*$-faithfulness of $\mcO_q(H(N))$ follows easily from known results, for example the injectivity of $\pi_T$ in Proposition \ref{PropInclChol} and the fact that the finite-dimensional $*$-representations of $\mcO_q(T(N))$ separate the elements of $\mcO_q(T(N))$. 

Let now $\varsigma_{\ext} = ([r],N_+,N_-,N_0)$ be an extended signature, and let $\chi_r$ be a $*$-character of extended signature $\varsigma_{\ext}$. For example, if $N_+N_-\neq 0$ and $L = \max\{N_+,N_-\}-|N_+-N_-|$, we can take the $*$-character
\[
\chi_r: \mcO_q(H(N)) \mapsto \C,\qquad Z \mapsto \pm \left(q^{-r}\sum_{i=N_0+L+1}^N e_{ii}+ \sum_{i=1}^L(e_{N_0+i,N-i} + e_{N-i,N_0+i})- q^{r} \sum_{i=1}^L e_{N-i+1,N-i+1}\right)
\] 
where $|N_+-N_-| = \pm(N_+-N_-)$.

Let $\pi$ be an arbitrary irreducible representation of $\mcO_q(H(N))$. Then $\pi$ has a well-defined extended signature $\varsigma_{\ext}$. From Theorem \ref{TheoQSLI} we know that $\pi$ is weakly contained in $\chi_r * \lambda$ for some representation $\lambda$ of $\mcO_q^{\R}(GL(N,\C))$. However, every irreducible representation of $\mcO_q^{\R}(GL(N,\C))$ factors through an irreducible representation of $\mcO_q(U(N)) \otimes \mcO_q(T(N))$. In particular, there exists an irreducible (and hence finite-dimensional) admissible representation $\lambda'$ of $\mcO_q(T(N))$ such that $C^*(\pi)$ is a quotient C$^*$-algebra of a C$^*$-subalgebra of $C(U_q(N))\otimes B(\Hsp_{\lambda'})$. As the latter C$^*$-algebra is type $I$, it follows that $C^*(\pi)$ is type $I$. 
\end{proof}

\end{document}